\documentclass{article}
\usepackage{amsmath}[1996/11/01]
\usepackage{amssymb,amsthm,amsxtra}

\def\[#1\]{\begin{equation}#1\end{equation}}
\makeatletter
\def\beq{%
   \relax\ifmmode
      \@badmath
   \else
      \ifvmode
         \nointerlineskip
         \makebox[.6\linewidth]%
      \fi
      $$
   \fi
}
\def\eeq{%
   \relax\ifmmode
      \ifinner
         \@badmath
      \else
         $$
      \fi
   \else
      \@badmath
   \fi
   \ignorespaces
}

\def\enddisplaymath{\eeq\global\@ignoretrue}
\makeatother

\newtheorem{thm}{Theorem}
\newtheorem{cor}[thm]{Corollary}
\newtheorem{lem}[thm]{Lemma}
\newtheorem{prop}[thm]{Proposition}
\newtheorem{conj}{Conjecture}
\newtheorem{conjL}{Conjecture}

\newtheorem{conjQ}{Conjecture}

\theoremstyle{remark}
\newtheorem*{rem}{Remark}
\newtheorem{rems}{Remark}[thm]

\theoremstyle{definition}

\numberwithin{equation}{section}
\numberwithin{thm}{section}
\numberwithin{eg}{section}

\newcommand{\C}{\mathbb C}

\newcommand{\Z}{\mathbb Z}
\newcommand{\N}{\mathbb N}

\DeclareMathOperator\pf{pf}
\DeclareMathOperator\GL{GL}
\DeclareMathOperator\Sp{Sp}
\DeclareMathOperator\SL{SL}
\DeclareMathOperator\cC{{\cal C}}
\DeclareMathOperator\cR{{\cal R}}
\DeclareMathOperator\cI{{\cal I}}
\DeclareMathOperator\tcR{\tilde{\cal R}}

\newcommand{\II}{\mathord{I\!I}}
\newcommand{\blambda}{{\boldsymbol\lambda}}
\newcommand{\bkappa}{{\boldsymbol\kappa}}
\newcommand{\bmu}{{\boldsymbol\mu}}
\newcommand{\bnu}{{\boldsymbol\nu}}

\newcommand{\la}{\langle}
\newcommand{\ra}{\rangle}
\newcommand{\obinomE}{\genfrac\la\ra{0pt}{}}

\def\Gampq{\Gamma_{\!p,q}}
\def\Gamppqq{\Gamma_{\!p^2,q^2}}
\def\Gampqq{\Gamma_{\!p,q^2}}
\def\Gamppq{\Gamma_{\!p^2,q}}
\def\Gamphq{\Gamma_{\!p^{1/2},q}}
\def\Gamphqh{\Gamma_{\!p^{1/2},q^{!/2}}}

\def\Gampqt{\Gamma^+_{\!p,q,t}}
\def\Gampqtt{\Gamma^+_{\!p,q,t^2}}

\begin{document}

\title{Elliptic Littlewood identities} \author{Eric
  M. Rains\\Department of Mathematics, California Institute of Technology}

\date{February 20, 2012}
\maketitle

\begin{abstract}
We prove analogues for elliptic interpolation functions of Macdonald's
version of the Littlewood identity for (skew) Macdonald polynomials, in the
process developing an interpretation of general elliptic ``hypergeometric''
sums as skew interpolation functions.  One such analogue has an
interpretation as a ``vanishing integral'', generalizing a result of
\cite{vanish}; the structure of this analogue gives sufficient insight to
enable us to conjecture elliptic versions of most of the other vanishing
integrals of \cite{vanish} as well.  We are thus led to formulate ten
conjectures, each of which can be viewed as a multivariate quadratic
transformation, and can be proved in a number of special cases.

\end{abstract}

\tableofcontents

\section{Introduction}

In recent work of the author \cite{bctheta} (see also
\cite{CoskunH/GustafsonRA:2006} for an independent treatment), a family of
``interpolation functions'' were introduced, generalizing Okounkov's
interpolation polynomials \cite{OkounkovA:1998}, which in turn generalize
shifted Macdonald polynomials \cite{SahiS:1996} and Macdonald polynomials
\cite{MacdonaldIG:1995} themselves.  Among the identities satisfied by the
interpolation functions is an analogue of the Cauchy identity, which for
Macdonald polynomials states
\[
\sum_\mu P_\mu(x_1,\dots,x_n;q,t)P_{\mu'}(y_1,\dots,y_m;t,q)
=
\prod_{\substack{1\le i\le n\\1\le j\le m}}
(1+x_iy_j).
\]
Macdonald also proved (generalizing a result of Kadell for Jack
polynomials) an analogue for Macdonald polynomials of the Littlewood
identity, see \cite[Ex.~VI.7.4]{MacdonaldIG:1995}:
\[
\sum_{\mu} c_\mu(q,t) P_{\mu^2}(x_1,\dots,x_n;q,t)
=
\prod_{1\le i<j\le n}
\frac{(t x_ix_j;q)}{(x_ix_j;q)},
\]
where $P_\lambda$ is a Macdonald polynomial, $\mu^2$ denotes the partition
with parts $(\mu^2)_i=\mu_{\lceil i/2\rceil}$,
\[
(x;q):=\prod_{k\ge 0} (1-q^k x),
\]
and the coefficients $c_\mu(q,t)$ are given by an explicit product:
\[
c_\mu(q,t)
=
\prod_{(i,j)\in \mu}
\frac{1-q^{\mu_i-j  }t^{2\mu'_j-2i+1}}
     {1-q^{\mu_i-j+1}t^{2\mu'_j-2i  }}.
\]
(The usual notation for $(x;q)$ would be $(x;q)_\infty$, but since we never
use finite $q$-Pochhammer symbols, we suppress $\infty$ throughout.)  This
is the $q,t$-analogue of Littlewood's identity for Schur functions:
\[
\sum_\mu s_{\mu^2}(x_1,\dots,x_n)
=
\prod_{1\le i<j\le n} (1-x_ix_j)^{-1},
\]
which describes the decomposition of $S^*(\wedge^2(\C^n))$ as a
representation of $\GL_n$, and thus by Frobenius reciprocity determines
which irreducible representations of $\GL_{2n}$ have invariants under
$\Sp_n$ (since the coordinate ring of the affine variety $\GL/\Sp$ is
obtained from $S^*(\wedge^2(\C^{2n}))$ by inverting the pfaffian).  The
purpose of the present note is to generalize such Littlewood-type
identities to the elliptic level.

The primary obstacle to such an extension is the fact that, unlike the
given form of the Cauchy identity, for which the terms vanish unless the
partition $\mu$ is contained in an $m\times n$ rectangle, the Littlewood
identity intrinsically involves a nonterminating sum.  Unfortunately, at
the elliptic level, infinite sums seem inevitably to encounter convergence
difficulties, making a direct extension problematical.  One must thus
either modify the sum in such a way as to force termination (say by
a suitable choice of the coefficients $c_\mu$), or replace the sum by an
integral.  We will, in fact, take both approaches.

Our first step is to observe that Macdonald's Littlewood identity has a
generalization (implicit in \cite{MacdonaldIG:1995}; the argument sketched
in Ex. I.5.27 and Ex. VI.7.6 op. cit. carries over mutatis mutandum) to
skew Macdonald polynomials:
\[
\sum_\mu c_\mu(q,t) P_{\mu^2/\lambda}(x_1,\dots,x_n;q,t)
=
\prod_{1\le i<j\le n}
\frac{(t x_ix_j;q)}{(x_ix_j;q)}
\sum_\mu c_\mu(q,t) Q_{\lambda/\mu^2}(x_1,\dots,x_n;q,t),
\]
where $c_\mu(q,t)$ is as above.  Of course, this in itself makes an
extension more difficult, given the absence (but see below) of a good
theory of skew versions of the interpolation functions.  On the other hand,
the proof of Macdonald's Littlewood identity uses only the case $n=1$ of
this skew Littlewood identity, together with a corresponding case of the
skew Cauchy identity.  This case is particularly amenable to
generalization, as both sums are finite (indeed, each has only one nonzero
term), and the case $n=1$ of the skew Macdonald polynomials {\em does} have
a very natural elliptic analogue.  Indeed, the principal specialization
(i.e., with variables specialized to $v,\dots,vt^{n-1}$) of a skew
Macdonald polynomial can be expressed as a limit of an elliptic binomial
coefficient, essentially just a value of an elliptic interpolation
function.  If one replaces the skew Macdonald polynomials by such elliptic
binomial coefficients in the $n=1$ case, one finds that both sums still
have only one surviving term, and one is led immediately to an elliptic
analogue of the identity, Lemma~\ref{lem:litt_t} below.

To obtain a more general elliptic analogue, there are two natural
approaches.  The first is to develop a theory of skew interpolation
functions, prove a corresponding skew Cauchy identity, then directly lift
Macdonald's argument to the elliptic level.  Roughly speaking, skew
interpolation functions should give the coefficients in a generalized
branching rule:
\begin{align}
\cR^{*(n+m)}_{\blambda}&(x_1,\dots,x_n,y_1,\dots,y_m;t_0,u_0;t;p,q)\notag\\
&{}=
\sum_{\bmu}
\cR^{*(m,n)}_{\blambda/\bmu}(y_1,\dots,y_m;t_0,u_0;t;p,q)
\cR^{*(n)}_{\bmu}(x_1,\dots,x_n;t_0,u_0;t;p,q).
\label{eq:finite_skew_branch}
\end{align}
(In contrast to the Macdonald case, these coefficients depend on $n$ in a
slightly nontrivial way.  Also, recall from \cite{xforms} that bold greek
letters denote pairs of partitions.)  Since these coefficients are
understood for $m=1$ (a special case of \cite[Thm.~4.16]{bctheta}), one
could simply define skew interpolation functions by induction, giving an
$m$-fold sum (in which each individual sum is over partitions).  However,
it turns out that one can use connection coefficients together with the
existence of a special case of interpolation functions expressible as a
product to obtain these coefficients via a single sum.  Moreover, if the
arguments $y_1,\dots,y_m$ contain partial geometric progressions of step
$t$, the coefficients of the sum simplify accordingly, and one is thus led
to the definition of Section~\ref{sec:skew}.  (See
Theorem~\ref{thm:lifted_interp} for the relation between the skew
interpolation functions so defined and ordinary interpolation functions;
the remark following the theorem expresses the above expansion coefficients
in terms of skew interpolation functions.)  A suitable analytic
continuation argument gives an analogue of the Cauchy identity
(Theorem~\ref{thm:cauchy_skew}), and then Macdonald's argument lifts to
give an elliptic Littlewood identity, Theorem~\ref{thm:litt_skew}.

The other natural approach to an elliptic analogue is to retain the use of
binomial coefficients (i.e., restrict one's attention to principally
specialized skew Macdonald polynomials), but hope for an analogue with
additional parameters.  It turns out that enough degrees of freedom survive
in the choice of coefficients that one can use those
coefficients to enforce termination, giving Theorem~\ref{thm:genlitt}
below.  Moreover, the structure of the coefficients is such that one can
analytically continue one of the two sums to a suitable contour integral,
Theorem~\ref{thm:int_litt}.  This in turns suggests a further extension in
which both sides are integrals, stated as Conjecture~\ref{conj:littbig},
for which we can prove a number of special cases.

The ``integral=sum'' version of the identity has a particularly striking
interpretation coming from the fact that one can invert the elliptic
binomial coefficients to move the sum inside the integral.  The resulting
sum of interpolation functions in the integrand then becomes a special case
of the elliptic biorthogonal functions of \cite{xforms,bctheta}, and one
thus deduces that a certain integral of such functions vanishes unless the
indexing partition (or, rather, partition pair) has the form $\bmu^2$.
This is the elliptic analogue of a result proved for Koornwinder
polynomials in \cite{vanish}, and in fact gives a stronger result even at
the Koornwinder level, since the techniques of \cite{vanish} gave no
information about the nonzero values.  This suggests in turn that the other
results of \cite{vanish} involving the same vanishing condition should also
be related to our elliptic Littlewood identity, and indeed we have been
able to formulate two conjectures along those lines,
Conjectures~\ref{conj:vanKsmall} and \ref{conj:vanPsmall}, which again hold
in a number of special cases, and have three different results of
\cite{vanish} as limiting cases.  In particular, every result of
\cite{vanish} that has $\lambda=\mu^2$ as the nonvanishing condition is a
limit of either Corollary~\ref{cor:vant} or one of
Conjectures~\ref{conj:vanKsmall} or \ref{conj:vanPsmall}.  (We also give
analogues for the results with condition $\lambda=2\mu$, but it remains an
open problem to lift the remaining vanishing theorems to the elliptic
level, even conjecturally.)  For instance, one limit of the latter
conjecture is the fact that
\[
\int
P_\lambda(\dots,z_i^{\pm 1},\dots;q,t)
\prod_{1\le i<j\le n}
  \frac{(z_i^{\pm 1}z_j^{\pm 1};q)}
       {(t z_i^{\pm 1}z_j^{\pm 1};q)}
\prod_{1\le i\le n}
  \frac{(z_i^{\pm 2};q)}{(tz_i^{\pm 2};q)}
  \frac{dz_i}{2\pi\sqrt{-1}z_i}
\]
vanishes unless $\lambda=\mu^2$ for some $\mu$, which in turn is a
$(q,t)$-analogue of the representation-theoretic fact that the integral of
a Schur function over the symplectic group similarly vanishes (equivalent
by Frobenius reciprocity to the fact that only those Schur functions appear
in the classical Littlewood identity).

Macdonald also gave a dual version of the Littlewood identity, in which
rather than summing over partitions with even multiplicities, one sums over
partitions with even parts.  (Littlewood's original version of this
identity gives the decomposition of $S^*(S^2(\C^n))$:
\[
\sum_\mu s_{2\mu}(x_1,\dots,x_n)
=
\prod_{1\le i\le j\le n} (1-x_ix_j)^{-1},
\]
and is related to the invariants of ${\mathrm O}_n$ inside irreducible
representations of $\GL_n$.)  This dual Littlewood
identity can, of course, be obtained from the usual Littlewood identity by
simply applying Macdonald's involution to conjugate the partitions
involved.  One can naturally do the same for the elliptic Littlewood
identities, but a new behavior arises.  For the $\bmu^2$-type Littlewood
identity, there is an analytical symmetry between the parameters $p$
(specifying an elliptic curve) and $q$ (specifying a point on that curve),
which is broken by duality.  If one attempts to restore this symmetry after
dualizing, one finds that, in contrast to the $\bmu^2$-type Littlewood
identity, which is a product of two equivalent identities, one
$p$-elliptic, and one $q$-elliptic, the restoration of symmetry in the dual
identity requires that one multiply by a conjectural $q$-elliptic identity
which is {\em not} equivalent to the original dual identity.  Moreover,
this partner identity itself has a different broken symmetry, namely the
natural action of $\SL_2(\Z)$ as modular transformations of the family of
elliptic curves.  One thus finds that each of our identities and
conjectures leads to a whole family of conjectures in this way; the
Littlewood identity itself gives rise to three conjectural integral
transforms, while the other vanishing conjectures correspond to seven
different integral transforms.  The latter group of conjectures (a single
orbit under the various formal symmetries) is particularly interesting, as
even without the interpolation functions in the integrands, they would give
rise to new transformations of higher-order elliptic Selberg integrals
(specifically, quadratic transformations).  In particular, several of the
special cases we prove give nontrivial identities of this form; of
particular note is Theorem~\ref{thm:vanpq_ev}, which expresses certain $2$-
and $3$-dimensional elliptic Selberg integrals as explicit linear
combinations of univariate integrals and constants.  See also
\cite{vandeBultFJ:2011}, which proves the special case $\blambda=0$ of
Conjectures~\ref{conj:vantq} and \ref{conj:vantm} below.  It can be shown
(work in progress) that this implies the $\blambda=0$ cases of the
remaining ``Q'' conjectures; for Conjecture~\ref{conj:vanmt}, this follows
from Conjecture~\ref{conj:vantm} by Proposition~\ref{prop:e7ab} below, but
the other cases require new machinery beyond the scope of the present work.

The plan of the paper is as follows.  After a discussion of notation at the
end of this introduction, we proceed in Section~\ref{sec:skew} to define
our skew interpolation functions, and discuss a number of their properties,
especially their connection to ordinary interpolation functions.  (We also
state a transformation of higher-order elliptic Selberg integrals
(conjectural in the original version of this paper, since proved by Van de
Bult) related to one of those properties, largely because the same
conjecture arose in a different context while working on \cite{xforms}.)
Then in Section~\ref{sec:cauchy}, we discuss the corresponding analogues of
the Cauchy identity, along with some necessary preliminaries concerning
when skew interpolation functions can be guaranteed to vanish, thus making
the relevant sums finite.  Section~\ref{sec:litt} gives the two main forms
of the elliptic Littlewood identity, as well as the three associated
conjectures at the integral level.  Finally, Section~\ref{sec:var}
discusses a number of conjectures related to the vanishing integrals of
\cite{vanish}, with sketches of proofs of various special cases.  Note that
although this last section may seem on first glance to have drifted away
from the theme of the paper, the corresponding ``vanishing'' conjectures,
when degenerated to identities of Macdonald or Koornwinder polynomials,
become Littlewood-type identities in a suitable limit as the number of
variables tends to infinity.  (More precisely, taking the limit
$n\to\infty$ as in \cite{bcpoly} gives either Macdonald's Littlewood
identity, its dual, or an identity originally conjectured by Kawanaka
\cite{KawanakaN:1999} and recently proved in
\cite{LangerR/SchlosserMJ/WarnaarSO:2009} (see also the discussion after
Conjecture~\ref{conj:littbig_m} below, which sketches an alternate proof).)

{\bf Acknowledgements} The author would like to thank V. Spiridonov and
O. Warnaar for helpful comments on an earlier draft, along with R. Askey
and M. Rahman for helpful discussions regarding the analogue of
Corollary~\ref{cor:dual_vanq} for Askey-Wilson polynomials, and F. van de
Bult for settling one conjecture from the original version and providing
substantial additional evidence for several more.  The author would also
like to thank P. Forrester for hosting the author's sabbatical at the
University of Melbourne, during which several key portions of the work were
done.  This work was supported in part by NSF Grants No. DMS-0401387,
DMS-0833464, and DMS-1001645; in addition, much of the work was performed
while the author was affiliated with the University of California at Davis.

{\bf Notation}.  We use the notation of \cite{bctheta} and \cite{xforms}.
In particular, bold-face greek letters refer to pairs of partitions; if
only one of the partitions is nonzero, we will either give the partition
pair explicitly, or rewrite using the notation of \cite{bctheta},
explicitly breaking the symmetry between $p$ and $q$.  Thus, for instance,
the interpolation functions are denoted by
\[
\cR^{*(n)}_{\blambda}(z_1,\dots,z_n;a,b;t;p,q),
\]
which factors as
\[
\cR^{*(n)}_{\lambda,\mu}(z_1,\dots,z_n;a,b;t;p,q)
=
R^{*(n)}_{\lambda}(z_1,\dots,z_n;a,b;p,t;q)
R^{*(n)}_{\mu}(z_1,\dots,z_n;a,b;q,t;p),
\]
with the first factor $q$-elliptic, and the second $p$-elliptic.  Relations
and operations on single partitions extend to partition pairs in the
obvious way; in particular, $\blambda\subset\bmu$ denotes the product of
the usual inclusion orders on the two pieces.  We will need some additional
notations for partitions.  Of particular importance are $\lambda^2$,
denoting the partition with $\lambda^2_i = \lambda_{\lceil i/2\rceil}$, and
$2\lambda$, denoting the partition with $(2\lambda)_i=2\lambda_i$, both
extending immediately to partition pairs.  If $\lambda_1\le m$, then
$m^n\cdot \lambda$ denotes the partition with
\[
(m^n\cdot\lambda)_i=\begin{cases}
m&i\le n\\
\lambda_{i-n}&i>n.
\end{cases}
\]
If $\ell(\lambda)\le n$, then $m^n+\lambda$ denotes the partition with
\[
(m^n+\lambda)_i = m+\lambda_i.
\]
Finally, if $\lambda_1\le m$ and $\ell(\lambda)\le n$, then
\[
(m^n-\lambda)_i = m-\lambda_{n+i-1}.
\]

We specifically recall the elliptic Gamma function
\[
\Gampq(z)
:=
\prod_{0\le i,j} \frac{1-p^{i+1} q^{j+1}/z}{1-p^i q^j z},
\]
with the convention here (and for $\Gamma^+$, $\theta$, etc.) that multiple
arguments express a product:
\[
\Gampq(z_1,\dots,z_n)
=
\prod_{1\le i\le n} \Gampq(z_i).
\]
This satisfies the functional equations
\begin{align}
\Gampq(qz) &= \theta_p(z)\Gampq(z)\\
\Gampq(pz) &= \theta_q(z)\Gampq(z)\\
\Gampq(pq/z) &= \Gampq(z)^{-1},
\end{align}
where
\[
\theta_p(z) := \prod_{0\le i} (1-p^i z)(1-p^{i+1}/z)
\]
is a theta function ($\theta_p(\exp(2\pi i x))$ is doubly quasiperiodic),
as well as the ``quadratic'' functional equations
\begin{align}
\Gampq(z) &= \Gampqq(z,qz)\\
\Gamppqq(z^2) &= \Gampq(z,-z),
\end{align}
which will be useful below.  The special values
\begin{align}
\Gampqq(q) &= \frac{1}{(q;q^2)} = \frac{(q^2;q^2)}{(q;q)} = (-q;q)\\
\Gampq(-1) &= \frac{(p;p^2)(q;q^2)}{2}\\
\lim_{x\to 1} (1-x)\Gampq(x) &= \frac{1}{(p;p)(q;q)}
\end{align}
will arise as well.  We will also need a third-order elliptic Gamma
function
\[
\Gampqt(z):=\prod_{0\le i,j,k} (1-p^{i+1}q^{j+1}t^{k+1}/z)(1-p^iq^jt^k z),
\]
with functional equations
\begin{align}
\Gampqt(tz) &= \Gampq(z)\Gampqt(z),\\
\Gampqt(pqt/z) &= \Gampqt(z),
\end{align}
and so forth.  (This will only be used to simplify notation; in all of the
cases in which it arises, it will appear only via a ratio that resolves via
the first functional equation into a product of usual elliptic Gamma
functions.)

The elliptic Selberg integral (introduced as the ``elliptic
Macdonald-Morris conjecture'' in \cite{vanDiejenJF/SpiridonovVP:2000}, and
renamed the ``Type II'' integral in the follow-up
\cite{vanDiejenJF/SpiridonovVP:2001}) is the integral with density
\begin{align}
\Delta^{(n)}(z_1,\dots,z_n;u_0,\dots,u_5;t;p,q)&\notag\\
{}:=
\frac{((p;p)(q;q)\Gampq(t))^n}{2^n n!}&
\prod_{1\le i<j\le n}
\frac{\Gampq(t z_i^{\pm 1}z_j^{\pm 1})}{\Gampq(z_i^{\pm 1}z_j^{\pm 1})}
\prod_{1\le i\le n}
\frac{\prod_{0\le r<6} \Gampq(u_r z_i^{\pm 1})}{\Gampq(z_i^{\pm 2})}
\frac{dz_i}{2\pi\sqrt{-1}z_i}
\end{align}
with associated evaluation (\cite{xforms}, conjectured in
\cite{vanDiejenJF/SpiridonovVP:2000})
\[
\int_{C^n}
\Delta^{(n)}(z_1,\dots,z_n;u_0,\dots,u_5;t;p,q)
=
\prod_{0\le i<n}
 \Gampq(t^{i+1})
 \prod_{0\le r<s<6} \Gampq(t^i u_r u_s),
\]
where the parameters satisfy the ``balancing condition''
\[
t^{2n-2}\prod_{0\le r<6} u_r = pq,
\]
and $C$ is a contour such that $C=C^{-1}$, and $C$ contains the rescaled
contour $tC$ together with all points of the form $u_r p^iq^j$.  (If one
allows suitable disjoint unions of contours, this condition can be
satisfied unless $u_ru_s p^iq^jt^k=1$ for some $0\le i,j,k$, $0\le
r,s<6$.)  By convention, the argument $u z_i^{\pm 1}$ to a function
indicates a {\em pair} of arguments $u z_i,u/z_i$, and similarly for $t
z_i^{\pm 1}z_j^{\pm 1}$, etc., so in particular the above integrand is
hyperoctahedrally symmetric.  This determines a natural normalized linear
functional
\[
\langle f \rangle^{(n)}_{u_0,\dots,u_5;t;p,q}
\propto
\int_{C^n}
f(z_1,\dots,z_n)
\Delta^{(n)}(z_1,\dots,z_n;u_0,\dots,u_5;t;p,q),
\]
where $f$ is a linear combination of products of hyperoctahedrally
symmetric $p$- and $q$-elliptic functions such that for some nonnegative
integers $l_r$, $m_r$, the function
\[
f(z_1,\dots,z_n)
\prod_{\substack{1\le i\le n\\0\le r\le 5}}
\frac{\Gampq(u_r z_i^{\pm 1})}{\Gampq(p^{-l_r}q^{-m_r}u_r z_i^{\pm 1})}
\]
is holomorphic, and the contour satisifes the conditions appropriate to
\[
\Delta^{(n)}(z_1,\dots,z_n;p^{-l_0}q^{-m_0}u_0,\dots,p^{-l_5}q^{-m_5}u_5;t;p,q);
\]
the integral is normalized so that
\[
\langle 1 \rangle^{(n)}_{u_0,\dots,u_5;t;p,q}=1.
\]
Note that if the contour satisfies the conditions for a given choice of
$l_r$, $m_r$, it satisfies them for all smaller choices, so for a finite
linear combination of such functions, one can (generically) choose a
contour valid for each term simultaneously, giving linearity.  However, the
families of functions we consider involve unbounded values of $l_0$, $m_0$,
and thus one cannot simply fix a single contour for {\em every} function in
the family.

The biorthogonal functions
\[
\tcR^{(n)}_{\blambda}(z_1,\dots,z_n;t_0{:}t_1,t_2,t_3;u_0,u_1;t;p,q)
\]
of \cite{xforms,bctheta} satisfy biorthogonality with respect to this
linear functional, i.e.,
\[
\langle
\tcR^{(n)}_{\blambda}(z_1,\dots,z_n;t_0{:}t_1,t_2,t_3;u_0,u_1;t;p,q)
\tcR^{(n)}_{\bmu}(z_1,\dots,z_n;t_0{:}t_1,t_2,t_3;u_1,u_0;t;p,q)
\rangle_{t_0,t_1,t_2,t_3,u_0,u_1;t;p,q}
\]
vanishes unless $\blambda=\bmu$.  More precisely, they are characterized
for generic parameters by this property and the triangularity property that
for any partition pairs $\bkappa$, $\blambda$, and integers $(l,m)$ with
$(l,m)\ge \bkappa_1,\blambda_1$ (relative to the product ordering),
\[
\lim_{z_i\to (p,q)^{-\bkappa_i}t^{i-1}u_0}
\prod_{1\le i\le n} \theta(pqz_i^{\pm 1}/u_0;p,q)_{l,m}
\tcR^{(n)}_{\blambda}(z_1,\dots,z_n;t_0{:}t_1,t_2,t_3;u_0,u_1;t;p,q)
=
0
\]
unless $\bkappa\subset\blambda$; here and below, $(p,q)^{(l,m)}:=p^lq^m$.
The biorthogonal functions are normalized by taking
\[
\tcR^{(n)}_{\blambda}(\dots,t^{n-i}t_0,\dots;t_0{:}t_1,t_2,t_3;u_0,u_1;t;p,q)=1;
\]
though this breaks the symmetry between the four $t_r$ parameters (only
mildly: the required changes in normalization have explicit product
formulas), it makes the biorthogonal function with index $(0,\mu)$
$p$-elliptic in every parameter.

We will also need higher order versions of the elliptic Selberg integral;
we define
\[
\II^{(m)}_n(u_0,\dots,u_{2m+5};t;p,q)
:=
\int_{C^n}
\Delta^{(n)}(z_1,\dots,z_n;u_0,\dots,u_{2m+5};t;p,q),
\]
subject to the balancing condition
\[
t^{2n-2}\prod_{0\le r<2m+6} u_r = (pq)^{m+1},
\]
in which the density is obtained from the original density ($m=0$) by
replacing
\[
\prod_{0\le r<6} \Gampq(u_r z_i^{\pm 1})
\mapsto
\prod_{0\le r<2m+6} \Gampq(u_r z_i^{\pm 1}),
\]
and the contour condition is extended in the obvious way.  In particular,
if $u_{2m+4}u_{2m+5}=pq$, then the reflection equation for $\Gampq$ causes
the two corresponding factors to cancel, reducing $m$ by $1$.  When $n=1$,
the higher-order elliptic Selberg integral is essentially independent of
$t$, apart from the factor $\Gampq(t)$; we thus define the higher-order
elliptic beta integral \cite{SpiridonovVP:2001} by
\[
I^{(m)}(u_0,\dots,u_{2m+5};p,q)
:=
\Gampq(t)^{-1}
\II^{(m)}_1(u_0,\dots,u_{2m+5};t;p,q);
\]
note that the constraint that the contour $C$ contains $tC$ is irrelevant
in this case.

When $m=1$, the elliptic Selberg integral satisfies an important
transformation (a special case of \cite[Thm.~9.7]{xforms}), namely that
\[
\II_n(u_0,\dots,u_7;t;p,q)
=
\II_n(u_0/v,u_1/v,u_2/v,u_3/v,u_4v,u_5v,u_6v,u_7v)
\prod_{\substack{0\le i<n\\0\le r<s<4}}\Gampq(t^i u_ru_s,t^i u_{r+4}u_{s+4}),
\]
where
$v^2=\frac{pqt^{1-n}}{u_4u_5u_6u_7}=\frac{u_0u_1u_2u_3}{pqt^{1-n}}=\sqrt{\frac{u_0u_1u_2u_3}{u_4u_5u_6u_7}}$.
Together with permutations of the parameters, this generates the Weyl group
of type $E_7$.  We note the following special case, which will arise
repeatedly in Section~\ref{sec:var} below.

\begin{prop}\label{prop:e7ab}
Define a function
\[
F_n(t_0,t_1,t_2,t_3;a,b;t;p,q)
=
\II_n(a^{\pm 1/2}t_0,a^{\pm 1/2}t_1,a^{\pm 1/2}t_2,(b^2a)^{\pm
  1/2}t_3;t;p,q)
\prod_{\substack{0\le i<n\\0\le r<s<3}}
  \frac{\Gampq(t^i b t_rt_s,t^i ab t_rt_s)}
       {\Gampq(t^i t_rt_s,t^i t_rt_s/a)}
,
\]
subject to the balancing condition $t^{n-1}t_0t_1t_2t_3=pq$.  Then
$F_n$ is invariant under permuting $t_0,t_1,t_2,t_3$ and under swapping $a$
and $b$.
\end{prop}

\begin{rem}
This function satisfies additional transformations
\[
F_n(t_0,t_1,t_2,t_3;a,b;t;p,q)
=
F_n(t_0,t_1,t_2,t_3;ab,1/b;t;p,q)\prod_{0\le i<n,0\le r<s<4}\Gampq(t^i t_rt_s b)
\]
and
\[
F_n(t_0,t_1,t_2,t_3;a,b;t;p,q)
=
F_n(\gamma/t_0,\gamma/t_1,\gamma/t_2,\gamma/t_3;a,b;t;p,q)
\prod_{\substack{0\le i<n\\0\le r<4}} \Gampq(t^i t_r^2)
\]
where $\gamma=(t^{1-n}pq)^{1/2}$.  These transformations generate a Weyl
group $B_3\times G_2$, and in much the same way as the general order 1
elliptic Selberg integral satisfies a formal $E_8$ symmetry (see discussion
in \cite[\S9]{xforms}), this group formally extends to an action of
$F_4\times G_2$.
\end{rem}

The factors
\[
\Delta^0_\blambda(a|b_0,\dots,b_{n-1};t;p,q)\quad\text{and}\quad
\Delta_\blambda(a|b_0,\dots,b_{n-1};t;p,q)
\]
that appear below are certain multivariate $q$-symbols (see the
introduction of \cite{xforms}).  The first is defined by
\[
\Delta^0_\blambda(a|b_0,\dots,b_{n-1};t;p,q)
=
\prod_{0\le r<n}
  \frac{\cC^0_\blambda(b_r;t;p,q)}
       {\cC^0_\blambda(pqa/b_r;t;p,q)},
\]
where
\[
\cC^0_\blambda(x;t;p,q):=\prod_{1\le i} \theta(t^{1-i}x;p,q)_{\blambda_i},
\]
and
\[
\theta(x;p,q)_{l,m}:=
\prod_{0\le j<l} \theta_q(p^j x)
\prod_{0\le j<m} \theta_p(q^j x).
\]
Note that
\[
\Delta^0_{\lambda,\mu}(a|b_0,\dots,b_{n-1};t;p,q)
=
\Delta^0_{\lambda,0}(a|b_0,\dots,b_{n-1};t;p,q)
\Delta^0_{0,\mu}(a|b_0,\dots,b_{n-1};t;p,q),
\]
and if $n=2m$, $\prod_{0\le r<2m} b_r = (pqa)^m$, then both factors are
elliptic subject to this constraint; i.e.,
\[
\Delta^0_{0,\mu}(a|b_0,\dots,b_{2m-1};t;p,q)
\]
is invariant under shifting the parameters by integer powers of $p$ such
that the balancing condition remains satisfied.

The other $\Delta$-symbol is more complicated:
\[
\Delta_\blambda(a|b_0,\dots,b_{n-1};t;p,q)
:=
\Delta^0_\blambda(a|b_0,\dots,b_{n-1};t;p,q)
\frac{\cC^0_{2\blambda^2}(pqa;t;p,q)}
     {\cC^-_\blambda(pq,t;t;p,q)\cC^+_\blambda(a,pqa/t;t;p,q)}
\]
where
\begin{align}
\cC^-_\blambda(x;t;p,q)
&:=
\prod_{1\le i\le j}
\frac{\theta(t^{j-i}x;p,q)_{\blambda_i-\blambda_{j+1}}}
     {\theta(t^{j-i}x;p,q)_{\blambda_i-\blambda_j}}\\
\cC^+_\blambda(x;t;p,q)
&:=
\prod_{1\le i\le j}
\frac{\theta(t^{2-i-j}x;p,q)_{\blambda_i+\blambda_j}}
     {\theta(t^{2-i-j}x;p,q)_{\blambda_i+\blambda_{j+1}}}.
\end{align}
The key property of $\Delta_{\blambda}$ is that the $\blambda$-dependent
factor of the residue of the elliptic Selberg integrand $\Delta^{(n)}$ at
the point $(\dots,(p,q)^{\blambda}t^{n-i}u_0,\dots)$ is
\[
\Delta_{\blambda}(t^{2n-2}u_0^2|t^n,t^{n-1}u_0u_1,\dots,t^{n-1}u_0u_{2m+5};t;p,q).
\]
The corresponding balancing condition to
ensure ellipticity is, for $n=2m$, that $\prod_{0\le r<2m} b_r = (t/pq)
(pqa)^{m-1}$.

In many respects, the most natural elliptic analogue of the Macdonald
polynomials is the interpolation functions, a special case of the
biorthogonal functions given by
\[
\cR^{*(n)}_{\blambda}(;t_0,u_0;t;p,q)
=
\Delta^0_{\blambda}(t^{n-1}t_0/u_0|t^{n-1}t_0t_1,t_0/t_1;t;p,q)
\tcR^{(n)}_{\blambda}(;t_1{:}t_0,t_2,t_3;u_0,t^{1-n}/t_0;t;p,q)
\]
with $t^{n-1}t_1t_2t_3u_0=pq$.  Note that the left-hand side is independent
of the remaining degrees of freedom.  The key property of the interpolation
functions is that
\[
\cR^{*(n)}_{\blambda}(\dots,(p,q)^{\bmu_i}t^{n-i}a,\dots;a,b;t;p,q)
=
0
\]
unless $\blambda\subset\bmu$ (\cite[Cor.~8.12]{xforms}); this property and
the triangularity property are related by a complementation symmetry, and
together determine the interpolation function up to normalization, which is
determined by
\[
\cR^{*(n)}_{\blambda}(\dots,t^{n-i}v,\dots;a,b;t;p,q)
=
\Delta^0_{\blambda}(t^{n-1}a/b|t^{n-1}av,a/v;t;p,q).
\]
The interpolation functions play a special role in the theory of the
elliptic biorthogonal functions, as certain connection coefficients between
biorthogonal functions with slightly different parameters can be expressed
via values of interpolation functions at partitions
\cite[Cor.~5.7]{bctheta}.  As a special case, any biorthogonal function can
be expanded as a linear combination of interpolation functions in which the
coefficients are themselves values of interpolation functions
\cite[Defn.~12 and Thm.~5.3]{bctheta}.

These values of interpolation functions appear frequently enough to merit
their own notation: we define
\[
\binom{\blambda}{\bmu}_{[a,b];t;p,q}
:=
\Delta_{\bmu}(a/b|t^n,1/b;t;p,q)
\cR^{*(n)}_{\bmu}(\dots,\sqrt{a}(p,q)^{\blambda_i}t^{1-i},\dots;t^{1-n}\sqrt{a},b/\sqrt{a};t;p,q);
\]
this is independent of the choice of square root, and factors as
\[
\binom{\lambda,\kappa}{\mu,\nu}_{[a,b];t;p,q}
=
\binom{\lambda}{\mu}_{[a,b];p,t;q}
\binom{\kappa}{\nu}_{[a,b];q,t;p}
\]
where the first factor is $q$-elliptic in $a$, $b$, $p$, and $t$, and
imilarly for the second factor.  We also use the alternate normalization of
\cite{bctheta}, which in the $p,q$-symmetric version reads
\[
\obinomE{\blambda}{\bmu}_{[a,b](v_1,\dots,v_k);t;p,q}
:=
\frac{\Delta^0_{\blambda}(a|b,v_1,\dots,v_k;t;p,q)}
     {\Delta^0_{\bmu}(a/b|1/b,v_1,\dots,v_k;t;p,q)}
\binom{\blambda}{\bmu}_{[a,b];t;p,q}.
\]
The binomial coefficients so normalized are products of elliptic functions
if $k=3$, $bv_1v_2v_3=(pqa)^2$.

%

\section{Skew interpolation functions}\label{sec:skew}

Consider the following generalized elliptic hypergeometric sum:
\begin{align}
\cR^*_{\blambda/\bkappa}([v_0,\dots,v_{2n-1}];a,b;t;p,q)
:=
\sum_{\bkappa\subset\bmu\subset\blambda}&
\obinomE{\blambda}{\bmu}_{[a/b,ab/pq];t;p,q}
\obinomE{\bmu}{\bkappa}_{[pq/b^2,pq\prod_{0\le r<2n} v_r/ab];t;p,q}\notag\\
&\times\Delta^0_{\bmu}(pq/b^2|pq/bv_0,pq/bv_1,\dots,pq/bv_{2n-1};t;p,q);
\end{align}
as the notation suggests, this will turn out to be our desired skew version
of the interpolation functions.  Note that each term in the rescaled sum
\begin{align}
\hat{\cR}{}^*_{\blambda/\bkappa}([v_0,\dots,v_{2n-1}]&;a,b;t;p,q)
:={}\notag\\
&\frac{\Delta^0_{\bkappa}(a/b\prod_{0\le r<2n} v_r|ab/pq\prod_{0\le r<2n} v_r;t;p,q)}
     {\Delta^0_{\blambda}(a/b|ab/pq;t;p,q)}
\cR^*_{\blambda/\bkappa}([v_0,\dots,v_{2n-1}];a,b;t;p,q)
\end{align}
is the product of $p$-abelian and $q$-abelian factors, so the same applies
to this rescaled sum; however, the rescaling introduces unfortunate poles,
so we will prefer to use the not-quite-elliptic form unless that would
introduce complicated factors from quasiperiodicity.  This is a generalized
elliptic hypergeometric sum in the same sense as the identities of
\cite{bctheta}; in particular, it includes the following very-well-poised,
balanced, and terminating multivariate elliptic hypergeometric series as a
special case:
\begin{align}
\hat{\cR}{}^*_{(l,m)^n/0}&([v_0,\dots,v_{2k-1}];a,b;t;p,q)
={}\\
&\sum_{\bmu\subset(l,m)^n}
\Delta_{\bmu}(pq/b^2|t^n,p^{-l}q^{-m},p^lq^ma/t^{n-1}b,pq/bv_0,pq/bv_1,\dots,pq/bv_{2k-1},pq\prod_{0\le r<2k}v_r/ab;t;p,q).
\notag
\end{align}
(Such sums arise as limiting cases of order $k-1$ elliptic Selberg
integrals via residue calculus.) We note that the skew interpolation
function is invariant under permutations of its arguments, as well as under
insertion or deletion of pairs $x,1/x$.  (The last statement follows from
the fact that
\[
\Delta^0_{\bmu}(pq/b^2|pq/bx,pqx/b;t;p,q)=1,
\]
which in turn is immediate from the definition.)  In particular, the
arguments are not directly arguments of interpolation functions, but play a
more plethystic role.  Roughly speaking, this corresponds to the plethystic
substitution
\[
p_k\mapsto 
\sum_{0\le r<2n} \frac{v_r^k-v_r^{-k}}{t^{k/2}-t^{-k/2}}
\]
at the trigonometric level, so that an ordinary argument corresponds to a
pair $t^{1/2}x,t^{1/2}/x$ of plethystic arguments.  (Compare
Theorem~\ref{thm:lifted_interp} below.)

The two main identities of \cite{bctheta} both involved sums of this form,
and thus one has the following.

\begin{prop} \cite[Cor.~4.3]{bctheta}
With no arguments, the skew interpolation function is a delta function:
\[
\cR^*_{\blambda/\bkappa}([];a,b;t;p,q)=\delta_{\blambda\bkappa}.
\]
\end{prop}

\begin{prop} \cite[Thm.~4.1]{bctheta}
With two arguments, the skew interpolation function is an elliptic binomial
coefficient:
\[
\cR^{*}_{\blambda/\bkappa}([v_0,v_1];a,b;t;p,q)
=
\obinomE{\blambda}{\bkappa}_{[a/b,v_0v_1](a/v_0,a/v_1);t;p,q}.
\]
\end{prop}

\begin{rem}
When $v_0v_1=1$, so we can eliminate the two arguments to the skew
interpolation function, the right-hand side specializes to a delta
function as required.
\end{rem}

\begin{prop} \cite[Thm.~4.9, Cor.~4.11]{bctheta}
With four arguments, the skew interpolation function has the alternate
expressions
\[
\cR^{*}_{\blambda/\bkappa}([v_0,v_1,v_2,v_3];a,b;t;p,q)
=
\sum_{\bmu}
\obinomE{\blambda}{\bmu}_{[a/b,v_0v_1](a/v_0,a/v_1);t;p,q}
\obinomE{\bmu}{\bkappa}_{[a/v_0v_1b,v_2v_3](a/v_0v_1v_2,a/v_0v_1v_3);t;p,q}
\label{eq:bailey}
\]
and
\begin{align}
\cR^*_{\blambda/\bkappa}([v_0,v_1,v_2,v_3];a,b;t;p,q)
={}&
\frac{\Delta^0_{\blambda}(a/b|a/v_0,a/v_1,a/v_2,a/v_3;t;p,q)}
     {\Delta^0_{\bkappa}(a/bV|av_0/V,av_1/V,av_2/V,av_3/V;t;p,q)}\notag\\
&\times\sum_\bmu
\frac{
\obinomE{\blambda}{\bmu}_{[a/b,pqV/ab];t;p,q}
\obinomE{\bmu}{\bkappa}_{[a^2/pqV,ab/pq];t;p,q}}
{\Delta^0_{\bmu}(a^2/pqV|a/v_0,a/v_1,a/v_2,a/v_3;t;p,q)},
\end{align}
where $V=v_0v_1v_2v_3$.
\end{prop}

In equation \eqref{eq:bailey}, the binomial coefficients can be expressed
in skew interpolation functions, giving
\[
\cR^{*}_{\blambda/\bkappa}([v_0,v_1,v_2,v_3];a,b;t;p,q)
=
\sum_{\bmu}
\cR^{*}_{\blambda/\bmu}([v_0,v_1];a,b;t;p,q)
\cR^{*}_{\bmu/\bkappa}([v_2,v_3];a/v_0v_1,b;t;p,q).
\]
This generalizes considerably.

\begin{prop}
The skew interpolation functions satisfy the identity
\begin{align}
\cR^*_{\blambda/\bkappa}([v_0,\dots,v_{2k-1},w_0,\dots,w_{2l-1}];a,b;t;p,q)
=
\sum_{\bmu}&
\cR^*_{\blambda/\bmu}([v_0,\dots,v_{2k-1}];a,b;t;p,q)\notag\\
&\times\cR^*_{\bmu/\bkappa}([w_0,\dots,w_{2l-1}];a/v_0\cdots v_{2k-1},b;t;p,q).
\end{align}
\end{prop}

\begin{proof}
If we expand the skew interpolation functions on the right via the
definition, the inner sum over $\bmu$ is itself a skew interpolation
function with no arguments, and thus the inner sum collapses as required.
\end{proof}

Thus to justify the name ``skew interpolation function'', it remains only
to show that when $\bkappa=0$, we obtain (a generalization of) the usual
interpolation function.

\begin{thm}\label{thm:lifted_interp}
The interpolation functions have the expression
\[
\cR^{*(n)}_{\blambda}(z_1,\dots,z_n;a,b;t;p,q)
=
\Delta^0_{\blambda}(t^{n-1}a/b|pqa/tb;t;p,q)
\cR^*_{\blambda/0}([t^{1/2}z_1^{\pm 1},\dots,t^{1/2}z_n^{\pm 1}];t^{n-1/2}a,t^{1/2}b;t;p,q).
\]
\end{thm}

\begin{proof}
By the connection coefficient identity \cite[Cor.~4.14]{bctheta}, we can
write
\[
\cR^{*(n)}_{\blambda}(z_1,\dots,z_n;a,b;t;p,q)
=
\sum_{\bmu}
\obinomE{\blambda}{\bmu}_{[t^{n-1}a/b,t^nab/pq](pqa/tb);t;p,q}
\cR^{*(n)}_{\bmu}(z_1,\dots,z_n;pq/t^nb,b;t;p,q).
\]
But the new interpolation functions are of ``Cauchy'' type, so by
\cite[Prop.~3.9]{bctheta},
\[
\cR^{*(n)}_{\bmu}(z_1,\dots,z_n;pq/t^nb,b;t;p,q)
=
\Delta^0_{\bmu}(pq/tb^2|pqz_1^{\pm 1}/tb,\dots,pqz_n^{\pm 1}/tb;t;p,q)
\label{eq:Cauchy_prod}
\]
as required.
\end{proof}

\begin{rems}
This can also be proved by induction via the branching rule
\cite[Thm.~4.16]{bctheta}.  Similarly, we find that the coefficients of
\eqref{eq:finite_skew_branch} are given by
\begin{align}
\cR^{*(m,n)}_{\blambda/\bmu}&(z_1,\dots,z_m;a,b;t;p,q)\notag\\
&=
\frac{\Delta^0_{\blambda}(t^{n+m-1}a/b|pqa/tb;t;p,q)}
     {\Delta^0_{\bmu}(t^{n-1}a/b|pqa/tb;t;p,q)}
\cR^*_{\blambda/\bmu}([t^{1/2}z_1^{\pm 1},\dots,t^{1/2}z_m^{\pm 1}];t^{n+m-1/2}a,t^{1/2}b;t;p,q)
\end{align}
\end{rems}

\begin{rems}
  Thus ordinary interpolation functions correspond to the case that the
  arguments multiply pairwise to $t$; similarly, the skew interpolation
  functions of \cite{CoskunH/GustafsonRA:2006} correspond to the special
  case in which the arguments multiply pairwise to some general, but
  fixed, $r$.
\end{rems}

\begin{rems}
The inverse expansion:
\[
\Delta^0_{\blambda}(pq/tb^2|pqz_1^{\pm 1}/tb,\dots,pqz_n^{\pm 1}/tb;t;p,q)
=
\sum_{\bmu}
\obinomE{\blambda}{\bmu}_{[pq/tb^2,pq/t^nab](pqa/tb);t;p,q}
\cR^{*(n)}_{\bmu}(z_1,\dots,z_n;a,b;t;p,q)
\]
holds even if $\ell(\blambda)>n$ (assuming generic parameters).  Indeed, if
$k$ is sufficiently large, so that $n+k\ge \ell(\blambda)$, then one may
set $z_{n+i}=t^{-i}a$ in
\begin{align}
\Delta^0_{\blambda}(pq/tb^2|pqz_1^{\pm 1}/tb&,\dots,pqz_{n+k}^{\pm
  1}/tb;t;p,q)
\notag\\
&{}=
\sum_{\bmu}
\obinomE{\blambda}{\bmu}_{[pq/tb^2,pq/t^nab](pqt^{-k}a/tb);t;p,q}
\cR^{*(n+k)}_{\bmu}(z_1,\dots,z_{n+k};t^{-k}a,b;t;p,q)
\end{align}
to obtain the desired result.  This will be useful in the sequel, as
products of this form satisfy a number of useful identities.  For
convenience in notation, we will use the product expression
\eqref{eq:Cauchy_prod} to extend the Cauchy-type interpolation functions to
the case that the indexing partition has more than $n$ parts, as the above
considerations eliminate most of the dangers in such an extension.
\end{rems}

With this in mind, we refer to the functions $\cR^*_{\blambda/0}$ as {\em
  lifted interpolation functions}; these seem to be about as close as one
can hope to get to an elliptic analogue of the lifted interpolation
polynomials of \cite[\S 6]{bcpoly}.  These functions have a somewhat
surprising additional symmetry.

\begin{prop}\label{prop:symmA2n}
The lifted interpolation function
$\cR^*_{\blambda/0}([v_0,\dots,v_{2n-1}];a,b;t;p,q)$ is invariant under
permutations of the $2n+1$ values
\[
v_0,\dots,v_{2n-1},a/\prod_{0\le r<2n} v_r.
\]
\end{prop}

\begin{proof}
Since
\[
\obinomE{\bmu}{0}_{[a,b];t;p,q}
=
\Delta^0_{\bmu}(a|b;t;p,q),
\]
we have
\begin{align}
\cR^*_{\blambda/0}([v_0,\dots,v_{2n-1}];a,b;t;p,q)
=
\sum_{\bmu}&
\obinomE{\blambda}{\bmu}_{[a/b,ab/pq];t;p,q}\\
&\times\Delta^0_{\bmu}(pq/b^2|pq/bv_0,pq/bv_1,\dots,pq/bv_{2n-1},pq\prod_{0\le
  r<2n} v_r/ab;t;p,q),\notag
\end{align}
which manifestly has the stated symmetry.
\end{proof}

It follows that the connection coefficient formula of \cite{bctheta}
extends, and in a particularly nice form.

\begin{cor}
One has the identity
\[
\cR^*_{\blambda/0}([v_0,\dots,v_{2n-1}];a,b;t;p,q)
=
\sum_{\bkappa}
\cR^*_{\blambda/\bkappa}([a/V,V/a'];a,b;t;p,q)
\cR^*_{\bkappa/0}([v_0,\dots,v_{2n-1}];a',b;t;p,q),
\]
where $V=\prod_{0\le r<2n}v_r$.
\end{cor}

\begin{proof}
Indeed, this reduces to showing
\[
\cR^*_{\blambda/0}([a/V,V/a',v_0,\dots,v_{2n-1}];a,b;t;p,q)
=
\cR^*_{\blambda/0}([v_0,\dots,v_{2n-1}];a,b;t;p,q),
\]
and this is simply deletion of the pair $V/a',a'/V$ from the left-hand
side, after applying Proposition~\ref{prop:symmA2n}.
\end{proof}

The relation of skew interpolation functions to the binomial coefficients
means that we can expect most symmetries of the latter to extend.  We begin
with duality, which breaks the symmetry between $p$ and $q$, but will be
useful in the sequel.  Here we can simply apply the symmetry term-by-term
in the definition.

\begin{prop} \cite[Cor.~4.4]{bctheta}
\[
\cR^*_{(0,\lambda)/(0,\kappa)}([v_0,\dots,v_{2k-1}];a,b;t;p,q)
=
\cR^*_{(0,\lambda')/(0,\kappa')}([v_0,\dots,v_{2k-1}];a,b/qt;1/q;p,1/t)
\]
\end{prop}

The other symmetries do respect the $p$, $q$ symmetry, but lead to
unpleasant scale factors since the skew interpolation functions are not
quite elliptic, so we use the $\hat{\cR}{}^*$ variant.  In particular, this
allows one to prove identities by factoring into $p$-elliptic and
$q$-elliptic factors, then using ellipticity to restore symmetry before
multiplying the identities back together.

\begin{prop} \cite[(4.10)]{bctheta}
If $\ell(\blambda),\ell(\bkappa)\le n$, then
\begin{align}
\hat{\cR}{}^*_{(l,m)^n+\blambda/(l,m)^n+\bkappa}([\dots,v_r,\dots]&;a,b;t;p,q)\notag\\
{}={}&
\Delta^0_{(l,m)^n}(pq/b^2|\dots,pq/bv_r,\dots,Q a/t^{n-1} b,p^2q^2 t^{n-1}
V/Qab;t;p,q)\notag\\
&
\times\left(\frac{\Delta^0_{\blambda}(Q^2 a/b|pqt^{n-1}Q,Q^2a/t^{n-1}b,ab/pq,p^2q^2 Q/b^2;t;p,q)}
     {\Delta^0_{\bkappa}(Q^2 a/Vb|pqt^{n-1}Q,Q^2a/t^{n-1}bV,ab/pqV,p^2q^2
       Q/b^2;t;p,q)}\right)\notag\\
&
\times\hat{\cR}{}^*_{\blambda/\bkappa}([\dots,v_r,\dots];Q a,b/Q;t;p,q),
\end{align}
where $V=\prod_r v_r$ and $Q=p^lq^m$.
\end{prop}

\begin{prop} \cite[Cor.~4.6]{bctheta}
If $\blambda_1,\bkappa_1\le (l,m)$, then
\begin{align}
\hat{\cR}{}^*_{(l,m)^n\cdot\blambda/(l,m)^n\cdot\bkappa}([\dots,v_r,\dots]&;a,b;t;p,q)\notag\\
{}={}&
\Delta^0_{(l,m)^n}(pq/b^2|\dots,pq/bv_r,\dots,Qa/t^{n-1}b,p^2q^2t^{n-1}V/Qab;t;p,q)\notag\\
&
\times\left(\frac{\Delta^0_{\blambda}(a/t^{2n}b|pq/Qt^{n+1},Qa/t^{2n-1}b,ab/pq,p^2q^2/t^nb^2;t;p,q)}
     {\Delta^0_{\bkappa}(a/t^{2n}Vb|pq/Qt^{n+1},Qa/t^{2n-1}bV,ab/pqV,p^2q^2/t^nb^2;t;p,q)}\right)\notag\\
&
\times \hat{\cR}{}^*_{\blambda/\bkappa}([\dots,v_r,\dots];t^{-n}a,t^n b;t;p,q).
\end{align}
\end{prop}

\begin{prop} \cite[Cor.~4.7]{bctheta}
If $\blambda, \bkappa\subset (l,m)^n$, then
\begin{align}
\hat{\cR}{}^*_{(l,m)^n-\bkappa/(l,m)^n-\blambda}([\dots,v_r,\dots]&;a,b;t;p,q)
\notag\\
{}={}&
\Delta^0_{(l,m)^n}(pq/b^2|\dots,pq/bv_r,\dots,Q a/t^{n-1}b,p^2q^2 t^{n-1}V/Q ab;t;p,q)
\notag\\
&\times\left(
\frac{\Delta_{\blambda}(t^{2n-2}bV/Q^2a|t^n,1/Q,t^{n-1}b^2/pqQ,pqV/ab;t;p,q)}
     {\Delta_{\bkappa}(t^{2n-2}b/Q^2a|t^n,1/Q,t^{n-1}b^2/pqQ,pq/ab;t;p,q)}
\right)\notag\\
&\times
\hat{\cR}{}^*_{\blambda/\bkappa}([\dots,v_r,\dots];pqVt^{n-1}/Qa,pq Q/t^{n-1}b;t;p,q).
\end{align}
\end{prop}

The above symmetries each follow by applying the corresponding symmetries
of elliptic binomial coefficients and $\Delta$ symbols to the definition of
the skew interpolation functions.  There is also an analogue of
\cite[Cor.~4.8]{bctheta}, but this is more subtle.  We give this in a
fairly general form, for ease of induction and later application.

\begin{thm}\label{thm:1.12}
Suppose the parameters $v_0,\dots,v_{2k-1}$ can be ordered in such a way
that $v_{2r}v_{2r+1}=t^{n_r}$ with $n_r\in \Z_{\ge 0}$, $0\le r<k$, and let $l$,
$m$, $n$, $n'$ be nonnegative integers with $n'=n+\sum_r n_r$ and
$\blambda_1,\bkappa_1\le (l,m)$.  Then
\begin{align}
\hat{\cR}{}^*_{(l,m)^{n'}\cdot\blambda/(l,m)^{n}\cdot\bkappa}([v_0,\dots,v_{2k-1}]&;a,b;t;p,q)\notag\\
{}={}&
\frac{\prod_{0\le r<k}
        \Delta^0_{(l,m)^{n_r}}(a/bt^{n'-n_r}|av_{2r}/t^{n'},av_{2r+1}/t^{n'};t;p,q)}
     {\Delta^0_{(l,m)^{n'-n}}(a/bt^n|ab/pq,pqa/t^{n'+n}b;t;p,q)}\notag\\
&\times\left(
\frac{\Delta_\bkappa(a/t^{n'+n}b|p^{-l}q^{-m},p^lq^m a/t^{n'-1}b,ab/pqt^{n'},pqt^{n'-n}/ab;t;p,q)}
     {\Delta_\blambda(a/t^{2n'}b|p^{-l}q^{-m},p^lq^m
       a/t^{n'-1}b,ab/pqt^{n'},pq/ab;t;p,q)}\right)\notag\\
&\times
\hat{\cR}{}^*_{\bkappa/\blambda}([v_0,\dots,v_{2k-1}];pq/t^{n}b,pqt^{n'}/a;t;p,q)
\end{align}
\end{thm}

\begin{proof}
When $k=1$, this follows immediately from Corollaries 4.6 and 4.8 of
\cite{bctheta}. (Corollary 4.8 corresponds to the case $n=0$, $k=1$, and
Corollary 4.6 allows one to extend this to $n>0$.)  We then proceed by
induction on $k$.  One first notes that
\begin{align}
\hat{\cR}{}^*_{(l,m)^{n'}\cdot\blambda/(l,m)^{n}\cdot\bkappa}&([v_0,\dots,v_{2k-1}];a,b;t;p,q)\notag\\
&{}=
\sum_{\bmu}
\hat{\cR}{}^*_{(l,m)^{n'}\cdot\blambda/\bmu}([v_0,v_1];a,b;t;p,q)
\hat{\cR}{}^*_{\bmu/(l,m)^{n}\cdot\bkappa}([v_2,\dots,v_{2k-1}];a/t^{n_0},b;t;p,q).
\end{align}
The key observation is that the first factor is
\[
\obinomE{(l,m)^{n'}\cdot\blambda}{\bmu}_{[a/b,t^{n_0}](a/v_0,a/v_1,p^2q^2/b^2);t;p,q},
\]
which vanishes unless
\[
\bmu_i\le ((l,m)^{n'}\cdot\blambda)_i\le \bmu_{i-n_0}.
\]
In particular, for $1\le i\le n'$, $\bmu_i\le (l,m)$, while for $1\le i\le
n'-n_0$, $\bmu_i\ge (l,m)$.  We can thus rewrite the sum as
\[
\sum_{\bnu}
\hat{\cR}{}^*_{(l,m)^{n'}\cdot\blambda/(l,m)^{n'-n_0}\cdot\bnu}([v_0,v_1];a,b;t;p,q)
\hat{\cR}{}^*_{(l,m)^{n'-n_0}\cdot\bnu/(l,m)^{n}\cdot\bkappa}([v_2,\dots,v_{2k-1}];a/t^{n_0},b;t;p,q).
\]
The result follows by applying the symmetry to each factor and simplifying.
\end{proof}

Dually, one has the following identity.

\begin{cor}
Suppose the parameters $v_0,\dots,v_{2k-1}$ can be ordered in such a way
that $v_{2r}v_{2r+1}=Q_r^{-1}$, $Q_r:=p^{l_r}q^{m_r}$, with $l_r,m_r\in
\Z_{\ge 0}$ for $0\le r<k$, and let $l$, $l'$, $m$, $m'$, $n$, be nonnegative
integers with $l'=l+\sum_r l_r$, $m'=m+\sum_r m_r$,
$\ell(\blambda),\ell(\bkappa)\le n$.  Then
\begin{align}
\hat{\cR}{}^*_{(l',m')^n+\blambda/(l,m)^n+\bkappa}([v_0,\dots,v_{2k-1}]&;a,b;t;p,q)\notag\\
{}={}&
\frac{
\prod_{0\le r<k}
  \Delta^0_{(l_r,m_r)^n}(Q'a/Q_rb|Q'a/Q_rv_{2r},Q'a/Q_rv_{2r+1};t;p,q)}
{\Delta^0_{(l'-l,m'-m)^n}(Q a/b|ab/pq,pq Q Q'a/b;t;p,q)}\notag\\
&\times\left(\frac{\Delta_\bkappa(a Q Q'/b|t^n,at Q'/t^nb,ab Q'/pq,pq Q/Q'ab;t;p,q)}
     {\Delta_\blambda(a Q^{\prime 2}/b|t^n,at Q'/t^nb,ab Q'/pq,pq/ab;t;p,q)}\right)\notag\\
&\times\hat{\cR}{}^*_{\bkappa/\blambda}([v_0,\dots,v_{2k-1}];pq Q/b,pq/Q'a;t;p,q),
\end{align}
where $Q=p^lq^m$, $Q'=p^{l'}q^{m'}$.
\end{cor}

If $\blambda=0$ in the first identity, one can apply complementation to
obtain a relation between $\hat{\cR}{}^*_{\bkappa/0}$ and
$\hat{\cR}{}^*_{(l,m)^n-\bkappa/0}$; the constraint on the arguments causes
both of these to be ordinary interpolation functions in $n$ variables, and
this is just the usual complementation symmetry of such functions.  (In
contrast, in the corresponding special case of the corollary, the lifted
interpolation functions are not simply ordinary interpolation functions.)
Particularly interesting is the case that both $\bkappa$ and its complement
are rectangles, since then the identity is a transformation of more
classically hypergeometric sums (under the hypotheses of
Theorem~\ref{thm:1.12}):
\begin{align}
\hat{\cR}{}^*_{(l,m)^n/0}([v_0,\dots,v_{2k-1}]&;a,b;t;p,q)
\notag\\
\propto{}&
\prod_{0\le i<2k}
\frac{\Gampqt(b/v_i,(Q(pqt/b))/v_i,(t^{-n-n'}at)/v_i,(pq/Qt^{-n-n'}a)/v_i)}
     {\Gampqt(bv_i,Q(pqt/b)v_i,t^{-n-n'}atv_i,(pqt/Qt^{-n-n'}at)v_i)}\notag\\
&
\times\hat{\cR}{}^*_{(l,m)^{n'}/0}([v_0,\dots,v_{2k-1}];t^{n+n'-1}b/Q,Q a/t^{n+n'-1};t;p,q),
\label{eq:dualKMsum}
\end{align}
where the constant of proportionality can be determined from the case
$k=1$, when both lifted interpolation functions have explicit evaluations.
This is a sort of dual Karlsson-Minton sum; in particular, the dual of this
sum (coming from the Corollary) is a multivariate analogue of
\cite[Cor.~4.5]{RosengrenH/SchlosserM:2005}. 

As usual with such sums, there is an integral analogue of
\eqref{eq:dualKMsum}.  This was stated as a conjecture in the original
version of this paper, and has since been proved by Van de Bult
\cite{vandeBultFJ:2009}.  It has also appeared in a physical context in
\cite[\S 7]{SpiridonovVP/VartanovGS:2009}.

\begin{thm}\cite[Thm.~3.1]{vandeBultFJ:2009} \label{conj:int_commut}
For integers $m,n,n_0,\dots,n_{k-1}\ge 0$, and parameters $t_0$, $t_1$, $t_2$,
$t_3$, $v_0$,\dots,$v_{2k-1}$ satisfying
\begin{align}
t_0t_1t_2t_3 &= t^{2+m-n}\\
v_{2i}v_{2i+1} &= pq/t^{n_i}\\
\sum_{0\le i<k} n_i &= m+n,
\end{align}
one has
\begin{align}
\II^{(k-1)}_n(t_0,t_1,t_2,t_3,v_0,\dots,v_{2k-1};t;p,q)
={}&
\prod_{m<i\le n} \prod_{0\le r<s<4}\Gampq(t^{n-i}t_rt_s)
\prod_{0\le i<2k}
  \prod_{0\le r<4}
    \frac{\Gampqt(pq t_r/v_i)}
         {\Gampqt(t_r v_i)}\notag\\
&\times\II^{(k-1)}_m(t/t_0,t/t_1,t/t_2,t/t_3,v_0,\dots,v_{2k-1};t;p,q).
\end{align}
\end{thm}

\begin{rem}
  Independently of \cite{vandeBultFJ:2009}, one can see that this holds
  when $k=1$ (both sides can be explicitly evaluated), as well as when
  $k=2$, as a special case of the $E_8$ symmetry of \cite{xforms} (a rare
  case in which a transformation outside the usual double cosets can be
  applied, via a sequence of two dimension-changing transformations).  The
  case $t\mapsto pq/t$, $|m-n|\le 1$ appears naturally if one attempts to
  give a direct proof of the commutation relations for the integral
  operators of \cite{xforms}.  (Note that the case $m=n$ implies the
  general case, as one may take the limit $v_0\to t^{-n_0}t_3$ to reduce
  the dimension on the right-hand side.)
\end{rem}

We will have occasion below to use the corresponding identity for
commutation of difference operators.

\begin{lem}\label{lem:diff_commut}
For any parameters $v_r$ such that $v_0v_1v_2v_3 = p^2q^2$, the
$BC_n$-symmetric function
\[
\sum_{\sigma\in \{\pm 1\}^n}
\prod_{1\le i\le n}
  \frac{\prod_{0\le r<4} \theta_p(v_r z_i^{\sigma_i})}
       {\theta_p(z_i^{2\sigma_i},pqz_i^{2\sigma_i})}
\prod_{1\le i<j\le n}
  \frac{\theta_p(t z_i^{\sigma_i}z_j^{\sigma_j},(pq/t)z_i^{\sigma_i}z_j^{\sigma_j})}
       {\theta_p(  z_i^{\sigma_i}z_j^{\sigma_j},  pq  z_i^{\sigma_i}z_j^{\sigma_j})}
\]
is invariant under $v_r\mapsto pq/v_r$.  In particular, the function
\[
\sum_{\sigma\in \{\pm 1\}^n}
\prod_{1\le i\le n}
  \frac{\theta_p(q p^{1/2} w^{\pm 1} z_i^{\sigma_i})}
       {\theta_p(p^{1/2} w^{\pm 1} z_i^{\sigma_i})\theta_p(z_i^{2\sigma_i},pqz_i^{2\sigma_i})}
\prod_{1\le i<j\le n}
  \frac{\theta_p(t z_i^{\sigma_i}z_j^{\sigma_j},(pq/t)z_i^{\sigma_i}z_j^{\sigma_j})}
       {\theta_p(  z_i^{\sigma_i}z_j^{\sigma_j},  pq  z_i^{\sigma_i}z_j^{\sigma_j})}
\]
is independent of $w$.
\end{lem}

\begin{proof}
As shown in \cite{xforms,bctheta}, the composed difference operator
\[
{\cal D}^{(n)}_q(u_0,t_0,t_1;t,p)
{\cal D}^{(n)}_q(q^{1/2}u_0,q^{1/2}t_0,q^{-1/2}t_2;t,p)
\]
is invariant under swapping $t_1$ and $t_2$, where
\begin{align}
({\cal D}^{(n)}_q(a,b,c;t,p)f)&(z_1,\dots,z_n)\notag\\
:=
\sum_{\sigma\in \{\pm 1\}^n}&
\prod_{1\le i\le n}
\frac{\theta_p(a z_i^{\sigma_i},b z_i^{\sigma_i},c z_i^{\sigma_i},t^{n-1}abc
  z_i^{-\sigma_i})}
     {\theta_p(z_i^{2\sigma_i},t^{n-i}ab,t^{n-i}ac,t^{n-i}cb)}
\prod_{1\le i<j\le n}
 \frac{\theta_p(t z_i^{\sigma_i}z_j^{\sigma_j})}
      {\theta_p(  z_i^{\sigma_i}z_j^{\sigma_j})}
f(\dots,q^{\sigma_i/2}z_i,\dots).
\end{align}
In particular, if we apply the composed operator to a function $f$, we
obtain a linear combination of shifts of $f$, and each coefficient must be
symmetric in $t_1$, $t_2$.  Taking the coefficient of the unshifted term
gives $\prod_{1\le i\le n} \theta_p(u_0 z_i^{\pm 1},t_0 z_i^{\pm 1})$ times
the (general) instance $(v_0,v_1,v_2,v_3) =
(t_1,pq/t_2,qt^{n-1}u_0t_0t_2,p/t^{n-1}u_0t_0t_1)$ of the above sum.
\end{proof}

\begin{rem}
  More generally, if one takes the coefficient of some shift of $f$ in
  which only $m$ variables remain unshifted, then one obtains the
  $n=m$ instance of this sum, apart from some common factors.
  This gives a proof of this transformation and commutation of the
  difference operators without reference to the theory of interpolation
  functions: by induction on $n$, it follows that
\[
{\cal D}^{(n)}_q(u_0,t_0,t_1;t,p)
{\cal D}^{(n)}_q(q^{1/2}u_0,q^{1/2}t_0,q^{-1/2}t_2;t,p)
-
{\cal D}^{(n)}_q(u_0,t_0,t_2;t,p)
{\cal D}^{(n)}_q(q^{1/2}u_0,q^{1/2}t_0,q^{-1/2}t_1;t,p)
\]
acts as a scalar; to show that this scalar vanishes, one need simply apply
it to 1, using the fact \cite[Lem.~6.2]{xforms} that
\[
{\cal D}^{(n)}_q(a,b,c;t,p)1=1.
\]
\end{rem}

\section{Elliptic Cauchy identities}\label{sec:cauchy}

From the results of the previous section, it is clear that the skew
interpolation functions behave very much as analogues of skew Macdonald
polynomials.  This is not entirely surprising, given that skew Macdonald
polynomials are limits of skew interpolation functions, as follows from
Theorem 8.5 of \cite{bctheta}.  More precisely, one has
\begin{align}
\lim_{p\to 0}
p^{|\lambda|/4-|\mu|/4}&
\cR^*_{(0,\lambda)/(0,\mu)}([p^{1/4}/v_0,\dots,p^{1/4}/v_{n-1},p^{-1/4}w_0,\dots,p^{-1/4}w_{n-1}];a,p^{1/2}b;t;p,q)\notag\\
&=
\frac{(-a)^{|\lambda|}q^{n(\lambda')}t^{-2n(\lambda)}C^-_\lambda(t;q,t)}
     {(-aV/W)^{|\mu|}q^{n(\mu')}t^{-2n(\mu)}C^-_\mu(t;q,t)}
P_{\lambda/\mu}([\frac{v_0^k+\dots+v_{n-1}^k-w_0^k-\dots-w_{n-1}^k}{1-t^k}];q,t),
\end{align}
by a straightforward induction from the case $n=1$, when it reduces to
\cite[Thm.~8.5]{bctheta}.  (Here
\[
C^-_\lambda(x;q,t) := \cC^-_{0,\lambda}(x;t;0,q)
=
\prod_{(i,j)\in\lambda} (1-q^{\lambda_i-j}t^{\lambda'_j-i}x)
\]
is the usual hook-product symbol that appears in Macdonald polynomial
theory (e.g., in the denominator of \cite[(VI.6.11')]{MacdonaldIG:1995}, or
in both numerator and denominator in \cite[(VI.6.19)]{MacdonaldIG:1995}),
and the argument to $P_{\lambda/\mu}$ denotes the image under a
homomorphism taking $p_k$ to the stated value.  As above, $V$ denotes the
product $v_0\cdots v_{n-1}$, and similarly for $W$.)

However, if we attempt to give a direct analogue of the Cauchy identity for
skew Macdonald polynomials, we encounter the difficulty that sums of
infinitely many elliptic terms rarely converge.  It will thus be important
to understand under what circumstances a skew interpolation function is
forced to vanish.

\begin{lem}\label{lem:skew_vanish1}
If $\blambda$, $\bkappa$ are partition pairs, $l,m,n$ nonnegative integers,
and $a$, $b$, and $v_0\in \C^*$ are generic, then
\[
\cR^*_{\blambda/\bkappa}([v_0,t^n/p^lq^m v_0];a,b;t;p,q)
\]
vanishes unless
\[
\bkappa_i\le \blambda_i\le \bkappa_{i-n}+(l,m)
\]
for all $i$, with the convention $\bkappa_0=\bkappa_{-1}=\cdots=\infty$.
\end{lem}

\begin{proof}
Observe that we can write
\[
\cR^*_{\blambda/\bkappa}([v_0,t^n/p^lq^m v_0];a,b;t;p,q)
=
\sum_{\bmu}
\cR^*_{\blambda/\bmu}([v_0,t^n/v_0];a,b;t;p,q)
\cR^*_{\bmu/\bkappa}([v_0/t^n,t^n/p^l q^m v_0];a/t^n,b;t;p,q),
\]
with
\[
\cR^*_{\blambda/\bmu}([v_0,t^n/v_0];a,b;t;p,q)
=
\frac{\Delta^0_{\blambda}(a/b|a/v_0,av_0/t^n;t;p,q)}{\Delta^0_{\bmu}(a/bt^n|a/v_0,av_0/t^n;t;p,q)}
\obinomE{\blambda}{\bmu}_{[ab,t^n];t;p,q},
\]
and
\[
\cR^*_{\bmu/\bkappa}([v_0/t^n,t^n/p^l q^m v_0];a/t^n,b;t;p,q)
=
\frac{\Delta^0_{\bmu}(a/bt^n|a/v_0,av_0p^lq^m/t^{2n};t;p,q)}
     {\Delta^0_{\bkappa}(ap^lq^m/bt^n|a/v_0,av_0p^lq^m/t^{2n};t;p,q)}
\obinomE{\bmu}{\bkappa}_{[a/bt^n,p^{-l}q^{-m}];t;p,q}.
\]
The binomial coefficients vanish unless \cite[Cor.~4.5]{bctheta}
\[
\bmu_i\le \blambda_i\le \bmu_{i-n}
\]
and \cite[Cor.~4.2]{bctheta}
\[
\bkappa_i\le \bmu_i\le \bkappa_i+(l,m),
\]
and (by genericity), this vanishing cannot be cancelled by a pole of the
remaining factors.
\end{proof}

The other significant source of vanishing is the following.

\begin{lem}
If $\blambda$, $\bkappa$ are partition pairs, $l,m,n$ are nonnegative
integers, $a$, $b$, and $v_0\in \C^*$ are generic, and $\bkappa_{n+1}\le
(l,m)$, then
\[
\cR^*_{\blambda/\bkappa}([v_0,a/p^{-l}q^{-m}t^n];a,b;t;p,q)
\]
vanishes unless $\blambda_{n+1}\le (l,m)$.
\end{lem}

\begin{proof}
We have
\[
\cR^*_{\blambda/\bkappa}([v_0,a/p^{-l}q^{-m}t^n];a,b;t;p,q)
=
\frac{\Delta^0_{\blambda}(a/b|p^{-l}q^{-m}t^n;t;p,q)}
{\Delta^0_{\bkappa}(p^{-l}q^{-m}t^n/bv_0|p^{-l}q^{-m}t^n;t;p,q)}
\obinomE{\blambda}{\bkappa}_{[a/b,av_0/p^{-l}q^{-m}t^n](v_0);t;p,q}.
\]
The binomial coefficient factor is generic, so cannot contribute any poles,
as are the factors coming from denominators of $\Delta^0$.  We are thus
left with considering the ratio
\[
\frac{\cC^0_{\blambda}(p^{-l}q^{-m}t^n;t;p,q)}
     {\cC^0_{\bkappa}(p^{-l}q^{-m}t^n;t;p,q)}.
\]
If $\bkappa_{n+1}\le (l,m)$, then the denominator is nonzero; the numerator
vanishes unless $\blambda_{n+1}\le (l,m)$.
\end{proof}

Both lemmas extend by induction to vanishing conditions on more general skew
interpolation functions.

\begin{thm}\label{thm:termination1}
Suppose the sequence $v_0,v_1,\dots,v_{2k-1}$ can be ordered in such a way
that $v_{2i}v_{2i+1} = t^{n_i}p^{-l_i}q^{-m_i}$ with $l_i,m_i,n_i\ge 0$,
for $0\le i<k$, and are otherwise generic.  Then for any partition pair
$\bkappa$,
\[
\cR^*_{\blambda/\bkappa}([v_0,v_1,\dots,v_{2k-1}];a,b;t;p,q)=0
\]
unless
\[
\bkappa_i\le \blambda_i\le \bkappa_{i-N}+(L,M),
\]
where $L=\sum_i l_i$, $M=\sum_i m_i$, $N=\sum_i n_i$.
\end{thm}

\begin{proof}
If $k=1$ or $k=0$, this follows from Lemma~\ref{lem:skew_vanish1}.
In general, we have
\[
\cR^*_{\blambda/\bkappa}([v_0,v_1,\dots,v_{2k-1}];a,b;t;p,q)
=
\sum_{\bmu}
\cR^*_{\blambda/\bmu}([v_0,v_1];a,b;t;p,q)
\cR^*_{\bmu/\bkappa}([v_2,\dots,v_{2k-1}];a/v_0v_1,b;t;p,q);
\]
the term associated to $\bmu$ vanishes unless
\[
\bmu_i\le \blambda_i\le \bmu_{i-n_0}+(l_0,m_0)
\]
and (by induction)
\[
\bkappa_i\le \bmu_i\le \bkappa_{i-(N-n_0)}+(L-l_0,M-m_0).
\]
The claim follows.
\end{proof}

Similarly, one has the following.

\begin{thm}\label{thm:termination2}
Let $l_0,\dots,l_{k-1}$, $m_0,\dots,m_{k-1}$, $n_0,\dots,n_{k-1}$ be
sequences of nonnegative integers, and suppose the otherwise generic
sequence $v_0,\dots,v_{2k-1}$ can be ordered in such a way that
$v_{2i}v_{2i+1}=t^{n_i}p^{-l_i}q^{-m_i}$ for $1\le i<k$, while
\[
a/\prod_{1\le i<2k} v_i = t^{n_0}p^{-l_0}q^{-m_0}.
\]
If $\bkappa_{n_0+1}\le (l_0,m_0)$, then
\[
\cR^*_{\blambda/\bkappa}([v_0,v_1,\dots,v_{2k-1}];a,b;t;p,q)=0
\]
unless $\blambda_{N+1}\le (L,M)$.
\end{thm}

When $\bkappa=0$, the two vanishing conditions coincide, and both simply
state that $\blambda_{N+1}\le (L,M)$.  This corresponds to the extra
symmetry explained in Proposition~\ref{prop:symmA2n} above.  One also
obtains an additional (albeit more delicate) source of vanishing in the
$\bkappa=0$ case.

\begin{thm}
Let $l\le L$; $m\le M$; $n\le N$ be nonnegative integers.  Then the lifted
interpolation function
\[
\cR^*_{\blambda/0}([p^Lq^Ma/t^N,v_1,\dots,v_{2k-1}];a,pqt^n/p^lq^ma;t;p,q)
\]
vanishes unless $\blambda_{N+1}\le (L,M)$.
\end{thm}

\begin{proof}
We have the expansion
\begin{align}
\cR^*_{\blambda/0}&([p^Lq^Ma/t^N,v_1,\dots,v_{2n-1}];a,pqt^n/p^lq^ma;t;p,q)\notag\\
&=
\sum_{\bmu}
\obinomE{\blambda}{\bmu}_{[a/b,t^n/p^lq^m];t;p,q}
\Delta^0_{\bmu}(pq/b^2|t^{N-n}/p^{L-l}q^{M-m},pq/bv_1,\dots,pq/bv_{2n-1},pq\prod_{0\le
  r<2n} v_r/ab;t;p,q),
\end{align}
where $b=pqt^n/p^lq^ma$.  The second factor vanishes unless
\[
\bmu_{N-n+1}\le (L-l,M-m),
\]
while the first factor vanishes unless
\[
\bmu_i\le \blambda_i\le \bmu_{i-n}+(l,m).
\]
\end{proof}


Since infinite sums of elliptic functions tend not to converge, we need to
insist in the elliptic Cauchy identity that the sum terminate; i.e.,
involve only finitely many terms.  To avoid potential obstructions to
analytic continuation arguments, we insist that the termination occurs
either because the partition pair being summed over occurs as the lower
partition in a skew interpolation function (or elliptic binomial
coefficient), or because using either $\Delta^0$ factors of the summand or
one of the first two vanishing theorems (Theorem~\ref{thm:termination1} or
\ref{thm:termination2}), one can bound both the first part and the length of
the partition pair.  In the latter case, we will refer to the source of the
bound on the first part as a horizontal termination condition (as it bounds
the horizontal extent of the corresponding diagram); similarly a vertical
termination condition is one that allows us to bound the length.  Note that
a sum over skew diagrams which are unions of finitely many horizontal
strips is vertically terminated, and vice versa.

With this in mind, we can now state our first version of the Cauchy
identity for skew interpolation functions.  Note that the termination
conditions allow the right-hand side to be simplified to an expression in
$p$-theta and $q$-theta functions; this would not hold if the sum were
finite by virtue of the third vanishing condition alone.

\begin{thm}
One has the identity
\begin{align}
\sum_{\bmu}
\Delta_{\bmu}(a/b|;t;p,q)
\cR^*_{\bmu/0}([v_0,\dots,v_{2n-1}];a,b;t;p,q)
\cR^*_{\bmu/0}([w_0,\dots,w_{2m-1}];\sqrt{pqt}/b,\sqrt{pqt}/a;t;p,q)\qquad&\notag\\
{}=
\prod_{\substack{0\le i<2n+2\\0\le j<2m+2}}
\frac{\Gampqt((pqt)^{1/2} v_i/w_j)}{\Gampqt((pqt)^{1/2} v_i w_j)},
\end{align}
where
\[
v_{2n}=a/\prod_{0\le r<2n} v_r,\quad v_{2n+1}=1/a,\quad
w_{2m}=(pqt)^{1/2}/b\prod_{0\le r<2m} w_r,\quad w_{2m+1}=b/(pqt)^{1/2},
\]
and the parameters are such that the sum terminates, but otherwise generic.
\end{thm}

\begin{proof}
Suppose first that the vertical termination of the sum is due to the $v$
parameters (i.e., the first skew interpolation function satisfies the
hypotheses of Theorem~\ref{thm:termination1} with $L=M=0$ (or
Theorem~\ref{thm:termination2}, but this is essentially equivalent in the
case of lifted interpolation functions)), while the horizontal termination
is due to the $w$ parameters (the second skew interpolation function
satisfies the hypotheses of Theorem~\ref{thm:termination1} with $N=0$).  We
may thus assume (adding or removing pairs $x,1/x$ as necessary) that
$v_{2i}v_{2i+1}=t$, $0\le i<n$, while $w_{2i}w_{2i+1}=1/p$ or $1/q$ for
each $0\le i<m$.  In that case, we may factor the sum into the product of a
$q$-elliptic sum and a $p$-elliptic sum.  Applying duality to the $w$
factor allows us to express both factors as interpolation functions, and
the claim becomes the Cauchy identity of \cite[Thm.~4.18]{bctheta}.

The other possibility (up to obvious symmetries) is that one set of
parameters (say the $w$ parameters) provides both termination conditions.
If the $v$ parameters {\em also} provide vertical termination, then the
result follows; in general, the set of $v$ parameters for which
$v_{2i}v_{2i+1}\in t^\N$, $0\le i<n$, is Zariski dense on both elliptic
curves, so we may analytically continue.
\end{proof}

There is also a skew version of the above identity.

\begin{thm}\label{thm:cauchy_skew}
One has the identity
\begin{align}
&\sum_{\bmu}
\frac{\Delta_{\bmu}(a/b|;t;p,q)}{\Delta_\blambda(a/bV|;t;p,q)}
\cR^{*}_{\bmu/\blambda}([v_0,\dots,v_{2n-1}];a,b;t;p,q)
\cR^{*}_{\bmu/\bkappa}([w_0,\dots,w_{2m-1}];\sqrt{pqt}/b,\sqrt{pqt}/a;t;p,q)
\notag\\
\propto&
\sum_{\bmu}
\cR^{*}_{\blambda/\bmu}([w_0,\dots,w_{2m-1}];\sqrt{pqt}/b,\sqrt{pqt}V/a;t;p,q)
\frac{\Delta_\bkappa(a/bW|;t;p,q)}
{\Delta_\bmu(a/bVW|;t;p,q)}
\cR^{*}_{\bkappa/\bmu}([v_0,\dots,v_{2n-1}];a,bW;t;p,q),
\end{align}
where $V=\prod_r v_r$, $W=\prod_r w_r$, for generic parameters such that
the left-hand side terminates.  The constant of proportionality is
independent of $\blambda$ and $\bkappa$, and is thus equal to the value of
the sum when $\blambda=\bkappa=0$.
\end{thm}

\begin{proof}
  First consider the case $\bkappa=0$, so that only the term with $\bmu=0$
  survives on the right-hand side, and suppose furthermore that
  $v_{2i-1}v_{2i}=t$, and $w_{2i-1}w_{2i}\in p^{-\N}q^{-\N}$ for each $i$.
  If we multiply both sides by
\[
\Delta_\blambda(a/bV|;t;p,q)
\cR^*_{\blambda/0}([t^{1/2} u_1^{\pm
    1},\dots,t^{1/2}u_{\ell(\blambda)}^{\pm 1}];a/V,b;t;p,q)
\]
and sum over $\blambda$, the right-hand side becomes an instance of the
previous theorem, while the left-hand side simplifies directly to an
instance of the previous theorem.  In particular, after so multiplying and
summing, the two sides agree.  But the test functions we have multiplied by
are linearly independent, and thus both sides agree before summing.

The arbitrary terminating case with $\bkappa=0$ then follows by analytic
continuation.  Similarly, the case $\bkappa\ne 0$ follows from the case
$\bkappa=0$, and the general claim follows by analytic continuation.
\end{proof}

Another approach to proving the above identity is by induction on $n$ and
$m$; it suffices to consider the case $n=m=1$, or in other words the
following special case.

\begin{cor}\label{cor:ipR}
One has the identity
\begin{align}
\sum_{\bmu}
\frac{\Delta_{\bmu}(a|v_0,v_1,v_2,v_3;t;p,q)}
     {\Delta_{\blambda}(a/b_0|v_0,v_1,v_2,v_3;t;p,q)}
&
\obinomE{\bmu}{\blambda}_{[a,b_0];t;p,q}
\obinomE{\bmu}{\bkappa}_{[a,b_1];t;p,q}
\notag\\
{}\propto
\sum_{\bmu}
&\obinomE{\blambda}{\bmu}_{[a/b_0,b_1];t;p,q}
\frac{\Delta_{\bkappa}(a/b_1|v_0,v_1,v_2,v_3;t;p,q)}
     {\Delta_{\bmu}(a/b_0b_1|v_0,v_1,v_2,v_3;t;p,q)}
\obinomE{\bkappa}{\bmu}_{[a/b_1,b_0];t;p,q},
\end{align}
assuming the termination conditions
\begin{align}
t^\N&\cap \{b_0,b_1,v_0,v_1,v_2,v_3\}\ne
\emptyset\\
p^{-\N}q^{-\N}&\cap
\{b_0,b_1,v_0,v_1,v_2,v_3\}\ne
\emptyset,
\end{align}
(with corresponding conditions on $\blambda$, $\bkappa$ if the $v_r$ are
used for termination), and the balancing condition $b_0b_1v_0v_1v_2v_3=pqta^2$.
The constant of proportionality is given by
\begin{align}
\sum_{\bmu}
\Delta_{\bmu}(a|b_0,b_1,v_0,v_1,v_2,v_3;t;p,q).
\end{align}
\end{cor}

This, in turn, gives an expression for the (discrete) inner product of two
interpolation functions with respect to the density of
\cite[Thm.~5.8]{bctheta}; in particular, it is a limiting case of an
integral identity \cite[Thm.~9.4]{xforms}.  It can also be obtained
(following ideas of \cite{RosengrenH:2007}) by computing connection
coefficients for interpolation theta functions in two different ways
(compare \cite[Thm.~4.15]{bctheta}).

%

\section{Elliptic Littlewood identities}\label{sec:litt}

Since the classical Littlewood identity only involves a single Schur
function, the termination conditions in any direct elliptic analogue must
be borne by a single skew interpolation function.  It turns out, however,
that the conditions can be weakened slightly; it is permissible for the
skew interpolation function to allow unbounded upper partitions, so long as
none of those satisfy the even multiplicity condition.  The point is that
since the first and second parts of $\bmu^2$ agree, we need only have a
bound on the {\em second} part of $\bmu^2$ to obtain a terminating sum.
Thus in the horizontal termination condition, we may allow one of the pairs
to multiply to $t p^{-l_i}q^{-m_i}$ instead of $p^{-l_i}q^{-m_i}$.

In particular, if $v_0v_1=t$, then this simultaneously gives both
horizontal and vertical termination conditions.  Indeed, we find that
if $\cR^*_{\bmu^2/\blambda}([v_0,t/v_0];a,b;t;p,q)\ne 0$, then
\begin{align}
\blambda_{2i-1}\le &(\bmu^2)_{2i-1}\le \blambda_{2i-2}\\
\blambda_{2i}\le &(\bmu^2)_{2i}\le \blambda_{2i-1},
\end{align}
and thus, since $(\bmu^2)_{2i-1}=(\bmu^2)_{2i}=\bmu_i$,
\[
\blambda_{2i-1}\le \bmu_i\le \blambda_{2i-1},
\]
so that $\bmu$ is uniquely determined.  With this in mind, define new
operations on partition pairs
\begin{align}
\blambda^+&: \blambda^+_i = \blambda_{2i-1}\\
\blambda^-&: \blambda^-_i = \blambda_{2i-2}.
\end{align}

\begin{lem}\label{lem:litt_t}
For any partition pair $\blambda$,
\begin{align}
\sum_{\bmu}
\obinomE{\bmu^2}{\blambda}_{[a,t];t;p,q}
&
\Delta_{\bmu}(a/t|v_0,\dots,v_{2n-1};t^2;p,q)
\notag\\
&{}=
\Delta_{\blambda}(a/t|v_0,\dots,v_{2n-1};t;p,q)
\sum_{\bmu}
\obinomE{\blambda}{\bmu^2}_{[a/t,t];t;p,q}
\frac{\Delta_{\bmu}(a/t^3|v_0,\dots,v_{2n-1};t^2;p,q)}
     {\Delta_{\bmu^2}(a/t^2|v_0,\dots,v_{2n-1};t;p,q)}
\end{align}
\end{lem}

\begin{proof}
On the left-hand side, only the term with $\bmu=\blambda^+$ contributes,
while on the right-hand side, only the term with $\bmu=\blambda^-$
contributes.  Now, it follows easily from the definition of $\Delta^0$ that
\begin{align}
\Delta^0_{\blambda}(a|v_0;t;p,q)
&=
\Delta^0_{\blambda^+}(a|v_0;t^2;p,q)
\Delta^0_{\blambda^-}(a/t^2|v_0/t;t^2;p,q),\\
\Delta^0_{\bmu^2}(a|v_0;t;p,q)
&=
\Delta^0_{\bmu}(a/t|v_0,v_0/t;t^2;p,q),
\end{align}
and thus the dependence on $v_r$ disappears.  It thus suffices to
show that
\[
\obinomE{\blambda^{+2}}{\blambda}_{[a,t];t;p,q}
=
\obinomE{\blambda}{\blambda^{-2}}_{[a/t,t];t;p,q}
\frac{\Delta_{\blambda}(a/t|;t;p,q)\Delta_{\blambda^-}(a/t^3|;t^2;p,q)}
     {\Delta_{\blambda^{-2}}(a/t^2|;t;p,q)\Delta_{\blambda^+}(a/t|;t^2;p,q)}.
\]
This can be proved by induction on $\ell(\blambda)$ via the observations
\begin{align}
((l,m)\cdot \blambda)^+ &= (l,m)\cdot \blambda^-\\
((l,m)\cdot \blambda)^- &= \blambda^+
\end{align}
and the relation \cite[Cor.~4.8]{bctheta}
\[
\obinomE{(l,m)\cdot\blambda}{\bmu}_{[a,t];t;p,q}
=
\Delta^0_{(l,m)}(a|t;t;p,q)
\frac{\Delta_\bmu(a/t|p^{-l}q^{-m},p^lq^m a;t;p,q)}
     {\Delta_\blambda(a/t^2|p^{-l}q^{-m},p^lq^m a;t;p,q)}
\obinomE{\bmu}{\blambda}_{[a/t,t];t;p,q}.
\]
Note that we may freely check the $p$-theta and $q$-theta portions of the
relation separately, and rescale so that both are elliptic.
\end{proof}

The first version of an elliptic Littlewood identity is the following.

\begin{thm}
We have
\begin{align}
\sum_{\bmu}
\cR^*_{\bmu^2/0}([v_0,\dots,v_{2n-1}];ta,(pqt)^{1/2}/a;t;p,q)
\Delta_{\bmu}(a^2/(pqt)^{1/2}|;t^2;p,q)\qquad\qquad&
\notag\\
=
\frac{\Gampqt((pqt)^{1/2})^n \Gampqt((pqt)^{1/2} t)
\prod_{0\le i<j<2n+2} \Gampqt((pqt)^{1/2}v_i/v_j)}
{\prod_{0\le i<2n+2} \Gampqtt((pqt)^{1/2}v_i^2)\prod_{0\le i<j<2n+2}\Gampqt((pqt)^{1/2}v_iv_j)},&
\end{align}
where $v_{2n} = ta/\prod_{0\le i<2n} v_i$, $v_{2n+1} = 1/a$, and the sum
terminates.
\end{thm}

\begin{proof}
Using the $S_{2n+1}$ symmetry of the lifted interpolation functions, we
may assume (inserting $x,1/x$ pairs as necessary) that the parameters
pairwise multiply to $t$, and are ordered in such a way that
$v_{2m},\dots,v_{2n-1}$ gives both horizontal and vertical termination
conditions for $0\le m<n$.  The proof then follows by a straightforward
induction on $n$:
\begin{align}
\sum_{\bmu}
\cR^*_{\bmu^2/0}&([v_0,\dots,v_{2n-1}];ta,b;t;p,q)
\Delta_{\bmu}(a/b|;t^2;p,q)\notag\\
=
\sum_{\blambda}&
\cR^*_{\blambda/0}([v_2,\dots,v_{2n-1}];a,b;t;p,q)
\sum_{\bmu}
\cR^*_{\bmu^2/\blambda}([v_0,t/v_0];ta,b;t;p,q)
\Delta_{\bmu}(a/b|;t^2;p,q)\\
=
\sum_{\bmu}&
\frac{\Delta_{\bmu}(a/bt^2|;t^2;p,q)}
     {\Delta_{\bmu^2}(a/bt|;t;p,q)}\notag\\
&\times\sum_{\blambda}
\Delta_{\blambda}(a/b|;t;p,q)
\cR^*_{\blambda/0}([v_2,\dots,v_{2n-1}];a,b;t;p,q)
\cR^*_{\blambda/\bmu^2}([v_0,t/v_0];a,b;t;p,q)
\\
\propto
\sum_{\bmu}&
\Delta_{\bmu}(a/bt^2|;t^2;p,q)
\cR^{*}_{\bmu^2/0}([v_2,\dots,v_{2n-1}];a,bt;t;p,q).
\end{align}
Note that the last step only works if $ab=(pqt)^{1/2}$.
\end{proof}

\begin{rem}
Note that the termination condition prevents one from obtaining a Macdonald
polynomial identity as a simple limit, except in the case $n=1$.  However,
if one ignores the issue of termination, and takes a limit above, one obtains
\[
\begin{split}
\sum_\mu
\frac{C^-_\mu(t;q,t^2)}
     {C^-_\mu(q;q,t^2)}
P_{\mu^2}(\bigl[\frac{v_0^k+\dots+v_{n-1}^k-w_0^k-\dots-w_{n-1}^k}{1-t^k}\bigr];q,t)\qquad\qquad\qquad\qquad\\
{}=
\frac{\prod_{0\le i,j<n} (v_iw_j;q,t)}
{
\prod_{0\le i<n}(w_i^2,tv_i^2;q,t^2)
\prod_{0\le i<j<n}(v_iv_j,w_iw_j;q,t)
},
\end{split}
\]
agreeing with Macdonald's $q,t$-Littlewood identity.  This agreement
results from the fact that the $n=1$ case and the Cauchy identity together
suffice to make the above induction work in the absence of termination.
\end{rem}

If the lifted interpolation function is terminating in the usual sense
(i.e., without taking advantage of the one extra factor of $t$), then it in
fact corresponds to an ordinary interpolation function evaluated at a
partition.  This gives rise to the following curious identity.

\begin{cor}
For every partition pair $\blambda$, one has the following identity of
meromorphic functions
\[
\sum_{\bmu}
\frac{\Delta_{\bmu}(a/(pqt)^{1/2}t|;t^2;p,q)}
     {\Delta_{\bmu^2}(a/(pqt)^{1/2}|;t;p,q)}
\obinomE{\blambda}{\bmu^2}_{[a,(pqt)^{1/2}];t;p,q}
=
\frac{\cC^-_{\blambda}((pqt)^{1/2};t;p,q)
      \cC^+_{\blambda}((pqt)^{1/2}a/t;t;p,q)}
     {\cC^0_{2\blambda}((pqt)^{1/2}a/t;t^2;p,q)}.
\]
\end{cor}

Following the argument of Theorem \ref{thm:cauchy_skew}, one has the
following skew analogue of the Littlewood identity.

\begin{thm}\label{thm:litt_skew}
The following identity holds:
\begin{align}
\sum_{\bmu}&
\cR^*_{\bmu^2/\blambda}([v_0,\dots,v_{2n-1}];ta,(pqt)^{1/2}/a;t;p,q)
\frac{\Delta_{\bmu}(a^2/(pqt)^{1/2}|;t^2;p,q)}
{\Delta_{\blambda}(a^2 t/(pqt)^{1/2}V|;t;p,q)}
\notag\\
&\propto
\sum_{\bmu}
\cR^*_{\blambda/\bmu^2}([v_0,\dots,v_{2n-1}];a,(pqt)^{1/2}V/ta;t;p,q)
\frac{\Delta_{\bmu}(a^2/(pqt)^{1/2}V^2|;t^2;p,q)}
{\Delta_{\bmu^2}(a^2 t/(pqt)^{1/2}V^2|;t;p,q)},
\end{align}
assuming the LHS terminates; the constant of proportionality is independent
of $\blambda$, and can be obtained by setting $\blambda=0$.
\end{thm}

\begin{proof}
One can again proceed by induction on $n$; for $n>1$, a terminating case
always has a pair multiplying to $t$ (possibly after adding a pair
multiplying to $1$) such that the various sums continue to terminate after
extracting that pair.  One thus reduces to the case $n=1$; if $v_0v_1=t$,
this has already been shown, while in general it follows from
Theorem~\ref{thm:genlitt} below.
\end{proof}

\begin{rem}
Again, this formally produces Macdonald's skew $q,t$-Littlewood identity in
the limit.
\end{rem}

One disappointing aspect of the above identities is the fact that $ab$ (or,
in the case of binomial coefficients, $b$) is constrained.  It appears that
this is a necessary constraint if we wish a completely general Littlewood
identity, but if we are willing to restrict our attention to binomial
coefficients, we can introduce more parameters.

\begin{thm}\label{thm:genlitt}
If $b^2v_0v_1v_2v_3=pqta^2$, and the LHS terminates, then
\[
\sum_{\bmu}
\obinomE{\bmu^2}{\blambda}_{[a,b];t;p,q}
\frac{\Delta_{\bmu}(a/t|v_0,v_1,v_2,v_3;t^2;p,q)}
     {\Delta_{\blambda}(a/b|v_0,v_1,v_2,v_3;t;p,q)}
\propto
\sum_{\bmu}
\obinomE{\blambda}{\bmu^2}_{[a/b,b];t;p,q}
\frac{\Delta_{\bmu}(a/tb^2|v_0,v_1,v_2,v_3;t^2;p,q)}
     {\Delta_{\bmu^2}(a/b^2|v_0,v_1,v_2,v_3;t;p,q)}
\]
where the constant of proportionality is independent of $\blambda$.
The termination condition on the LHS is that
\begin{align}
t^{2\N}\cap \{v_0,v_1,v_2,v_3,b,b/t\}&\ne \emptyset\\
p^{-\N}q^{-\N}\cap \{v_0,v_1,v_2,v_3,b,b/t\}&\ne \emptyset,
\end{align}
with corresponding constraints on $\blambda$ if a $v_r$ is used for
termination.
\end{thm}

\begin{proof}
If we write
\begin{align}
\Delta^0_{\bmu}(a/t|v_3;t^2;p,q)
\obinomE{\bmu^2}{\blambda}_{[a,b];t;p,q}
={}&
\Delta^0_{\bmu}(a/t|bv_3/t,bv_3/t^2,pqa/v_3;t^2;p,q)\notag\\
&\times\sum_{\bnu}
\obinomE{\bmu^2}{\bnu}_{[a,t];t;p,q}
\obinomE{\bnu}{\blambda}_{[a/t,b/t](v_3/t,pqa/bv_3);t;p,q}
\end{align}
(where $v_3$ is not used for termination), and apply
Lemma~\ref{lem:litt_t}, we reduce to the case with
\[
(a,b,v_3)\mapsto (a/t^2,b/t,v_3/t);
\]
thus by induction (the claim being trivial when $b=1$), we obtain every
case with $b\in t^\N$, and the general result by analytic continuation.
\end{proof}

\begin{rem}
When $\blambda=0$, the right-hand sum becomes 1, while the left-hand side
becomes
\[
\sum_{\bmu} \Delta_{\bmu}(a/t|v_0,v_1,v_2,v_3,b,b/t;t^2;p,q),
\]
which can be evaluated, thus determining the normalization.
\end{rem}

\begin{cor}\label{cor:diff_eq_t}
  If $v_0v_1v_2v_3=p^{2l+2}q^{2m+2}a^2$ and $\ell(\blambda)\le 2n$ with
  $l,m,n\in \N$, then
\[
\sum_{\bmu}
\obinomE{\bmu^2}{\blambda}_{[a,p^{-l}q^{-m}];t;p,q}
\Delta_\bmu(a/t|t^{2n},a/t^{2n};t^2;p,q)
\prod_{0\le r<4}
\frac{\Delta^0_\bmu(a/t|v_r;t^2;p,q)}
{\cC^0_{(l,m)^n}(pqt^{2n-1}/v_r;t^2;p,q) \cC^0_\blambda(v_r;t;p,q)}
\]
is invariant under $v_r\mapsto p^{l+1}q^{m+1}a/v_r$.
\end{cor}

\begin{proof}
  When $v_0=pqt^{2n-1}$ and $\blambda_{2n}\ge (l,m)$, the left-hand side
  simplifies to the case $v_0=t^{2n}$, $b=p^{-l}q^{-m}$ of the left-hand
  side of Theorem~\ref{thm:genlitt}, while the right-hand sides become
  equivalent after applying Corollary 4.8 of \cite{bctheta}.  Substituting
  $\blambda\mapsto (l',m')^{2n}+\blambda$ for $l'\ge l$, $m'\ge m$, then
  shifting the variable of summation gives the case
  $v_0=p^{l'+1}q^{m'+1}t^{2n-1}$ of the corollary.  Since these cases are
  Zariski dense, the corollary follows.
\end{proof}

We observed above that Corollary~\ref{cor:ipR} can be interpreted as
giving the inner product of two interpolation functions, and is in
particular a special case of a more general integral identity.  The same
applies to Theorem~\ref{thm:genlitt}.  The basic observation is that the
sequence of points
\[
t^{2n-i} (p,q)^{(\bmu^2)_i} a, \quad 1\le i\le 2,
\]
which arises when evaluating an interpolation function at $\bmu^2$, can also
be expressed in the form
\[
t^{\pm 1/2} (t^2)^{n-i} (p,q)^{\bmu_i} t^{1/2} a, \quad 1\le i\le n.
\]
This gives rise to the following result, where we recall that
\[
\langle f(z_1,\dots,z_n)\rangle^{(n)}_{t_0,t_1,t_2,t_3,u_0,u_1;t;p,q}
\]
is the normalized linear functional associated to the $n$-dimensional
elliptic Selberg integral.

\begin{thm}\label{thm:int_litt}
For any partition pair $\blambda$, and generic parameters satisfying the
balancing condition
\[
t^{4n-2} t_0t_1t_2t_3u_0^2 = pq,
\]
one has
\begin{align}
&\langle
\cR^{*(2n)}_\blambda(\dots,t^{\pm 1/2}z_i,\dots;t_0,u_0;t;p,q)
\rangle^{(n)}_{t^{1/2}t_0,t^{1/2}t_1,t^{1/2}t_2,t^{1/2}t_3,t^{\pm
    1/2}u_0;t^2;p,q}\notag\\
&\qquad =
\Delta^0_\blambda(t^{2n-1}t_0/u_0|t^{2n-1}t_0t_1,t^{2n-1}t_0t_2,t^{2n-1}t_0t_3;t;p,q)\\
&\phantom{\qquad =\quad}
\times\sum_{\bmu}
\obinomE{\blambda}{\bmu^2}_{[t^{2n-1}t_0/u_0,t^{2n-1}t_0u_0];t;p,q}
\frac{\Delta_{\bmu}(1/tu_0^2|t^{2n},t^{2n-1}t_0t_1,t^{2n-1}t_0t_2,t^{2n-1}t_0t_3;t^2;p,q)}
     {\Delta_{\bmu^2}(1/u_0^2|t^{2n},t^{2n-1}t_0t_1,t^{2n-1}t_0t_2,t^{2n-1}t_0t_3;t;p,q)}\notag
\end{align}
\end{thm}

\begin{proof}
If $t^{2n-1}t_0t_1 = p^{-l}q^{-m}$, so that the integral reduces to a sum,
the identity is a case of Theorem~\ref{thm:genlitt}.  But the left-hand
side can be computed by expanding
\[
\cR^{*(2n)}_\blambda(\dots,t^{\pm 1/2}z_i,\dots;t_0,u_0;t;p,q)
\]
as a linear combination of products of a function
\[
\tcR^{(n)}_\bmu(\dots,z_i,\dots;t^{1/2}t_0{:}t^{1/2}t_1,t^{1/2}t_2,t^{1/2}t_3;t^{1/2}u_0,t^{-1/2}u_0;t^2;p,q)
\]
and a function
\[
\tcR^{(n)}_\bnu(\dots,z_i,\dots;t^{1/2}t_0{:}t^{1/2}t_1,t^{1/2}t_2,t^{1/2}t_3;t^{-1/2}u_0,t^{1/2}u_0;t^2;p,q).
\]
(This is not to say that this expansion can be done explicitly; it suffices
that such an expansion exists, which follows from the fact that all allowed
poles are covered.)  In particular, it follows that the left-hand side is
the product of $p$- and $q$-theta functions, as is the right-hand side, so
we may analytically continue to obtain the desired result.
\end{proof}

The following special case has a particularly simple summand.

\begin{cor}
If $t^{4n-2} t_0^2u_0^2v_0v_1 = pq$, then
\begin{align}
&\left\langle
\frac{\cR^{*(2n)}_\blambda(\dots,t^{\pm 1/2}z_i,\dots;t^{1/2}t_0,t^{1/2}u_0;t;p,q)}
{\Delta^0_\blambda(t^{2n-1}t_0/u_0|t^{2n-1}t_0v_0,t^{2n-1}t_0v_1;t;p,q)}
\right\rangle^{(n)}_{t_0,tt_0,u_0,tu_0,v_0,v_1;t^2;p,q}\notag\\
&=
\Delta^0_\blambda(t^{2n-1}t_0/u_0|t^{2n-1}t_0^2;t;p,q)
\sum_{\bmu}
\obinomE{\blambda}{\bmu^2}_{[t^{2n-1}t_0/u_0,t^{2n}t_0u_0];t;p,q}
\frac{\Delta_{\bmu}(1/t^2u_0^2|t^{2n},t^{2n-1}t_0^2;t^2;p,q)}
     {\Delta_{\bmu^2}(1/tu_0^2|t^{2n},t^{2n-1}t_0^2;t;p,q)}
\end{align}
\end{cor}

If we multiply both sides of Theorem~\ref{thm:int_litt} by
\[
\Delta^0_\blambda(t^{2n-1}t_0/u_0|t^{2n-1}t_0t_1,t^{2n-1}t_0t_2,t^{2n-1}t_0t_3;t;p,q)^{-1}
\obinomE{\bkappa}{\blambda}_{[1/u_0^2,1/t^{2n-1}t_0u_0];t;p,q}
\]
and sum over $\blambda$, the right-hand sum collapses to a delta function,
and thus vanishes unless $\bkappa=\bmu^2$ for some $\bmu$.  The effect on
the left-hand side is to produce a biorthogonal function, and we thus
obtain the following vanishing identity.

\begin{cor}\label{cor:vant}
For generic parameters such that $t^{4n-2}t_0t_1t_2t_3u_0^2=pq$, the
integral
\[
\langle
\tcR^{(2n)}_{\blambda}(\dots,t^{\pm 1/2} z_i,\dots;t_0{:}t_1,t_2,t_3;u_0,u_0;t;p,q)
\rangle^{(n)}_{t^{1/2}t_0,t^{1/2}t_1,t^{1/2}t_2,t^{1/2}t_3,t^{\pm 1/2}u_0;t^2;p,q}
\]
vanishes unless $\blambda=\bmu^2$ for some partition pair $\bmu$, in which
case it equals
\[
\frac{\Delta_{\bmu}(1/tu_0^2|t^{2n},t^{2n-1}t_0t_1,t^{2n-1}t_0t_2,t^{2n-1}t_0t_3,1/t^{2n-1}t_0u_0,1/t^{2n}t_0u_0;t^2;p,q)}
     {\Delta_{\bmu^2}(1/u_0^2|t^{2n},t^{2n-1}t_0t_1,t^{2n-1}t_0t_2,t^{2n-1}t_0t_3,1/t^{2n-1}t_0u_0,1/t^{2n-1}t_0u_0;t;p,q)}
.
\]
\end{cor}

\begin{rem}
  In fact, although we have referred to the functions above as
  ``biorthogonal'' functions, since $u_0=u_1$, they actually form an
  orthogonal basis of the appropriate space of functions.
\end{rem}

If we fix $t_0$, $t_1$, $t_2$, $t_3$ and let $p\to 0$ (solving for $u_0$
via the balancing condition, so that $u_0\sim\sqrt{p}$), the biorthogonal
functions converge to Koornwinder polynomials, and the density converges to
a (different) Koornwinder density.  The result is one of the vanishing
integrals of \cite{vanish} (Theorem~4.9 op. cit.), {\em together} with the
nonzero values (which were not accessible to the methods used there).  In
the notation of \cite{bcpoly}, one has
\begin{align}
&I_K(\tilde{K}_\lambda([p_k(t^{k/2}+t^{-k/2})];q,t;T;t_0,t_1,t_2,t_3)
;q,t^2,T;t^{1/2}t_0,t^{1/2}t_1,t^{1/2}t_2,t^{1/2}t_3)\notag\\
&=
\delta_{\lambda\mu^2}
\frac{t^{-|\mu|}
      \prod_{0\le r<s\le 3} C^0_\mu(T t_rt_s/t;q,t^2)
      C^0_\mu(T,Tt_0t_1t_2t_3/t^2;q,t^2)
      C^+_\mu(T^2t_0t_1t_2t_3/t^4;q,t^2)
      C^-_\mu(qt;q,t^2)}
     {C^0_{2\mu^2}(T^2t_0t_1t_2t_3/t^2;q,t^2)
      C^+_\mu(T^2t_0t_1t_2t_3/qt^3;q,t^2)
      C^-_\mu(t^2;q,t^2)}.
\end{align}
(If $T=t^{2n}$, this states that the integral of
$K^{(2n)}_\lambda(\dots,t^{\pm 1/2}z_i,\dots;q,t;t_0,t_1,t_2,t_3)$ against
the normalized density with parameters
$q,t^2;t^{1/2}t_0,t^{1/2}t_1,t^{1/2}t_2,t^{1/2}t_3$ vanishes unless
$\lambda=\mu^2$, when the value is as given.)

We furthermore conjecture that Theorem~\ref{thm:int_litt} extends to the
following transformation (much as Corollary~\ref{cor:ipR} extends to
Theorem 9.7 of \cite{xforms}).  For the significance of the label
$(t^{-1/2})$, see the end of Section~\ref{sec:var}.

\begin{conjL}[$t^{-1/2}$]\label{conj:littbig}
For generic parameters such that $t^{4n-4}t_0^2u_0^2v_0v_1v_2v_3=p^2q^2$,
one has
\begin{align}
\int
\cR^{*(2n)}_{\blambda}&(\dots,t^{\pm 1/2} z_i,\dots;t_0,u_0;t;p,q)
\Delta^{(n)}(;t^{\pm 1/2}t_0,t^{\pm 1/2}u_0,v_0,v_1,v_2,v_3;t^2;p,q)\notag\\
&=
\prod_{0\le r\le 3}
\Delta^0_{\blambda}(t^{2n-1}t_0/u_0|t^{2n-3/2}t_0v_r;t;p,q)
  \prod_{1\le i\le 2n}
    \Gampq(t^{2n-1/2-i} t_0 v_r,t^{2n-1/2-i} u_0 v_r)
\notag\\
&\phantom{{}={}}\times\int
\cR^{*(2n)}_{\blambda}(\dots,t^{\pm 1/2}z_i,\dots;t_0,u_0;t;p,q)
\Delta^{(n)}(;t^{\pm 1/2}t_0,t^{\pm 1/2}u_0,v'_0,v'_1,v'_2,v'_3;t^2;p,q),
\label{eq:littbig}
\end{align}
where $v'_r = pq/t^{2n-2}t_0u_0v_r$.
\end{conjL}

This is accessible in a number of special cases.  When $t^{-1/2}t_0v_0 =
pq$, so the left-hand side reduces to the left-hand side of
Theorem~\ref{thm:int_litt}, the transformed parameters satisfy
$t^{2n-2}t^{1/2}u_0v'_0=1$, and thus the right-hand side degenerates to a
sum.  If one traces through the relevant contour conditions, one finds that
the sum is over partitions contained in $\bmu$, and one obtains the
right-hand side of Theorem~\ref{thm:int_litt}.  It follows, then, (using
the consistency under parameter shifts, below) that any ``algebraic'' case
(i.e., in which both sides can be renormalized to products of $p$- and
$q$-theta functions) of the conjecture holds.

When $\ell(\blambda)=1$, the skew interpolation function is independent of
$t$, which implies
\[
\cR^{*(2n)}_{\blambda}(\dots,t^{\pm 1/2} z_i,\dots;t_0,u_0;t;p,q)
=
\cR^{*(n)}_{\blambda}(\dots,z_i,\dots;t^{1/2}t_0,t^{-1/2}u_0;t^2;p,q)
\]
when $\ell(\blambda)\le 1$.  Thus in that case, the conjecture becomes a
special case of \cite[Thm.~9.7]{xforms}.  We also find that the identity
for $(l,m)^{2n}+\blambda$ follows trivially from that for $\blambda$;
combining these two facts proves the identity when $n=1$, and then trivially
the case $t=1$.  In the case $t^{2n}t_0u_0=pq$, the interpolation function
is of Cauchy type, and thus factors for general $\blambda$:
\begin{align}
\cR^{*(2n)}_{\blambda}(\dots,t^{\pm 1/2}z_i,\dots;t_0,pq/t^{2n}t_0;t;p,q)
={}&
\cR^{*(n)}_{\blambda^+}(\dots,z_i,\dots;t^{1/2}t_0,pq/t^{2n+1/2}t_0;t^2;p,q)\notag\\
&
\cR^{*(n)}_{\blambda^-}(\dots,z_i,\dots;t^{-1/2}t_0,pq/t^{2n-1/2}t_0;t^2;p,q),
\end{align}
and again the identity reduces to the transformation of
\cite[Thm.~9.7]{xforms}.  The case $t=q$ (and, by symmetry, $t=p$) can be
dealt with via the observation that the $p$-elliptic interpolation functions can
in that case be expressed as a ratio of determinants, while the
$q$-elliptic interpolation function is a symmetrized product, just as for
$t=1$.  More precisely, one has
\[
\cR^{*(n)}_{\lambda,\mu}(\dots,z_i,\dots;t_0,u_0;q;p,q)
\propto
\frac{
\sum_{\pi,\rho\in S_n}
\sigma(\rho)
\prod_{1\le i\le n}
  \cR^{*(1)}_{\lambda_{\pi_i},\mu_{\rho_i}+n-\rho_i}(z_i;t_0,q^{n-1}u_0;q;p,q)
  }
{
\prod_{1\le i\le n}
\theta_p(u_0 z_i^{\pm 1};q)_{n-1}^{-1}
\prod_{1\le i<j\le n} z_i^{-1}\theta_p(z_i z_j^{\pm 1})}.
\]
(Since the same interpolation function appears on
both sides of \eqref{eq:littbig}, we can freely ignore constants.)  After
specializing the variables, the denominator cancels out the cross terms
from the density, so that one can express the identity as a sum of products
of instances with $n=1$.  (When $\blambda=(0,\mu)$, this sum of products is a
pfaffian, compare \cite{deBruijnNG:1955}.)

One final set of special cases is of interest.

\begin{prop}\label{prop:adj_t}
Conjecture~\ref{conj:littbig} holds whenever $t^{2n-2}t_0u_0v_0v_1/pq\in
\{1,1/p,1/q,t\}$.
\end{prop}

\begin{proof}
  If $t^{2n-2}t_0u_0v_0v_1/pq=1$, then the transformation is trivial.  If
  $t^{2n-2}t_0u_0v_0v_1/pq=t$, we may use the integral equation
  \cite[(8.12)]{xforms} to write
\begin{align}
\cR^{*(2n)}_{\blambda}(;t_0,u_0;t;p,q)
={}&
\Delta^0_{\blambda}(t^{2n-1}t_0/u_0|t^{2n-3/2}t_0v_0,t^{2n-3/2}t_0v_1;t;p,q)
\notag\\
&
\times\cI^{(2n)}(t^{-1/2}u_0{:}t^{-1/2}t_0,t^{-1}v_0;p,q)
\cR^{*(2n)}_{\blambda}(;t^{-1/2}t_0,t^{-1/2}u_0;t;p,q),
\end{align}
and thus reduce to showing that
\begin{align}
&\Gampq(t^{-3/2} t_0 v_0,t^{-3/2}u_0v_0,t^{-3/2} t_0
v_1,t^{-3/2}u_0v_1)\notag\\
&
\times\int
(
\cI^{(2n)}(t^{-1/2}u_0{:}t^{-1/2}t_0,v_0/t;p,q)f
)(\dots,t^{\pm 1/2}z_i,\dots)
\Delta^{(n)}(;t^{\pm 1/2}t_0,t^{\pm 1/2}u_0,v_0,v_1,v_2,v_3;t^2;p,q)
\end{align}
is invariant under $(v_0,v_1,v_2,v_3)\mapsto (v'_2,v'_3,v'_0,v'_1)$ for any
function $f$ in the span of the interpolation functions.  Now, specializing
the output of the integral operator pinches the contour, and thus we pick
up an $n$-fold residue.  We thus find in general that if
$t^{2n}u_0u_1u_2u_3=pq$, then
\begin{align}
\prod_{0\le r<s<4} \Gampq(u_ru_s)
(\cI^{(2n)}_t(u_0{:}u_1,u_2;p,q) f)&(\dots,t^{\pm 1/2}z_i,\dots)\notag\\
=
\frac{((p;p)(q;q)/2\Gampq(t^2))^n}{n!}
\int_{C^n}&
f(x_1,\dots,x_n,z_1,\dots,z_n)
\frac{\prod_{1\le i,j\le n} \Gampq(t x_i^{\pm 1}z_j^{\pm 1})}
   {\prod_{1\le i<j\le n} \Gampq(x_i^{\pm 1}x_j^{\pm 1},t^2 z_i^{\pm 1}z_j^{\pm 1})}\notag\\
&\times\prod_{1\le i\le n} 
  \frac{\prod_{0\le r<4} \Gampq(u_r x_i^{\pm 1})}
       {\Gampq(x_i^{\pm 2})\prod_{0\le r<4} \Gampq(t u_r z_i^{\pm 1})}
  \frac{dx_i}{2\pi\sqrt{-1}x_i}.
\end{align}
Substituting in and exchanging order of integration gives the desired
result.

By symmetry, it remains only to consider the case
$t^{2n-2}t_0u_0v_0v_1/pq=1/q$.  Here we use the difference equation
\cite[(8.11)]{xforms} rather than the integral equation, and reduce to
showing that
\[
\int
(
D^{(2n)}_q(u_0,t_0,t^{-1/2}v_0,t^{-1/2}v_1;t;p)
f)(\dots,t^{\pm 1/2}z_i,\dots)
\Delta^{(n)}(;t^{\pm 1/2}t_0,t^{\pm 1/2}u_0,v_0,v_1,v_2,v_3;t^2;p,q)
\]
is invariant under $(v_0,v_1,v_2,v_3)\to (v'_2,v'_3,v'_0,v'_1)$.
After specializing the image of the difference operator, $f$ appears
in principle in $4^n$ different specializations: corresponding to each
$z_i$ is a pair of arguments, one of
\[
(q^{-1/2}t^{-1/2}z_i,q^{-1/2}t^{1/2}z_i),
(q^{-1/2}t^{-1/2}z_i,q^{1/2}t^{1/2}z_i),
(q^{1/2}t^{-1/2}z_i,q^{-1/2}t^{1/2}z_i),
(q^{1/2}t^{-1/2}z_i,q^{1/2}t^{1/2}z_i).
\]
The third pair never actually occurs (the coefficient vanishes), and
we can arrange to combine the first and fourth cases by shifting the
variable by $q^{1/2}$ or $q^{-1/2}$ as appropriate.  (This changes the
contour of that portion of the integral, but we can move it back without
crossing over any poles.)  We thus obtain $2^n$ different specializations
of $f$, involving the pairs
\[
(t^{-1/2}z_i,t^{1/2}z_i) \text{ and }
(q^{-1/2}t^{-1/2}z_i,q^{1/2}t^{1/2}z_i);
\]
the coefficient of a given specialization of $f$ is a sum of $2^m$ terms
where $m$ is the number of times the first pair is used.  If we fix a given
specialization of $f$, we can remove common factors of the $2^m$ terms of
its coefficient to obtain the instance $(p,q,t,v_0,v_1,v_2,v_3)\mapsto
(p,qt,t^2,pq^{1/2}t/v_0,pq^{1/2}t/v_1,q^{-1/2}v_2,q^{-1/2}v_3)$ of
Lemma~\ref{lem:diff_commut}.
\end{proof}

%
%

\medskip
Of course, the usual Littlewood identity also comes in a dual form, and the
same applies at the elliptic level.  Since duality breaks the symmetry
between $p$ and $q$, it in particular does not apply at the level of
partition pairs.  However, we do obtain the following, purely $p$-elliptic,
identity.

\begin{cor}
One has the identity
\[
\sum_{\mu}
\obinomE{2\mu}{\lambda}_{[a,b];q,t;p}
\frac{\Delta_{\mu}(a|v_0,v_1,v_2,v_3;q^2,t;p)}
     {\Delta_{\lambda}(a/b|v_0,v_1,v_2,v_3;q,t;p)}
\propto
\sum_{\mu}
\obinomE{\lambda}{2\mu}_{[a/b,b];q,t;p}
\frac{\Delta_{\mu}(a/b^2|v_0,v_1,v_2,v_3;q^2,t;p)}
     {\Delta_{2\mu}(a/b^2|v_0,v_1,v_2,v_3;q,t;p)},
\]
subject to the balancing condition
\[
v_0v_1v_2v_3b^2 = pqta^2,
\]
and the termination conditions
\begin{align}
t^{\N}\cap \{v_0,v_1,v_2,v_3,b,bq\}&\ne\emptyset\\
q^{-2\N}\cap \{v_0,v_1,v_2,v_3,b,bq\}&\ne\emptyset,
\end{align}
with associated conditions on $\lambda$.  The constant is given by the
value for $\lambda=0$.
\end{cor}

Analytically continuing the left-hand side to an integral produces the
following dual vanishing integral.

\begin{cor}\label{cor:dual_vanq}
If $t^{2n-2}t_0t_1t_2t_3u_0^2=pq$, then
\[
\langle
\tilde{R}^{(n)}_\lambda(;t_0{:}t_1,t_2,t_3;u_0,u_0;q,t;p)
\rangle_{t_0,t_1,t_2,t_3,u_0,qu_0;t;p,q^2}
\]
vanishes unless $\lambda$ is of the form $2\mu$, when the integral is
\[
\frac{\Delta_{\mu}(1/u_0^2|t^n,t^{n-1}t_0t_1,t^{n-1}t_0t_2,t^{n-1}t_0t_3,1/t^{n-1}t_0u_0,q/t^{n-1}t_0u_0;q^2,t;p)}
     {\Delta_{2\mu}(1/u_0^2|t^n,t^{n-1}t_0t_1,t^{n-1}t_0t_2,t^{n-1}t_0t_3,1/t^{n-1}t_0u_0,1/t^{n-1}t_0u_0;q,t;p)}.
\]
\end{cor}

Again, here, the vanishing corresponds to the fact that evaluation at a
partition with respect to $q^2$, $t$, is also evaluation at the doubled
partition with respect to $q$, $t$:
\[
(q^2)^{\mu_i}t^{n-i}t_0 = q^{2\mu_i}t^{n-i}t_0.
\]
This continues to hold even for partition pairs:
\[
p^{\lambda_i}(q^2)^{\mu_i} t^{n-i}t_0 = p^{\lambda_i}q^{2\mu_i}t^{n-i}t_0,
\]
suggesting the conjecture that
\[
\langle
\tcR^{(n)}_\blambda(;t_0{:}t_1,t_2,t_3;u_0,u_0;t;p,q)
\rangle^{(n)}_{t_0,t_1,t_2,t_3,u_0,qu_0;t;p,q^2}
\]
vanishes unless $\blambda=(1,2)\bmu$ for some partition pair $\bmu$ (where
$(1,2)(\mu,\nu)=(\mu,2\nu)$), when it equals
\[
\frac{\Delta_{     \bmu}(1/u_0^2|t^n,t^{n-1}t_0t_1,t^{n-1}t_0t_2,t^{n-1}t_0t_3,1/t^{n-1}t_0u_0,q/t^{n-1}t_0u_0;t;p,q^2)}
     {\Delta_{(1,2)\bmu}(1/u_0^2|t^n,t^{n-1}t_0t_1,t^{n-1}t_0t_2,t^{n-1}t_0t_3,1/t^{n-1}t_0u_0,1/t^{n-1}t_0u_0;t;p,q)}.
\]
Note, however, that this does not correspond to a vanishing result with
respect to the other partition.

The transformation analogue of this extended conjecture is the following.

\begin{conjL}[$q^{1/2}$]\label{conj:littbig_dual}
If $t^{2n-2}t_0^2u_0^2v_0v_1v_2v_3=p^2q^2$, then
\begin{align}
\int
&\cR^{*(n)}_{\blambda}(;t_0,u_0;t;p,q)
\Delta^{(n)}(;t_0,qt_0,u_0,qu_0,v_0,v_1,v_2,v_3;t;p,q^2)\notag\\
&\qquad\qquad{}=
\prod_{0\le r\le 3}
  \Delta^0_{\blambda}(t^{n-1}t_0/u_0|t^{n-1}t_0v_r;t;p,q)
  \prod_{1\le i\le n}
     \Gampq(t^{n-i} t_0v_r,t^{n-i} u_0v_r)\notag\\
&\qquad\qquad\phantom{{}={}}
\times\int
\cR^{*(n)}_{\blambda}(;t_0,u_0;t;p,q)
\Delta^{(n)}(;t_0,qt_0,u_0,qu_0,v'_0,v'_1,v'_2,v'_3;t;p,q^2)
,
\end{align}
where $v'_r = pq/t^{n-1}t_0u_0v_r$.
\end{conjL}

\begin{rem}
  If $\blambda=(l,m)^n$ (so in particular if $n=1$), or $t^n t_0u_0=pq$,
  this is again a special case of the transformation of \cite[\S
  9]{xforms}.  When $t=q$ or $t=p$, this again essentially reduces to a
  pfaffian, except that the individual entries include cases with $n=2$, so
  this does not quite lead to a proof in that case.  Note that in this
  case, the cross-terms do not quite cancel, so each term involves the
  factor
\[
\prod_{1\le i<j\le n}
  \frac{q^{1/2}z_i^{-1}\theta_{q^2}(z_iz_j^{\pm 1})}
       {\theta_{q^2}(qz_iz_j^{\pm 1})}.
\]
Since (\cite[Thm.~2.10]{recur})
\[
\pf_{1\le i<j\le 2n}
\frac{q^{1/2}z_i^{-1}\theta_{q^2}(z_iz_j^{\pm 1})}
     {\theta_{q^2}(qz_iz_j^{\pm 1})}
=
\prod_{1\le i<j\le 2n}
  \frac{q^{1/2}z_i^{-1}\theta_{q^2}(z_iz_j^{\pm 1})}
       {\theta_{q^2}(qz_iz_j^{\pm 1})},
\]
and similarly for odd $n$, one can adapt the argument of
\cite{deBruijnNG:1955}.  Indeed, one finds that for all $n\ge 0$,
\[
\prod_{1\le i<j\le n}
  \frac{q^{1/2}z_i^{-1}\theta_{q^2}(z_iz_j^{\pm 1})}
       {\theta_{q^2}(qz_iz_j^{\pm 1})}
=
\frac{1}{2^{\lfloor n/2\rfloor} \lfloor n/2\rfloor !}
\sum_{\pi\in S_n}
\sigma(\pi)
\pi\cdot
\prod_{1\le i\le \lfloor n/2\rfloor}
  \frac{q^{1/2}z_{2i-1}^{-1}\theta_{q^2}(z_{2i-1}z_{2i}^{\pm 1})}
       {\theta_{q^2}(qz_{2i-1}z_{2i}^{\pm 1})},
\]
where $\pi\in S_n$ acts by permuting the variables.  Since the remainder of
the integrand is antisymmetric, each term in this sum has the same
integral, so that one again obtains a sum of products of low-dimensional
integrals.
\end{rem}

\begin{prop}\label{prop:adj_q}
  Conjecture~\ref{conj:littbig_dual} holds if $t^{n-1}t_0u_0v_0v_1/pq\in
  \{1,1/p,1/q,t\}$.
\end{prop}

\begin{proof}
Again, the case $t^{n-1}t_0u_0v_0v_1/pq=1$ is trivial.  The case
$t^{n-1}t_0u_0v_0v_1/pq=1/q$ corresponds to the fact that
\[
\int
\bigl(D^{(n)}_q(t_0,u_0,v_0;t;q) f\bigr)
\Delta^{(n)}(t_0,qt_0,u_0,qu_0,v_0,v_1,v_2,v_3;t;p,q^2)
\]
is invariant under $(v_0,v_1,v_2,v_3)\to (v'_2,v'_3,v'_0,v'_1)$ for any
function $f$.  Expanding this as a sum of $2^n$ terms and undoing all
variable shifts gives a manifestly invariant sum; indeed, changing the $v$
parameters has the same effect as inverting all the variables.

Similarly,
\[
\int
\bigl(D^{(n)}_p(t_0,u_0,v_0;t;q) f\bigr)
\Delta^{(n)}(t_0,qt_0,u_0,qu_0,v_0,v_1,v_2,v_3;t;p,q^2)
\]
is invariant; after expanding and unshifting, one obtains the special case
of Lemma~\ref{lem:diff_commut} with 
\[
(p,q,t,v_0,v_1,v_2,v_3)\mapsto
(q^2,p/q,t,p^{1/2}q/v_0,p^{1/2}q/v_1,p^{-1/2}v_2,p^{-1/2}v_3).
\]

When $t^{n-1}t_0u_0v_0v_1/pq=t$, we can use the integral equation to write
\begin{align}
\cR^{*(n)}_{\blambda}(;t_0,u_0;t;p,q)
&{}=
\Delta^0_{\blambda}(t^{n-1}t_0/u_0|t^{n-1}t_0v_0,t^{n-1}t_0v_1;t;p,q)\notag\\
&\hphantom{{}={}}\times\cI^{(n)}(t^{-1/2}u_0{:}t^{-1/2}t_0,t^{-1/2}v_0;p,q)
\cR^{*(n)}_{\blambda}(;t^{-1/2}t_0,t^{-1/2}u_0;t;p,q)
\end{align}
on the right-hand side.  After changing order of integration, the
inner integral has the form
\[
\II^{(n-1)}_n(pq/v_0,pq/v_1,v_2,v_3,\dots,t^{1/2}x_i^{\pm 1},qt^{1/2}x_i^{\pm 1},\dots;pq/t;p,q^2),
\]
so can be transformed using Theorem~\ref{conj:int_commut} to give the
desired identity.
\end{proof}

\begin{rem}
It is natural to try to extend the proofs for $1/p$ and $1/q$ using the
iterated difference operators introduced in the proof of
\cite[Thm.~9.7]{xforms}.  We find that the proof in those cases would
reduce to showing that when $v_0v_1v_2v_3=p^2q^2$, 
\[
\prod_{\substack{1\le i\le n\\0\le r<4}} \frac{1}{\Gampqq(v_r x_i^{\pm 1})}
D^{(n)}_{l,m}(t;p,q)
\prod_{\substack{1\le i\le n\\0\le r<4}} \Gampqq(p^{l/2}q^{m/2}v_r x_i^{\pm
  1})
\frac{\Delta^{(n)}(x_1,\dots,x_n;t;p,q^2)}
     {\Delta^{(n)}(x_1,\dots,x_n;t;p,q)}
\]
is invariant under $(v_0,v_1,v_2,v_3)\mapsto
(pq/v_0,pq/v_1,pq/v_2,pq/v_3)$.  The case $(l,m)=(0,1)$ holds even without
balancing condition (or limit on the number of $v_r$ factors).  Using this,
one can mimic the proof of Theorem~\ref{thm:genlitt} above to show that the
identity for $(l,m)$ implies the identity for $(l,m+1)$; this implies
Conjecture~\ref{conj:littbig_dual} whenever $t^{n-1}t_0u_0v_0v_1/pq\in
\{q^{-m},p^{-1}q^{-m}\}$ for $m\ge 0$.  Note that when $m=0$, the ratio of
any two terms of the sum is $q^2$-elliptic, so this reduces to an algebraic
statement, presumably equivalent to the conjectured equation
\eqref{conj:genlitt_dual_qp} below.  It should be possible to interpret and
extend the proof of Proposition~\ref{prop:adj_t} in a similar way, with the
corresponding identity an analytic continuation of
Corollary~\ref{cor:diff_eq_t}, though the specialization of the variables
makes this nontrivial.
\end{rem}

\begin{cor}
If $t^{2n-2}t_0^2u_0^2v_0v_1 = p$, then
\begin{align}
&\left\langle
\frac{\cR^{*(n)}_\blambda(;t_0,u_0;t;p,q)}
     {\Delta^0_\blambda(t^{n-1}t_0/u_0|t^{n-1}t_0v_0,t^{n-1}t_0v_1;t;p,q)}
\right\rangle^{(n)}_{t_0,qt_0,u_0,qu_0,v_0,v_1;t;p,q^2}\notag\\
&=
\Delta^0_{\blambda}(t^{n-1}t_0/u_0|qt^{n-2}t_0^2;t;p,q)
\sum_{\bmu}
\obinomE{\blambda}{(1,2)\bmu}_{[t^{n-1}t_0/u_0,t^{n-1}t_0u_0];t;p,q}
\frac{\Delta_{     \bmu}(1/u_0^2|t^n,qt^{n-1}t_0^2;t;p,q^2)}
     {\Delta_{(1,2)\bmu}(1/u_0^2|t^n,qt^{n-1}t_0^2;t;p,q)}.
\end{align}
\end{cor}

\begin{proof}
  The left-hand side is simply the instance $v_2v_3=pq$ of the left-hand
  side of Conjecture~\ref{conj:littbig_dual}, which is manifestly invariant of
  the choice of $v_2$.  In particular, we may arrange for
  $t^{n-1}t_0u_0v_0v_2\in \{p,q\}$ so that we may apply our known special
  cases of the conjecture.  In these cases, the transformation simply
  applies a $p$- or $q$-shift to $v_0$ and $v_1$; since the left-hand side
  is meromorphic in $v_0/v_1$ and invariant under both $p$ and $q$-shifts,
  it is in fact independent of $v_0/v_1$.  Taking the limit $v_0\to
  1/t^{n-1}u_0$ gives the desired result.
\end{proof}

The ``$q$-elliptic'' half\footnote{To be precise, this is only
  $q^2$-elliptic, but we abuse terminology to distinguish it from the
  $p$-elliptic half.} of the dual Littlewood identity reads (after swapping
$p$ and $q$)
\[
\sum_{\mu}
\obinomE{\mu}{\lambda}_{[a,b];q,t;p}
\frac{\Delta_{\mu}(a|v_0,v_1,v_2,v_3;q,t;p^2)}
     {\Delta_{\lambda}(a/b|v_0,v_1,v_2,v_3;q,t;p)}
\propto^?
\sum_{\mu}
\obinomE{\lambda}{\mu}_{[a/b,b];q,t;p}
\frac{\Delta_{\mu}(a/b^2|v_0,v_1,v_2,v_3;q,t;p^2)}
     {\Delta_{\mu}(a/b^2|v_0,v_1,v_2,v_3;q,t;p)},
\label{conj:genlitt_dual_qp}
\]
and in particular involves both $p$-abelian and $p^2$-abelian functions.
This transformation is taken to itself by duality, but if we use a modular
transformation, we can replace the $2$-isogeny
\[
z\in \C^*/\langle p^2\rangle\mapsto z\in \C^*/\langle p\rangle
\]
by
\[
z\in \C^*/\langle p^{1/2}\rangle \mapsto z^2\in \C^*/\langle p\rangle.
\]
This then gives rise to the following conjecture, upon lifting back to an
integral transformation.

\begin{conjL}[$-1$]\label{conj:littbig_m}
If $t^{n-1}t_0^2u_0^2v_0v_1v_2v_3=pq$, then
\begin{align}
\int&
\cR^{*(n)}_{\blambda}(\dots,z_i^2,\dots;t_0^2,u_0^2;t;p,q)
\Delta^{(n)}(;t_0,-t_0,u_0,-u_0,v_0,v_1,v_2,v_3;t^{1/2};p^{1/2},q^{1/2})\notag\\
&{}=
\prod_{0\le r\le 3}
  \Delta^0_{\blambda}(t^{n-1}t_0^2/u_0^2|t^{n-1}t_0^2v_r^2;t;p,q)
  \prod_{0\le i<n}
    \Gampq(t^i t_0^2 v_r^2,t^i u_0^2 v_r^2)\notag\\
&\phantom{{}={}}
\times\int
\cR^{*(n)}_{\blambda}(\dots,z_i^2,\dots;t_0^2,u_0^2;t;p,q)
\Delta^{(n)}(;t_0,-t_0,u_0,-u_0,v'_0,v'_1,v'_2,v'_3;t^{1/2};p^{1/2},q^{1/2}),
\end{align}
where $v'_r = p^{1/2}q^{1/2}/t^{(n-1)/2}t_0u_0v_r$.
\end{conjL}

Again, this holds if $\blambda=(l,m)^n$ or $t^{n/2}t_0u_0 =
p^{1/2}q^{1/2}$.  When $t=p$ or $t=q$ (really four cases, as either square
root of $t$ will work), the integral has a similar structure to the
pfaffian case of Conjecture~\ref{conj:littbig_dual}, except that the pfaffian
factor is now either
\[
\prod_{1\le i<j\le n}
  \frac{\theta_{p^{1/2}}(z_iz_j^{\pm 1})}{\theta_{p^{1/2}}(-z_iz_j^{\pm
      1})}
\]
or
\[
\prod_{1\le i<j\le n}
  \frac{\theta_{q^{1/2}}(z_iz_j^{\pm 1})}{\theta_{q^{1/2}}(-z_iz_j^{\pm
      1})}
\]
(This is a pfaffian by virtue of being a modular transformation of the
previous pfaffian.)

Although it does not have an associated vanishing result per se, there is
an associated analogue of Corollary~\ref{cor:vant}, namely the conjecture
that if $t^{n-1}t_0t_1t_2t_3u_0^2 = -p^{1/2}q^{1/2}$, then
\begin{align}
\langle
&\tcR^{(n)}_{\blambda}(\dots,z_i^2,\dots;t_0^2{:}t_1^2,t_2^2,t_3^2;u_0^2,u_0^2;t;p,q)
\rangle^{(n)}_{t_0,t_1,t_2,t_3,u_0,-u_0;t^{1/2};p^{1/2},q^{1/2}}\notag\\
&\quad
=
\frac{\Delta_{\blambda}(1/u_0^2|t^{n/2},t^{(n-1)/2}t_0t_1,t^{(n-1)/2}t_0t_2,t^{(n-1)/2}t_0t_3,1/t^{(n-1)/2}t_0u_0,-1/t^{(n-1)/2}t_0u_0;t^{1/2};p^{1/2},q^{1/2})}
{\Delta_{\blambda}(1/u_0^4|t^n,t^{n-1}t_0^2t_1^2,t^{n-1}t_0^2t_2^2,t^{n-1}t_0^2t_3^2,1/t^{n-1}t_0^2u_0^2,1/t^{n-1}t_0^2u_0^2;t;p,q)}.
\end{align}
If we eliminate $u_0$ using the balancing condition then take the limit
$p\to 0$, we obtain a conjectural integral of Koornwinder polynomials which
in the notation of \cite{bcpoly} reads
\begin{align}
I_K&(\tilde{K}_\lambda([p_{2k}];q^2,t^2,T^2;t_0^2,t_1^2,t_2^2,t_3^2)
   ;q,t,T;t_0,t_1,t_2,t_3)\notag\\
&{}=
\frac{(-1)^{|\lambda|}
      C^0_\lambda(T,Tt_0t_1t_2t_3/t^2;q,t)      C^+_\lambda(T^2t_0t_1t_2t_3/t^3;q,t)
      C^-_\lambda(-q;q,t)\prod_{0\le r<s<4} C^0_\lambda(Tt_rt_s/t;q,t)
      }
     {C^0_{2\lambda^2}(T^2t_0t_1t_2t_3/t^2;q,t)C^+_\lambda(-T^2t_0t_1t_2t_3/qt^2;q,t)
      C^-_\lambda(t;q,t)}
\end{align}
(When $T=t^n$, this is an $n$-dimensional integral of $n$-variable
Koornwinder polynomials evaluated at $z_1^2,\dots,z_n^2$.)  This is no
longer related to a Littlewood identity, but is instead related to an
identity conjectured by Kawanaka in \cite{KawanakaN:1999} and proved (via
elliptic means) in \cite{LangerR/SchlosserMJ/WarnaarSO:2009}.  Indeed, if
we set $T=0$ (i.e., take $n\to\infty$) we obtain an identity which can be
used to integrate the left-hand side of the appropriate Cauchy identity
\cite[Thm.~7.19]{bcpoly} term by term.  Since the coefficients in that
Cauchy identity are Macdonald polynomials, we obtain the sum
\[
\sum_\lambda
\frac{C^-_\lambda(-t^{1/2};q^{1/2},t^{1/2})}
     {C^-_\lambda(q^{1/2};q^{1/2},t^{1/2})}
P_\lambda(\dots,x_i,\dots;q,t)
=
\prod_j \frac{(-t^{1/2}x_j;q^{1/2})}{(x_j;q^{1/2})}
\prod_{j<k} \frac{(tx_jx_k;q)}{(x_jx_k;q)},
\]
using \cite[Thm.~7.17]{bcpoly} to compute the integral and obtain the right-hand
side.  Note that since this is independent of $t_0$, $t_1$, $t_2$, $t_3$,
and the case $t_1=-t_0$ of the elliptic conjecture follows from
Corollary~\ref{cor:vanm} below, this argument gives a second proof of
Kawanaka's conjecture.

\begin{prop}\label{prop:adj_m}
  Conjecture~\ref{conj:littbig_m} holds if
  $t^{n-1}t_0^2u_0^2v_0^2v_1^2/pq\in \{1,1/p,1/q,t\}$.
\end{prop}

\begin{proof}
Again, the $1$ case is trivial, and by symmetry the $1/p$ and $1/q$ cases
are equivalent.  The $1/q$ case reduces to the invariance of
\begin{align}
\int
(
D^{(n)}_{q}(t_0^2,u_0^2,v_0^2;t;p)&
\cR^{*(n)}_{\blambda}(;q^{1/2}t_0^2,q^{1/2}u_0^2;t;p,q)
)(\dots,z_i^2,\dots)\notag\\
&
\times\Delta^{(n)}(;t_0,-t_0,u_0,-u_0,v_0,v_1,v_2,v_3;t^{1/2};p^{1/2},q^{1/2})
\end{align}
under $v_0,v_1,v_2,v_3\to v'_2,v'_3,v'_0,v'_1$.  This follows as above,
this time via the special case $(p,q,t,v_0,v_1,v_2,v_3)\mapsto
(p^{1/2},-q^{1/2},t^{1/2},-p^{1/2}q^{1/4}/v_0,-p^{1/2}q^{1/4}v_1,q^{-1/4}v_2,q^{-1/4}v_3)$
of Lemma~\ref{lem:diff_commut}.

Similarly, the case $t^{n-1}t_0^2u_0^2v_0^2v_1^2/pq=t$ reduces to a special
case of Theorem~\ref{conj:int_commut} as before.
\end{proof}

\begin{rem}
  As in the remark following Proposition~\ref{prop:adj_q}, the identity
  used in the proof for $1/p$ and $1/q$ can be interpreted as a special case
  of a more general difference equation.  To wit, if $v_0v_1v_2v_3=pq$, then
\[
\prod_{\substack{1\le i\le n\\0\le r<4}} \frac{1}{\Gamphqh(v_r x_i^{\pm 1/2})}
D^{(n)}_{l,m}(t;p,q)
\prod_{\substack{1\le i\le n\\0\le r<4}} \Gamphqh(S_{l,m}^{1/4}v_r x_i^{\pm 1/2})
\frac{\Delta^{(n)}(x_1^{1/2},\dots,x_n^{1/2};t^{1/2};p^{1/2},q^{1/2})}
     {\Delta^{(n)}(x_1,\dots,x_n;t;p,q)}
\]
should be invariant under $(v_0,v_1,v_2,v_3)\mapsto
(-p^{1/2}q^{1/2}/v_0,-p^{1/2}q^{1/2}/v_1,-p^{1/2}q^{1/2}/v_2,-p^{1/2}q^{1/2}/v_3)$.
Here, one must be careful to make consistent choices of fourth roots of $p$
and $q$ so that the action of the difference operator is still well-defined
despite having taken square roots of the variables.
\end{rem}

\begin{cor}\label{cor:vanm}
If $t^{n-1}t_0^2u_0^2v_0v_1 = p^{1/2}q^{1/2}$, then
\begin{align}
&\left\langle
\frac{\cR^{*(n)}_{\blambda}(\dots,z_i^2,\dots;t_0^2,u_0^2;t;p,q)}
{\Delta^0_{\blambda}(t^{n-1}t_0^2/u_0^2|t^{n-1}t_0^2v_0^2,t^{n-1}t_0^2v_1^2;t;p,q)}
\right\rangle^{(n)}_{t_0,-t_0,u_0,-u_0,v_0,v_1;t^{1/2};p^{1/2},q^{1/2}}\notag\\
&=
\Delta^0_{\blambda}(t^{n-1}t_0^2/u_0^2|t^{n-1}t_0^4;t;p,q)
\sum_\bmu
\obinomE{\blambda}{\bmu}_{[t^{n-1}t_0^2/u_0^2,t^{n-1}t_0^2u_0^2];t;p,q}
\frac{\Delta_{\bmu}(1/u_0^2|t^{n/2},-t^{(n-1)/2}t_0^2;t^{1/2};p^{1/2},q^{1/2})}
     {\Delta_{\bmu}(1/u_0^4|t^n,t^{n-1}t_0^4;t;p,q)}.
\end{align}
\end{cor}

\section{Quadratic transformations}\label{sec:var}

In \cite{vanish}, there were several other integrals that vanished unless a
given partition (or its conjugate) was of the form $\mu^2$.  If we restate
the integrals in terms of interpolation polynomials rather than Koornwinder
polynomials, the right-hand side becomes a sum over binomial coefficients
$\binom{\lambda}{\mu^2}$, multiplied by the nonzero values, suggesting that
it should be a special case of the right-hand side of
Theorem~\ref{thm:int_litt}.  For most of the results of \cite{vanish}, the
nonzero values were not established, but in the case of Theorem 4.8
op.~cit., they are known, and one can thus use Theorem~\ref{thm:int_litt}
as a guide to formulating the following conjecture.

\begin{conj}\label{conj:vanKsmall}
For generic parameters satisfying $t^{4n-1}t_0^2t_1^2u_0^2=pq$, the integral
\[
\langle
\tcR^{(2n)}_{\blambda}(\dots,\pm \sqrt{-z_i},\dots;t_0\sqrt{-1}{:}-t_0\sqrt{-1},t_1\sqrt{-1},-t_1\sqrt{-1};u_0\sqrt{-1},-tu_0\sqrt{-1};t;p,q)
\rangle^{(n)}_{t_0^2,t_1^2,u_0^2,t,pt,qt;t^2;p^2,q^2}
\]
vanishes unless $\blambda$ is of the form $\bmu^2$, in which case the
integral is
\[
\frac{\Delta_{\bmu}(1/t^2u_0^2|t^{2n},t^{2n-1}t_0^2,1/t^{2n-1}t_0u_0,1/t^{2n}t_0u_0;t^2;p,q)}
     {\Delta_{\bmu^2}(1/tu_0^2|t^{2n},t^{2n-1}t_0^2,1/t^{2n-1}t_0u_0,1/t^{2n}t_0u_0;t;p,q)}.
\]
\end{conj}

If one solves for $u_0$ in the balancing condition then lets $p\to 0$, this
naturally becomes Theorem 4.8 of \cite{vanish} (and agrees with the nonzero
values computed in \cite{bcpoly}); one also notes that the conjecture is
consistent under negating $t_0$ or swapping $t_0$ and $t_1$.

There is an alternate formulation of this conjecture as an identity
of ``hypergeometric'' sums.  The key observation is that the Cauchy-type
interpolation function satisfies the transformation
\[
\cR^{*(2n)}_{\blambda}(\dots,\pm \sqrt{-z_i},\dots;pq/t^{2n}\sqrt{-1}u_0,\sqrt{-1}u_0;t;p,q)
=
\cR^{*(n)}_{\blambda}(\dots,z_i,\dots;p^2q^2/t^{2n}u_0^2,u_0^2;t^2;p^2,q^2),
\]
which follows immediately from the product formula for such functions.
(Recall also that, as remarked after Theorem~\ref{thm:lifted_interp}, we
may freely extend the right-hand side to the case $\ell(\blambda)>n$,
without invalidating our further computations.)

Thus if the conjecture holds, one can compute the integral
\[
\langle
\cR^{*(n)}_{\blambda}(;p^2q^2/t^{2n}u_0^2,u_0^2;t^2;p^2,q^2)
\rangle^{(n)}_{t_0^2,t_1^2,u_0^2,t,pt,qt;t^2,p^2,q^2}
\]
in two different ways: either by expanding
\[
\cR^{*(n)}_{\blambda}(;p^2q^2/t^{2n}u_0^2,u_0^2;t^2;p^2,q^2)
=
\sum_{\bmu}
\obinomE{\blambda}{\bmu}_{[p^2q^2/t^2u_0^4,p^2q^2/t^{2n}t_0^2u_0^2](p^2q^2t_0^2/t^2u_0^2);t^2;p^2,q^2}
\cR^{*(n)}_{\bmu}(;t_0^2,u_0^2;t^2;p^2,q^2)
\]
and applying the elliptic analogue of Kadell's lemma
\cite[Cor.~9.3]{xforms}, or by expanding
\begin{align}
\cR^{*(2n)}_{\blambda}(;pq/t^{2n}\sqrt{-1}u_0&,\sqrt{-1}u_0;t;p,q)\notag\\
&{}=
\sum_{\bmu}
\obinomE{\blambda}{\bmu}_{[-pq/tu_0^2,-pq/t^{2n}t_0u_0](pqt_0/tu_0);t;p,q}
\cR^{*(2n)}_{\bmu}(;t_0\sqrt{-1},u_0\sqrt{-1};t;p,q)
\end{align}
and applying the conjecture in the form
\begin{align}
\langle&
R^{*(2n)}_\bmu(\dots,\pm\sqrt{-z_i},\dots;t_0\sqrt{-1},u_0\sqrt{-1};t;p,q)
\rangle^{(2n)}_{t_0^2,t_1^2,u_0^2,t,pt,qt;t^2,p^2,q^2}
\notag\\
&{}=
\Delta^{0}_\bmu(t^{2n-1}t_0/u_0|t^{2n-1}t_0^2,\pm t^{2n-1}t_0t_1;t;p,q)
\sum_{\bnu}
\obinomE{\bmu}{\bnu^2}_{[t^{2n-1}t_0/u_0,t^{2n}t_0u_0];t;p,q}
\frac{\Delta_{\bnu}(1/t^2u_0^2|t^{2n},t^{2n-1}t_0^2;t^2;p,q)}
     {\Delta_{\bnu^2}(1/tu_0^2|t^{2n},t^{2n-1}t_0^2;t;p,q)}.
\end{align}
One thus finds that the conjecture implies
\begin{align}
\sum_{\bmu}&
\obinomE{\blambda}{\bmu}_{[p^2q^2/t^2u_0^4,p^2q^2/t^{2n}t_0^2u_0^2](p^2q^2t_0^2/t^2u_0^2);t^2;p^2,q^2}
\Delta^0_{\bmu}(t^{2n-2}t_0^2/u_0^2|t^{2n-2}t_0^2t_1^2,t^{2n-1}t_0^2,t^{2n-1}pt_0^2,t^{2n-1}qt_0^2;t^2,p^2,q^2)\notag\\
&{}=
\sum_{\bmu,\bnu}
\Delta^{0}_\bmu(t^{2n-1}t_0/u_0|t^{2n-1}t_0^2,\pm
t^{2n-1}t_0t_1;t;p,q)
\frac{\Delta_{\bnu}(1/t^2u_0^2|t^{2n},t^{2n-1}t_0^2;t^2;p,q)}
     {\Delta_{\bnu^2}(1/tu_0^2|t^{2n},t^{2n-1}t_0^2;t;p,q)}
\notag\\
&\phantom{{}=\sum_{\bmu,\bnu}}
\times\obinomE{\blambda}{\bmu}_{[-pq/tu_0^2,-pq/t^{2n}t_0u_0](pqt_0/tu_0);t;p,q}
\obinomE{\bmu}{\bnu^2}_{[t^{2n-1}t_0/u_0,t^{2n}t_0u_0];t;p,q}.
\label{eq:vanKsmall_sum}
\end{align}
In fact, this is equivalent to the conjecture, since one can equally well
expand the biorthogonal functions in Cauchy-type interpolation functions.
Note also that $n$ enters here only via $t^{2n}$, and thus one can
analytically continue in this extra parameter.

\begin{prop}
Conjecture~\ref{conj:vanKsmall} holds whenever $\ell(\blambda)\le 1$.
\end{prop}

\begin{proof}
  By triangularity, it suffices to prove \eqref{eq:vanKsmall_sum} when
  $\ell(\blambda)\le 1$.  On the right-hand side, this forces
  $\ell(\bmu)\le 1$, so $\ell(\bnu^2)\le 1$, and thus $\bnu=0$, making the
  double sum on the right collapse to a single sum.  The resulting identity
  of univariate elliptic hypergeometric sums is a special case of a known
  quadratic transformation \cite[Thm.~5.1]{SpiridonovVP:2002} (the discrete
  version of Proposition~\ref{prop:univ_tm} below).
\end{proof}

The integral analogue of Conjecture~\ref{conj:vanKsmall} appears to be the
following.  The label ``$(-1,t^{-1/2})$'' and similar labels below will be
explained at the end of the section.  For consistency with the later
conjectures, we replace $2n$ by $n$ but insist that $n$ be even.

\begin{conjQ}[$-1,t^{-1/2}$]\label{conj:vanmt}
  For otherwise generic parameters satisfying $t^n t_0t_1t_2u_0=-pq$, and
  even $n\ge 0$, one has
\begin{align}
\int
&\cR^{*(n)}_{\blambda}(\dots,\pm \sqrt{-z_i},\dots;t_0\sqrt{-1},u_0\sqrt{-1};t;p,q)
\Delta^{(n/2)}(\dots,
z_i,\dots;t_0^2,t_1^2,t_2^2,u_0^2,t,pt,qt,pqt;t^2;p^2,q^2)\notag\\
&{}=
\frac{\Delta^0_{\blambda}(t^{n-1}t_0/u_0|-t^{n-1}t_0t_1;t;p,q)}
     {\Delta^0_{\blambda}(t^{n-1}t_0/u_0|t^n t_0t_1;t;p,q)}
\prod_{0\le i<n}
\prod_{0\le r<s<3}
\frac{\Gampq(-t^i t_rt_s)}
{\Gampq(t^{i+1} t_rt_s)}
\notag\\
&\phantom{{}\propto{}}
\times\int
\cR^{*(n)}_{\blambda}(\dots,t^{\pm 1/2} z_i,\dots;t^{1/2}t_0,t^{1/2}u_0;t;p,q)
\Delta^{(n/2)}(\dots,z_i,\dots;t_0,tt_0,t_1,tt_1,t_2,tt_2,u_0,tu_0;t^2;p,q).
\end{align}
\end{conjQ}

\begin{rem}
Note that the $\Delta^0$ factor is symmetric under swapping $t_1$ and
$t_2$, since the balancing condition makes
$(-t^{n-1}t_0t_1)(t^{n}t_0t_2)=pqt^{n-1}t_0/u_0$.
\end{rem}

If we use the connection coefficient formula \cite[Cor.~4.14]{bctheta} to
expand the interpolation functions on each side in terms of the
corresponding functions with $t_0$ replaced by $t_1$, the result is a
linear combination of instances of the conjecture with $t_0$ and $t_1$
swapped.  The conjecture is similarly consistent under $\blambda\mapsto
(l,m)^{n}+\blambda$, and (combining the two) under shifts in $t_1$, $t_2$,
expanding via the Pieri identity, \cite[Thm.~4.17]{bctheta}, or
equivalently the special case $\blambda=0$, $m=1$, $w_0w_1\in
p^{-\N}q^{-\N}$ of the skew Cauchy identity, Theorem~\ref{thm:cauchy_skew}
above.  In particular, the case $t_2=p^lq^m\sqrt{pq/t}$ reduces to the case
$t_2=\sqrt{pq/t}$, which in turn via Theorem~\ref{thm:int_litt} reduces to
Conjecture~\ref{conj:vanKsmall}.  This, in fact, was how the above
conjecture was formulated, by analytically continuing the result of
shifting parameters and applying the Pieri identity.  The fact that the
resulting transformation is symmetric in $t_1$, $t_2$ is a reassuring
consistency; the fact that the right-hand side appears not to be
symmetrical under $t_1,t_2\mapsto -t_1,-t_2$ is less reassuring, but in
fact a special case of Conjecture~\ref{conj:littbig} would restore this
symmetry.

This conjecture satisfies a further consistency condition.  Consider the
linear functional on the space spanned by the (linearly independent)
functions $\cR^{*(n)}_{\blambda}(;t_0\sqrt{-1},u_0\sqrt{-1};t;p,q)$ defined
by taking
\[
\sum_{\blambda}
c_{\blambda}\cR^{*(n)}_{\blambda}(;t_0\sqrt{-1},u_0\sqrt{-1};t;p,q)
\]
to the sum
\[
\sum_{\blambda}
c_{\blambda}
\frac{\Delta^0_{\blambda}(t^{n-1}t_0/u_0|-t^{n-1}t_0t_1;t;p,q)}
     {\Delta^0_{\blambda}(t^{n-1}t_0/u_0|t^nt_0t_1;t;p,q)}
\cR^{*(n)}_{\blambda}(\dots,t^{\pm 1/2}z_i,\dots;t^{1/2}t_0,t^{1/2}u_0;t;p,q)
\]
and then integrating against the density
\[
\Delta^{(n/2)}(\dots,z_i,\dots;t_0,tt_0,t_1,tt_1,t_2,tt_2,u_0,tu_0;t^2;p,q).
\]
If the conjecture holds, then this linear functional factors through the
homomorphism
\[
f\mapsto f(\dots,\sqrt{-z_i},\dots)
\]
and must therefore vanish on the kernel of that homomorphism.  It suffices
to verify that the image of the functional on the product (and its analogue
with $p$ and $q$ swapped)
\[
\prod_{1\le i\le n}
\frac{\theta_p(v\sqrt{-1}x_i^{\pm 1})}{\theta_p((pq/u_0\sqrt{-1})x_i^{\pm 1})}
\cR^{*(n)}_{\blambda}(;t_0\sqrt{-1},u_0\sqrt{-1}/q)
\]
is invariant under $v\mapsto -v$.  The relevant expansion coefficients come
from the Pieri identity, and we can recognize the resulting integrand as
proportional to
\[
D^{+(n)}_q(t^{1/2}u_0{:}t^{1/2}t_0{:}t^{1/2}t_1,t^{1/2}t_2,-t^{-1/2}v;t;p)
\cR^{*(n)}_{\blambda}(;q^{1/2}t^{1/2}t_0,q^{-1/2}t^{1/2}u_0;t;p,q).
\]
Since $\blambda$ was arbitrary, we conclude that for consistency, we need
\begin{align}
\int
({\cal D}^{+(n)}_q(t^{1/2}u_0{:}t^{1/2}t_0,t^{1/2}t_1,t^{1/2}t_2,-t^{-1/2}v;t;p)f)(\dots,t^{\pm
  1/2}z_i,\dots)\qquad\qquad&
\notag\\
\Delta^{(n/2)}(\dots,z_i,\dots;t_0,tt_0,t_1,tt_1,t_2,tt_2,u_0,tu_0;t^2;p,q)&
\end{align}
to be invariant under $v\mapsto -v$ for any function $f$ in the span of the
interpolation functions.  But this follows by essentially the same argument
as the $1/q$ case of Proposition~\ref{prop:adj_t}.  (In fact, this
adjointness relation is formally equivalent to a special case of the
adjointness relation proved there.)
%
%
%

\begin{prop}
Conjecture~\ref{conj:vanmt} holds when $n=2$.
\end{prop}

\begin{proof}
Note that this is a nontrivial claim even when $\blambda=0$, as the two
integrands involve different values of $p$ and $q$.  However, we observe
that in general the case of the conjecture with $\blambda=(l,m)^n+\bmu$
reduces to the case with $\blambda=\bmu$, so for $n=2$, it suffices to
consider the case $\ell(\blambda)\le 1$.  But then the integral
representation of \cite{xforms} implies the following expression.
\begin{align}
{\cal R}^{*(2)}_{(l,m)}&(\pm \sqrt{-z};v\sqrt{-1},u_0\sqrt{-1};t;p,q)\notag\\
&{}=
\frac{\Gamppqq(t^2 u_0^2 z^{\pm 1})}
     {\Gamppqq(u_0^2 z^{\pm 1},tz^{\pm 1},ptz^{\pm 1},qtz^{\pm 1},pqtz^{\pm
         1})}
\times \frac{(p;p)(q;q)\Gampq(t^2)}{2\Gamppqq(t^2)^2}
\notag\\
&\phantom{{}={}}
\times\int_{C'}
{\cal R}^{*(1)}_{(l,m)}(y;t^{1/2}v\sqrt{-1},t^{-1/2}u_0\sqrt{-1};t;p,q)
\frac{
\Gampq(t^{-1/2}u_0 y^{\pm 1}\sqrt{-1})\Gamppqq(-tz^{\pm 1}y^{\pm 2})}
{\Gampq(t^{3/2}u_0 y^{\pm 1}\sqrt{-1})\Gampq(y^{\pm 2})}
\frac{dy}{2\pi\sqrt{-1}y}.
\end{align}
If we substitute this in and exchange order of integration, the integral
over $z$ becomes an instance of the order 0 elliptic beta integral (the
$z$-dependent factor above cancels five parameters then adds three, making
a final total of six), so can be explicitly evaluated, and we thus conclude
\begin{align}
\int
{\cal R}^{*(2)}_{(l,m)}&(\pm \sqrt{-z};v\sqrt{-1},u_0\sqrt{-1};t;p,q)
\Delta^{(1)}(z;t_0^2,t_1^2,t_2^2,u_0^2,t,pt,qt,pqt;t^2,p^2,q^2)\notag\\
{}={}&
\prod_{0\le r<s<3} \Gamppqq(t_r^2t_s^2)
\prod_{0\le r<3} \Gamppqq(t^2 u_0^2t_r^2)\notag\\
&
\times\int
{\cal
  R}^{*(1)}_{(l,m)}(y;t^{1/2}v\sqrt{-1},t^{-1/2}u_0\sqrt{-1};t^2;p,q)
\notag\\
&\phantom{\times\int{}}\times
\Delta^{(1)}(y;\pm t^{1/2}t_0\sqrt{-1},\pm t^{1/2}t_1\sqrt{-1},\pm
  t^{1/2}t_2\sqrt{-1},t^{-1/2}u_0\sqrt{-1},-t^{3/2}u_0\sqrt{-1};t^2;p,q),
\end{align}
where we have used the fact that a univariate interpolation function is
independent of $t$.  Now, this argument is not actually rigorous, as there
are in general difficulties in choosing the contours in allowing the change
of variables (except if $l=0$ or $m=0$, when there is an open set of
parameters allowing both contours to be the unit circle).  If
$v=-pq/t^2u_0$, then the interpolation functions can both be written as
products, and the result is a special case of
Proposition~\ref{prop:univ_tm} below (which can be viewed as the analytic
continuation in $p^lq^m$).  The corresponding result for $v=t_0$ then
follows from the fact that the connection coefficients are the same on both
sides.

On the right-hand side, we have
\[
{\cal R}^{*(2)}_{(l,m)}(t^{\pm 1/2}z;t^{1/2}t_0,t^{1/2}u_0;t;p,q)
=
{\cal R}^{*(1)}_{(l,m)}(z;t t_0,u_0;t^2;p,q),
\]
and thus the integral on the right-hand side is
\[
\int
{\cal R}^{*(1)}_{(l,m)}(z;t t_0,u_0;t^2;p,q)
\Delta^{(1)}(z;t_0,tt_0,t_1,tt_1,t_2,tt_2,u_0,tu_0;t^2;p,q)
\]
The desired special case of Conjecture~\ref{conj:vanmt} then follows as a
special case of \cite[Cor.~9.11]{xforms}.
\end{proof}

This immediately implies that Conjecture~\ref{conj:vanmt} holds when $t=1$.
Another special case arises when $t=q$ (or, symmetrically, $t=p$), much as
in the discussion following Conjecture~\ref{conj:littbig}.  Since we no
longer have the same interpolation function on both sides of the identity,
we need to control the constants somewhat better.  Note first that for
general $t$, if we replace the interpolation functions by appropriate
versions of
\[
F^{(n)}_{\blambda}(z_1,\dots,z_n;t_1,t_0,u_0;t;p,q)
:=
\frac{\cR^{*(n)}_{\blambda}(z_1,\dots,z_n;t_0,u_0;t;p,q)}
     {\cR^{*(n)}_{\blambda}(t^{n-1}t_1,\dots,t_1;t_0,u_0;t;p,q)
\prod_{1\le i\le n} \Gampq(t^{n-i}t_0t_1,t^{n-i}t_0u_0,t^{n-i}t_1u_0)},
\]
then this absorbs the constants outside the integrals.  With this in mind,
we note that
\[
F^{(n)}_{\lambda,\mu}(z_1,\dots,z_n;ct_1,ct_0,cu_0;q;p,q)
\propto
\frac{
\sum_{\pi,\rho\in S_n}
\sigma(\rho)
\prod_{1\le i\le n}
  F^{(1)}_{\lambda_{\pi_i},\mu_{\rho_i}+n-\rho_i}(z_i;ct_1,ct_0,cq^{n-1}u_0;q;p,q)
}
{
\prod_{1\le i\le n}
  \theta_p(cu_0 z_i^{\pm 1};q)_{n-1}^{-1}
\prod_{1\le i<j\le n} cz_i^{-1} \theta_p(z_i z_j^{\pm 1})},
\]
where the constant of proportionality is independent of $c$.  (The constant
can be explicitly evaluated using Warnaar's determinant identity
\cite[Lem.~5.3]{WarnaarSO:2002}.)  As in Conjecture~\ref{conj:littbig}, we
find that substituting in this expression reduces the $t=q$ case of
Conjecture~\ref{conj:vanmt} to a sum of products of instances with $t=q$,
$n=2$ (essentially a pfaffian).

We also have an additional special case when $\blambda=0$ and in low
dimensions (see also \cite{vandeBultFJ:2011}).

\begin{prop}\label{prop:vanmt_ev}
If $t\in \{p^{1/2},q^{1/2}\}$, then Conjecture~\ref{conj:vanmt} holds if
either $\blambda=0$ or $n\le 6$.
\end{prop}

\begin{proof}
  First suppose $\blambda=0$.  When $t=q^{1/2}$, two parameters cancel in
  the left-hand side, allowing it to be evaluated, while the right-hand
  side can be evaluated by observing that its integrand is equal to an
  elliptic Selberg integrand of order 0 with $q\mapsto q^{1/2}$; the result
  follows upon simplifying the gamma factors.

  Since the conjecture is consistent under parameter shifts and with
  respect to the homomorphism $f\mapsto f(\dots\pm \sqrt{-z_i}\dots)$, we
  find that we can integrate any function of the form
\[
\prod_{1\le i\le n/2}
\frac{
\theta(t_0^2 z_i^{\pm 1};p^2,q^2)_{l_0,m_0}
\theta(t_1^2 z_i^{\pm 1};p^2,q^2)_{l_1,m_1}
\theta(t_2^2 z_i^{\pm 1};p^2,q^2)_{l_2,m_2}
}{\theta((p^2q^2/u_0^2)z_i^{\pm 1};p^2,q^2)_{l_0+l_1+l_2,m_0+m_1+m_2}}
\]
against the left-hand side density in two ways: either directly by shifting
parameters in the left-hand side density and applying the $\blambda=0$
transformation, or indirectly by expanding in images of interpolation
functions and transforming term by term.  Our consistency conditions show
that both approaches will give the same integral (independent of the choice
of expansion).  Since for $n\le 6$ the above functions generically span the
full image of the space of interpolation functions, we conclude that the
transformation actually holds termwise; i.e., Conjecture~\ref{conj:vanmt}
holds for all $\blambda$ in this special case.
\end{proof}

The above evaluation of the right-hand side generalizes to a
transformation, again by observing that both sides have the same integrand.

\begin{prop}\label{prop:pmq_q}
For any odd integer $m>0$,
\[
\II^{(m)}_n(u_0,qu_0,\dots,u_{m+2},qu_{m+2};q^2;p,q^2)
=
\II^{((m-1)/2)}_n(u_0,\dots,u_{m+2},\pm q^{1/2};q;p,q),
\]
subject to the balancing condition $q^{2n-1}\prod_{0\le r<m+3}u_r =
-(pq)^{(m+1)/2}$.
\end{prop}

\begin{rem}
  When $m=1$, this is the aforementioned evaluation.  One can relax the
  condition that $m$ is odd by taking $u_{m+2}=p^{1/2}q^{1/2}$ or
  $-p^{1/2}q^{1/2}$, thus causing a pair of parameters to cancel on the
  left, but not the right; similarly, one can change the sign of the
  balancing condition at the cost of increasing the order on the right.
\end{rem}

A similar argument gives the following, univariate only, transformation.

\begin{prop}\label{prop:pmq_m}
For any even integer $m\ge 0$,
\[
I^{(m)}(\pm \sqrt{u_0},\dots,\pm \sqrt{u_{m+2}};p,q)
=
2\Gampq(-1)
I^{(m/2)}(u_0,\dots,u_{m+2},-1,-q,-p;p^2,q^2),
\]
subject to the balancing condition $\prod_{0\le r<m+3}u_r = -(pq)^{m+1}$.
\end{prop}

\medskip

Since Conjecture~\ref{conj:vanKsmall} involves a choice of $4$-torsion
point on the elliptic curves (namely $\sqrt{-1}$), it has an equivalent
form under modular transformation.  This should then extend back to general
partition pairs, although we have rather less guidance in this case.
Luckily, the argument for $n=2$ carries over with little change other than
replacing Proposition~\ref{prop:univ_tm} with
Proposition~\ref{prop:univ_tq}.  The resulting integral breaks symmetry
between $u_1$, $u_2$, but this can be restored by adding an additional
parameter, as follows.

\begin{conjQ}[$p^{1/2},t^{-1/2}$]\label{conj:vanpt}
  For otherwise generic parameters satisfying $t^n t_0t_1t_2u_0=p^{1/2}q$,
  and even $n\ge 0$, one has
\begin{align}
\int&
\cR^{*(n)}_\blambda(\dots,p^{1/4} z_i^{\pm 1},\dots;p^{1/4}t_0,p^{1/4}u_0;t;p,q)
  \Delta^{(n/2)}(;t_0,t_1,t_2,u_0,\pm\sqrt{t},\pm\sqrt{qt},p^{1/2}v,p^{1/2}q/v;t;p^{1/2},q)
\notag\\
&{}=
\frac{\Delta^0_{\blambda}(t^{n-1}t_0/u_0|t^{n-1}p^{1/2}t_0t_1;t;p,q)}
     {\Delta^0_{\blambda}(t^{n-1}t_0/u_0|t^{n}t_0t_1;t;p,q)}
\prod_{\substack{0\le i<n\\0\le r<s<3}}
\frac{\Gampq(t^i p^{1/2}t_rt_s)}
     {\Gampq(t^{i+1} t_r t_s)}
\\
&\phantom{{}\propto{}}
\times\int
  \cR^{*(n)}_\blambda(\dots,t^{1/2} z_i^{\pm 1},\dots;t^{1/2}t_0,t^{1/2}u_0;t;p,q)
  \Delta^{(n/2)}(;t_0,tt_0,t_1,tt_1,t_2,tt_2,u_0,tu_0,pv,pq/v;t^2;p,q).
\notag\end{align}
\end{conjQ}

\begin{rem}
The additional parameter has the effect of multiplying each integrand by
\[
\prod_{1\le i\le n/2} \theta_q(v z_i^{\pm 1}).
\]
Since these functions span the ($n/2+1$-dimensional) space of
$BC_{n/2}$-symmetric $q$-theta functions of degree 1
\cite[Defn.~1]{bctheta}, one may replace this factor by an arbitrary such
function without affecting the validity of the conjecture.  In particular,
for $n=2$, it suffices to verify the conjecture for two values of $v$, say
$v=t_1$, $v=t_2$, which eliminates the extra parameter, and allows the
previous argument to apply.  The cases $t=1$, $t=p$, $t=q$ follow as
before.
\end{rem}

\begin{rem}
  Again, this (and the remainder of the conjectures we will formulate along
  these lines) is consistent with respect to connection coefficients,
  $\blambda\mapsto (l,m)^n+\blambda$, and shifts in $t_1$, $t_2$,
  regardless of the additional parameter.  Another important consistency
  condition is that if we take $v=t_2$ then multiply $t_2$ by $p^{-1/2}$,
  the left-hand side is again symmetric in $t_1$ and $t_2$, while an
  application of Conjecture~\ref{conj:littbig} exhibits the corresponding
  symmetry on the right-hand side.
\end{rem}

The corresponding vanishing conjecture (obtained by taking $v=t_2$ to
eliminate the extra parameter, then $t_2=q^{1/2}t^{-1/2}$ or
$t_1=p^{1/2}q^{1/2}t^{-1/2}$, then applying Theorem~\ref{thm:int_litt})
is as follows.

\begin{conj}\label{conj:vanPsmall}
For generic parameters satisfying $t^{2n-1/2}t_0t_1u_0=p^{1/2}q^{1/2}$,
the integral
\[
\langle
\tcR^{(2n)}_{\blambda}(\dots,p^{1/4}z_i^{\pm 1},\dots;p^{1/4}t_0{:}p^{-1/4}t_0,p^{1/4}t_1,p^{-1/4}t_1;p^{1/4}u_0,p^{-1/4}tu_0;t;p,q)
\rangle^{(n)}_{t_0,t_1,u_0,t^{1/2},-t^{1/2},-q^{1/2}t^{1/2};t;p^{1/2},q}
\]
vanishes unless $\blambda$ is of the form $\bmu^2$, in which case the
integral is
\[
\frac{\Delta_{\bmu}(1/t^2u_0^2|t^{2n},t^{2n-1}t_0^2,1/t^{2n-1}t_0u_0,1/t^{2n}t_0u_0;t^2;p,q)}
     {\Delta_{\bmu^2}(1/tu_0^2|t^{2n},t^{2n-1}t_0^2,1/t^{2n-1}t_0u_0,1/t^{2n}t_0u_0;t;p,q)}.
\]
\end{conj}

\begin{rem}
  Note that the nonzero values are the same as in
  Conjecture~\ref{conj:vanKsmall}.
\end{rem}

This would imply Conjecture~\ref{conj:vanKsmall} via a modular
transformation, as discussed above, but the $q$-elliptic half of the
identity would be new.  In the limit $q\to 0$, $t_0$, $t_1$ fixed of that
$q$-elliptic identity, the biorthogonal function becomes a Koornwinder
polynomial, and one obtains the vanishing identity given as Theorem 4.10
of \cite{vanish}, together with a conjecture for the nonzero values.  The
case $t_0=q^{1/2}t^{1/2}$, $t_1\mapsto p^{1/4}a$ is also of interest, as in
that case the biorthogonal function becomes a suitably normalized
interpolation function.  One can then take the limit $p\to 0$ with $a$
fixed, in which limit the integral becomes
\[
\int
\frac{P_\lambda(\dots,z_i^{\pm 1},\dots;q,t)}
     {P_\lambda(q^{-1/2}t^{1/2-2n},\dots,q^{-1/2}t^{-1/2};q,t)}
\prod_{1\le i<j\le n}
  \frac{(z_i^{\pm 1}z_j^{\pm 1};q)}
       {(t z_i^{\pm 1}z_j^{\pm 1};q)}
\prod_{1\le i\le n}
  \frac{(z_i^{\pm 2};q)}{(tz_i^{\pm 2};q)}
  \frac{dz_i}{2\pi\sqrt{-1}z_i}.
\]
Apart from the change in normalization of the Macdonald polynomial, this
integral is that of Theorem 4.1 of \cite{vanish}, and therefore vanishes
unless $\lambda=\mu^2$, as predicted by the conjecture.  Moreover, the
corresponding nonzero values (known in this case) agree with those obtained
by degenerating Conjecture~\ref{conj:vanPsmall}.  We also recall from
\cite{bcpoly} that in the limit $n\to\infty$ this becomes Macdonald's
Littlewood identity.

Again, there is a sum version, this time based on the identity
\[
\cR^{*(2n)}_\blambda(\dots p^{1/4} z_i^{\pm 1}\dots;p^{3/4}q/t^{2n}u_0,p^{1/4}u_0;t;p,q)
=
\cR^{*(n)}_{(2,1)\blambda}(\dots z_i\dots;p^{1/2}q/t^{n}u_0,u_0;t;p^{1/2},q)
\]
of Cauchy-type interpolation functions (suitably extended); we omit the
details.  When $\ell(\blambda)\le 1$, the sum is a special case of the
discrete version of Proposition~\ref{prop:univ_tq} below (which discrete
version in turn combines a quadratic transform of Warnaar
\cite{WarnaarSO:2003} with the modular transform of the transform of
Spiridonov \cite[Thm.~5.1]{SpiridonovVP:2002} mentioned above).
%
%

\medskip Of course, the next step is to dualize the above conjectures;
however, we see by reference to the known trigonometric cases that some
subtleties will arise.  The vanishing integral of Theorem 4.8 of
\cite{vanish} (corresponding to Conjecture~\ref{conj:vanKsmall}) actually
dualizes to a pair of vanishing integrals (depending on whether the number
of variables is even or odd), while for Theorem 4.1 of \cite{vanish} (a
limit of Conjecture~\ref{conj:vanPsmall}), not only are there two dual
identities, but each dual identity itself involves a sum of two integrals.

One case is straightforward, namely the ``other'' dual of
Conjecture~\ref{conj:vanpt} (i.e., exchange $p$ and $q$ before dualizing).
Here, and in the other two cases, we begin by dualizing the algebraic
versions of the conjectures, \`a la \eqref{eq:vanKsmall_sum}, having first
analytically continued in $t^{2n}=T$.  After reparametrizing and
specializing $T$ appropriately, we can recognize the left-hand side as the
integral of a Cauchy-type interpolation function, and thus reexpress the
dual as a vanishing identity.  At that point, one may use the Pieri
identity to extend to a large set of cases of an integral transformation.

\begin{conjQ}[$t^{-1/2},q^{1/2}$]\label{conj:vantq}
Subject to the balancing conditions
$t^{n-1}t_0t_1t_2u_0=pq^2t^{1/2}$,
$tv_0v_1 = pq$,
one has
\begin{align}
\int&
\cR^{*(n)}_{\blambda}(;t^{-1/4}t_0,t^{-1/4}u_0;t;p,q)\notag\\
&\times\Delta^{(n)}(;t^{-1/4}t_0,t^{-1/4}t_1,t^{-1/4}t_2,t^{-1/4}u_0,\pm t^{1/4},\pm p^{1/2}t^{1/4},
              t^{1/4}v_0,t^{1/4}v_1;t^{1/2};p,q)\\
&\qquad{}={}
\frac{\Delta^0_\blambda(t^{n-1}t_0/u_0|t^{n-3/2}t_0t_1;t;p,q)}
     {\Delta^0_\blambda(t^{n-1}t_0/u_0|t^{n-1}t_0t_1/q;t;p,q)}
\prod_{0\le i<n}
\prod_{0\le r<s<3}
\frac{\Gampq(t^{i-1/2}t_rt_s)}{\Gampq(t^i t_rt_s/q)}
\notag\\
&\qquad\phantom{{}={}}\times\int
\cR^{*(n)}_{\blambda}(;q^{-1/2}t_0,q^{-1/2}u_0;t;p,q)
\Delta^{(n)}(;q^{\pm 1/2}t_0,q^{\pm 1/2}t_1,q^{\pm 1/2}t_2,q^{\pm 1/2}u_0,
              q^{1/2}v_0,q^{1/2}v_1;t;p,q^2)
\notag\end{align}
\end{conjQ}

\begin{rem}
Here the extra parameter multiplies the integrands by factors
\[
\prod_{0\le i<n}
\frac{\Gampq(t^{1/4}v_0 z_i^{\pm 1})}
     {\Gampq(t^{3/4}v_0 z_i^{\pm 1})}
\quad\text{and}\quad
\prod_{0\le i<n}
\frac{\Gampqq(q^{1/2}v_0 z_i^{\pm 1})}
     {\Gampqq(q^{1/2}tv_0 z_i^{\pm 1})},
\]
which are not, in fact, theta functions.  They do, however, closely
resemble the generating function for the $q,t$-analogues $g_k$ of the
complete symmetric functions \cite{MacdonaldIG:1995}.
\end{rem}

\begin{rem}
  If we eliminate the extra parameters from the normalization (i.e.,
  $\blambda=0$) cases of this conjecture and Conjecture~\ref{conj:vanpt},
  the resulting quadratic transformations are equivalent by
  Proposition~\ref{prop:e7ab}.
\end{rem}

\begin{prop}\label{prop:univ_tq}
Conjecture~\ref{conj:vantq} holds when $n=1$.
\end{prop}

\begin{proof}
Since $\ell(\blambda)\le 1$, we immediately reduce to the case
$\blambda=0$, for which we need simply exchange order of integration in
the double integral (which on an open set of parameters can use unit
circle contours)
\[
\int
\int
\Gampq(q^{-1/2}t^{1/4} x^{\pm 1} y^{\pm 1})
\frac{\prod_{0\le r<4}\Gampq(t^{-1/4}u_r x^{\pm 1})}
     {\Gampq(x^{\pm 2})}
\frac{dx}{2\pi\sqrt{-1}x}
\frac{\Gampqq(q^{1/2}v_0 y^{\pm 1},q^{1/2}v_1 y^{\pm 1})}
     {\Gampqq(y^{\pm 2})}
\frac{dy}{2\pi\sqrt{-1}y}
\]
and simplify the resulting integrands.
\end{proof}

Unfortunately, this case is not sufficient to prove the $t=p$ and $t=q$
cases, as although they can again (and for the later conjectures) be
expressed via (generalized) pfaffians, the entries of the pfaffians include
instances with $n=2$.  The univariate case does, however, suffice to prove
the case $t=1$.

Regarding the $n=2$ case, we note that it suffices to prove the case
$\blambda=0$; indeed, the analogue of the argument in
Proposition~\ref{prop:vanmt_ev} applies, although the functions coming from
parameter shifts only span for $n\le 3$.  (In fact,
Conjecture~\ref{conj:vantq} has recently been proved for $\blambda=0$ by
Van de Bult \cite{vandeBultFJ:2011}, thus implying the general $n\le 3$
case as well as the general pfaffian cases.)

The corresponding vanishing conjecture states that
\[
\langle
\tcR^{(n)}_{\blambda}(;t^{-1/4}t_0{:}t^{1/4}t_0,t^{-1/4}t_1,t^{1/4}t_1;t^{-1/4}u_0,t^{1/4}u_0/q;t;p,q)
\rangle^{(n)}_{t^{-1/4}t_0,t^{-1/4}t_1,t^{-1/4}u_0,t^{1/4},-t^{1/4},-p^{1/2}t^{1/4};t^{1/2};p,q}
\]
vanishes unless $\blambda=(1,2)\bmu$, when it equals
\[
\frac{\Delta_{     \bmu}(q/u_0^2|t^n,t^{n-1}t_0^2,1/t^{n-1}t_0u_0,q/t^{n-1}t_0u_0;t;p,q^2)}
     {\Delta_{(1,2)\bmu}(q/u_0^2|t^n,t^{n-1}t_0^2,1/t^{n-1}t_0u_0,q/t^{n-1}t_0u_0;t;p,q)}.
\label{eq:nonzero_q}
\]
The Koornwinder-type limit $p\to 0$ is again a vanishing integral of
\cite{vanish} (the dual of Theorem 4.10 op.~cit.), together with a
conjecture for the nonzero values.

\smallskip

The next simplest case is the dual of Conjecture~\ref{conj:vanmt}.  Here we
find that the integral on the left-hand side is half the dimension of that
on the right-hand side, which is thus necessarily even.  This constraint
can be avoided, however, by observing that the corresponding integral of
Cauchy-type interpolation functions, when analytically continued in
$T=t^{2n}$, then specialized to $T=t^{2n+1}$ can still be expressed as an
integral.  One thus obtains the following conjecture.

\begin{conjQ}[$-1,q^{1/2}$]\label{conj:vanmq}
For otherwise generic parameters satisfying $t^{n-1}t_0t_1t_2u_0 = -pq^2$,
the integral
\begin{align}
&\frac{\Delta^0_\blambda(t^{n-1}t_0/u_0|-t^{n-1}t_0t_1;t;p,q)}
     {\Delta^0_\blambda(t^{n-1}t_0/u_0|t^{n-1}t_0t_1/q;t;p,q)}
\prod_{0\le i<n}\prod_{0\le r<s<3}
  \frac{\Gampq(-t^i t_rt_s)}
       {\Gampq(t^i t_rt_s/q)}
\notag\\
&
\qquad\times\int
\cR^{*(n)}_{\blambda}(\dots z_i\dots;q^{-1/2}t_0,q^{-1/2}u_0;t;p,q)
\Delta^{(n)}(;q^{\pm 1/2}t_0,q^{\pm 1/2}t_1,q^{\pm 1/2}t_2,q^{\pm 1/2}u_0;t;p,q^2)
\end{align}
is equal to
\[
\int
\cR^{*(n)}_{\blambda}(\dots,\pm\sqrt{-z_i},\dots;t_0\sqrt{-1},u_0\sqrt{-1};t;p,q)
\Delta^{(n/2)}(;t_0^2,t_1^2,t_2^2,u_0^2,1,p,t,pt;t^2;p^2,q^2)
\]
if $n$ is even, and
\begin{align}
&\Gamppqq(t_0^2,t_1^2,t_2^2,u_0^2,p,t,pt)\notag\\
&\times\int
\cR^{*(n)}_{\blambda}(\dots,\pm\sqrt{-z_i},\dots,\sqrt{-1};t_0\sqrt{-1},u_0\sqrt{-1};t;p,q)
\Delta^{((n-1)/2)}(;t_0^2,t_1^2,t_2^2,u_0^2,t^2,p,t,pt;t^2;p^2,q^2)
\end{align}
if $n$ is odd.
\end{conjQ}

Again, this is consistent with respect to parameter shifts and permutations
and annihilates the kernel of the relevant homomorphism; the latter
involves a (formal) special case of the adjointness arguments in the proof
of Proposition~\ref{prop:adj_q}.

\begin{prop}
  If $t=q^2$, then Conjecture~\ref{conj:vanmq} holds if $\blambda=0$ or
  $n\le 7$.
\end{prop}

\begin{proof}
  As in Proposition~\ref{prop:vanmt_ev}, one side is an elliptic Selberg
  integral, while the other can be evaluated using
  Proposition~\ref{prop:pmq_q}.  The extension to general $\blambda$ when
  $n\le 7$ is similar as well.
\end{proof}

In particular, when $n=1$, both sides are essentially independent of $t$,
and thus the conjecture holds in that case as well.  Note, however, that
the usual deduction of the $t=1$ case from the univariate case founders on
the fact that the integral becomes singular when $t=1$.  The cases $t\in
\{p,q\}$ follow from the fact that the conjecture holds for $n=1$ and $n=2$.

\begin{thm}\label{thm:vanmq_ev}
Conjecture~\ref{conj:vanmq} holds for $n\le 3$.
\end{thm}

\begin{proof}
As before, we may reduce to the case $\blambda=0$.  Let $F(t_0,t_1,t_2)$
denote the right-hand side of the conjecture, either
\[
\II^{(1)}_1(t_0^2,t_1^2,t_2^2,u_0^2,1,p,t,pt;t^2;p^2,q^2)
\]
or
\[
\Gamppqq(t_0^2,t_1^2,t_2^2,u_0^2,p,t,pq)
\II^{(1)}_1(t_0^2,t_1^2,t_2^2,u_0^2,t^2,p,t,pt;t^2;p^2,q^2),
\]
depending on whether $n=2$ or $n=3$; we solve for $u_0$ via the balancing
condition.  Similarly, let $G(t_0,t_1,t_2)$ denote the corresponding
left-hand side.

Now, it follows as a special case of the general elliptic hypergeometric
equation \cite{SpiridonovVP:2007} (also Spiridonov's habilitation thesis)
that $F$ satisfies a pair of difference equations
\begin{align}
F(qt_0,t_1,t_2) &=
A(t_0,t_1,t_2)F(t_0,t_1,t_2)+B(t_0,t_1,t_2)F(t_0/q,t_1,t_2)\\
F(pt_0,t_1,t_2) &=
C(t_0,t_1,t_2)F(t_0,t_1,t_2)+D(t_0,t_1,t_2)F(t_0/p,t_1,t_2),
\end{align}
where $A$ and $B$ are $p$-theta functions and $C$ and $D$ are $q$-theta
functions, the specific formulas for which we will not use.  (Moreover,
there exists a rescaling, see \cite{dets}, that makes the coefficients
elliptic functions of $t_0$.)  By Corollary 11 of \cite{dets}, if either of
these equations has generically irreducible Galois group, then $F$ is the
unique solution of the pair of equations, up to a factor independent of
$t_0$.  Since irreducibility is preserved under degeneration
\cite{AndreY:2001}, we may take a limit $q\to 1$ with $t\to 1$,
$t_0\to\sqrt{-p}$, $t_2$, $t_2\to \sqrt{-1}$.  The result depends on the
various rates of approach, and is an Euler integral evaluated at
$s=\lambda(p)$, where $\lambda$ is the cross-ratio of the $2$-torsion of
the elliptic curve with modular parameter $p$, and the only other
constraint is that two of the exponents are equal.  Thus in particular we
can obtain a general complete elliptic integral as a limit, and since the
corresponding equation is a second-order differential equation with
nonelementary solutions, it has irreducible monodromy.  Since both sides
agree when $t_0=p^{1/2}q$, we conclude that it indeed suffices to show that
$G$ satisfies the same equations.

Following Spiridonov, the first stage in deriving the
elliptic hypergeometric equation is to observe that the integrands of the
three integrals
\[
F(t_0,t_1,t_2),F(qt_0,t_1,t_2),F(t_0,qt_1,t_2)
\]
are linearly dependent.  Now, by consistency of the conjecture with respect
to parameter shifts, we can write each integrand as the $F(t_0,t_1,t_2)$
integrand times a function $f(\pm\sqrt{-z},\{\sqrt{-1}\})$ in such a way
that expanding $f$ in interpolation functions and applying the conjecture
term-by-term gives the corresponding shift of $G$.  But consistency with
respect to the homomorphism tells us that any linear combination of such
functions that makes the image vanish makes the transformed integral vanish
as well.  It follows that the three integrals
\[
G(t_0,t_1,t_2),G(qt_0,t_1,t_2),G(t_0,qt_1,t_2)
\]
satisfy the same dependence.

The next step in Spiridonov's derivation is to use an $E_7$ transformation to
obtain relations between the integrals
\[
F(t_0,t_1,t_2),F(t_0/q,t_1,t_2),F(t_0,t_1/q,t_2).
\]
It turns out that for a suitable choice of element (different from
Spiridonov's), we can do so with transformations that preserves the forms
of the integrals.  We find, in particular, that
\[
F(t_0,t_1,t_2)
=
\prod_{0\le i<n}
 \Gampqq(t^i t_0^2,t^i t_1^2,t^i t_2^2,t^i u_0^2)
F(\sqrt{pq^2/t^{n-1}}/t_0,\sqrt{pq^2/t^{n-1}}/t_1,\sqrt{pq^2/t^{n-1}}/t_2),
\]
and similarly for $G$.  (For $F$ we transform using equation (9.49) of
\cite{xforms} (the composition of the reflection in
$(1/2,1/2,1/2,1/2,-1/2,-1/2,-1/2,-1/2)$ with the central element of
$W(E_7)$), while for $G$ we use Corollary 9.13 of \cite{xforms} (the
central element).)  But this reverses the direction of shifting, as
required.

These two recurrences are enough to generate the difference equation: using
the second recurrence with $t_1\mapsto qt_1$, we can express
$F(t_0,t_1,t_2)$ in terms of $F(t_0,qt_1,t_2)$ and $F(t_0/q,qt_1,t_2)$, and
these can in turn be expressed using the first recurrence in terms of
$F(qt_0,t_1,t_2)$, $F(t_0,t_1,t_2)$ and $F(t_0/q,t_1,t_2)$.
\end{proof}

\begin{rem}
  Similar considerations also produce recurrences for $n=4$ and $n=5$,
  but these do not appear to be enough to generate a difference equation.
\end{rem}

The corresponding vanishing conjecture (take $t_2=p^{1/2}q$) reads
that the integrals
\[
\langle
\tcR^{(n)}_{\blambda}(\dots,\pm\sqrt{-z_i},\dots;t_0\sqrt{-1}{:}-t_0\sqrt{-1},\pm
t_1\sqrt{-1};u_0\sqrt{-1},-u_0\sqrt{-1}/q;t;p,q)
\rangle^{(n/2)}_{t_0^2,u_0^2,t_1^2,1,t,pt;t^2;p^2,q^2}
\]
and
\[
\langle
\tcR^{(n)}_{\blambda}(\dots,\pm\sqrt{-z_i},\dots,\sqrt{-1};t_0\sqrt{-1}{:}-t_0\sqrt{-1},\pm
t_1\sqrt{-1};u_0\sqrt{-1},-u_0\sqrt{-1}/q;t;p,q)
\rangle^{((n-1)/2)}_{t_0^2,u_0^2,t_1^2,t^2,t,pt;t^2;p^2,q^2},
\]
subject to the balancing condition $t^{2n-2}t_0^2t_1^2u_0^2 = pq^2$, vanish
unless $\blambda$ has the form $(1,2)\bmu$, when the value is as in
\eqref{eq:nonzero_q} above.  The $n=1$ instance of this is a quadratic
evaluation formula due to Warnaar \cite[(1.4,1.10)]{WarnaarSO:2005}.

\medskip

The remaining case of the three is the direct dual of
Conjecture~\ref{conj:vanPsmall}.  If we attempt to proceed as above, we
find that the most straightforward version of the half-integer case fails
to hold; although taking $t_1=p^{1/2}q$ reduces to the dual vanishing
identity, taking $t_1=-p^{1/2}q$ makes the right-hand side vanish.  If one
instead takes a sum of two terms, symmetric under $p^{1/2}\mapsto
-p^{1/2}$, this problem disappears.  This, of course, corresponds to the
fact that the known Macdonald polynomial limit itself involves a sum of two
integrals.  The corresponding structure for the integer case is then
reasonably straightforward to guess.  One thus formulates the following
conjecture.

\begin{conjQ}[$p^{1/2},q^{1/2}$]\label{conj:vanpq}
For otherwise generic parameters satisfying $t^{n-1}t_0t_1t_2u_0 = p^{1/2}q^2$,
the rescaled integral
\begin{align}
&\frac{\Delta^0_\blambda(t^{n-1}t_0/u_0|t^{n-1}p^{1/2}t_0t_1;t;p,q)}
      {\Delta^0_\blambda(t^{n-1}t_0/u_0|t^{n-1}t_0t_1/q;t;p,q)}
\prod_{\substack{0\le i<n\\0\le r<s<3}}
  \frac{\Gampq(t^i p^{1/2}t_rt_s)}
       {\Gampq(t^i t_rt_s/q)}\notag\\
&\qquad\times\int
\cR^{*(n)}_\blambda(;q^{-1/2}t_0,q^{-1/2}u_0;t;p,q)
\Delta^{(n)}(;q^{\pm 1/2}t_0,q^{\pm 1/2}t_1,q^{\pm 1/2}t_2,q^{\pm
  1/2}u_0,pqv^{\pm 1};t;p,q^2)
\end{align}
admits the following expressions as sums of lower-dimensional integrals:

If $n$ is odd, then it equals
\begin{align}
&\Gamphq(t_0,t_1,t_2,u_0,-1,\pm t^{1/2},p^{1/2}q^{1/2}v^{\pm 1})\notag\\
&\qquad\qquad\qquad\times\int
\cR^{*(n)}_\blambda(\dots,p^{1/4}z_i^{\pm
  1},\dots,p^{1/4};p^{1/4}t_0,p^{1/4}u_0;t;p,q)\notag\\
&\phantom{\qquad\qquad\qquad\times\int}\qquad\qquad\qquad
\times\Delta^{((n-1)/2)}(;t_0,t_1,t_2,u_0,t,-1,\pm t^{1/2},p^{1/2}q^{1/2}v^{\pm 1};t;p^{1/2},q)\notag\\
{}+{}
&\Gamphq(-t_0,-t_1,-t_2,-u_0,-1,\pm t^{1/2},-p^{1/2}q^{1/2}v^{\pm
  1})\notag\\
&\qquad\qquad\qquad\times\int
\cR^{*(n)}_\blambda(\dots,p^{1/4}z_i^{\pm 1},\dots,-p^{1/4};p^{1/4}t_0,p^{1/4}u_0;t;p,q)\notag\\
&\phantom{\qquad\qquad\qquad\times\int}\qquad\qquad\qquad\times
\Delta^{((n-1)/2)}(;t_0,t_1,t_2,u_0,1,-t,\pm t^{1/2},p^{1/2}q^{1/2}v^{\pm 1};t;p^{1/2},q),
\end{align}
while if $n$ is even, it equals
\begin{align}
&\qquad\qquad\qquad\phantom{{}\times{}}\int
\cR^{*(n)}_\blambda(\dots,p^{1/4}z_i^{\pm
  1},\dots;p^{1/4}t_0,p^{1/4}u_0;t;p,q)\notag\\
&\phantom{\qquad\qquad\qquad\times\int}\qquad\qquad\qquad
\times\Delta^{(n/2)}(;t_0,t_1,t_2,u_0,\pm 1,\pm t^{1/2},p^{1/2}q^{1/2}v^{\pm
  1};t;p^{1/2},q)
\notag\\
{}+{}
&\Gampqq(t_0^2,t_1^2,t_2^2,u_0^2,t,t,pqv^{\pm 2})
\Gamphq(-1,-t)\notag\\
&\qquad\qquad\qquad\times\int
\cR^{*(n)}_\blambda(\dots,p^{1/4}z_i^{\pm 1},\dots,\pm p^{1/4};p^{1/4}t_0,p^{1/4}u_0;t;p,q)\notag\\
&\phantom{\qquad\qquad\qquad\times\int}\qquad\qquad\qquad\times
\Delta^{(n/2-1)}(;t_0,t_1,t_2,u_0,\pm t,\pm t^{1/2},p^{1/2}q^{1/2}v^{\pm 1};t;p^{1/2},q).
\end{align}
\end{conjQ}

\begin{rem}
Here the extra parameter multiplies the first integrand by
\[
g(z_1,\dots,z_n)=\prod_{1\le i\le n}\theta_{q^2}(qv z_i^{\pm 1}),
\]
while the integrands on the right are multiplied by
\begin{align}
&g(q^{1/2}z_1^{\pm 1},\dots,q^{1/2}z_{(n-1)/2}^{\pm 1},q^{1/2}),
&&g(q^{1/2}z_1^{\pm 1},\dots,q^{1/2}z_{(n-1)/2}^{\pm 1},-q^{1/2}),\\
&g(q^{1/2}z_1^{\pm 1},\dots,q^{1/2}z_{n/2}^{\pm 1}),
&&g(q^{1/2}z_1^{\pm 1},\dots,q^{1/2}z_{n/2-1}^{\pm 1},\pm q^{1/2}),
\end{align}
and this factor accounts for all dependence on $v$.  One could thus just as
well replace $g$ by an arbitrary $BC_n$-symmetric $q^2$-theta function of
degree 1.  From this perspective, each of the above four specializations
induces a linear transformation from the space of such theta functions to a
corresponding space of $BC_n$-symmetric $q$-theta functions.  Moreover, each
such transformation is surjective, with kernel one of the two eigenspaces of the
operator
\[
g(z_1,\dots,z_n)\to q^{n/2}(\prod_{1\le i\le n}z_i)g(qz_1,\dots,qz_n).
\]

The interpolation functions are similarly specialized, and thus if $g$ is
an eigenfunction, then we should expect the integral to vanish when the
interpolation function on the right is replaced by anything in the
corresponding kernel.  The $q$-elliptic part of this kernel is not an
eigenspace of an involution, so the earlier argument fails on that half; we
thus only consider the case that the $p$-elliptic portion of the integrand
is in the kernel.  We can then argue as we did after
Conjecture~\ref{conj:vanmq}, to find that the corresponding left-hand side
can again be expressed in terms of a difference operator, and thus reduce
to showing that the integral
\begin{align}
\int
{\cal D}^{+(n)}_q(q^{-1/2}u_0{:}q^{-1/2}t_0,q^{-1/2}t_1,q^{-1/2}t_2,q^{1/2}p^{3/4}w;t;p)
f\qquad\qquad&\notag\\
\times\Delta^{(n)}(;q^{\pm 1/2}t_0,q^{\pm 1/2}t_1,q^{\pm 1/2}t_2,q^{\pm 1/2}u_0,
              pqv^{\pm 1};t;p,q^2)&
\end{align}
is quasiperiodic (multiplied by $(q^{1/2}p^{1/4}vw)^{-n}$) under
$(v,w)\mapsto (qv,p^{1/2}w)$.  But this again follows by an adjointness
argument.
\end{rem}

\begin{prop}
Conjecture~\ref{conj:vanpq} holds when $n=1$.
\end{prop}

\begin{proof}
Since $\ell(\blambda)\le 1$ when $n=1$, we may as well take $\blambda=0$.
We thus need simply to prove that when $u_0u_1u_2u_3 = p^{1/2}q^2$,
\begin{align}
\frac{(p;p)(q^2;q^2)}{2}&
\int
\frac{\prod_{0\le r<4} \Gampqq(q^{\pm 1/2}u_r z^{\pm 1})}
     {\Gampqq(z^{\pm 2})}
\frac{\theta_{q^2}(q v z^{\pm 1}) dz}{2\pi\sqrt{-1}z}
\notag\\
&{}=
\frac{
\Gamphq(-1,u_0,u_1,u_2,u_3)
\theta_{q}(q^{1/2} v)
}{
\prod_{0\le r<s<4} \Gampq(p^{1/2}u_ru_s)
}
+
\frac{\Gamphq(-1,-u_0,-u_1,-u_2,-u_3)
\theta_{q}(-q^{1/2} v)
}{
\prod_{0\le r<s<4} \Gampq(p^{1/2}u_ru_s)
}
\end{align}
Each of the three terms is a $BC_1$-symmetric $q^2$-theta function in $v$
of degree 1, and thus the relation will follow if
we check it at any two independent points.  By symmetry under $v\mapsto
-v$, we may reduce to the case $v=q^{-1/2}$, when the left-hand side can be
expressed (via Proposition~\ref{prop:pmq_q}) as
\[
\frac{(q^2;q^2)}{(q;q)}
I^{(0)}(q^{-1/2}u_0,q^{-1/2}u_1,q^{-1/2}u_2,q^{-1/2}u_3,-q^{1/2},-p^{1/2} q^{1/2};p,q).
\]
The proposition follows upon simplifying the resulting product of elliptic
gamma functions.
\end{proof}

\begin{rem}
This can also be obtained as the limit $t\to 1/p$ of
Proposition~\ref{prop:univ_tq}; one finds that the left-hand side of that
Proposition violates the contour conditions in two different ways in the
limit, and thus becomes a sum of two residues, corresponding to the two
terms above.
\end{rem}

The proof of Theorem~\ref{thm:vanmq_ev} carries over, with some additional
subtleties.  

\begin{thm}\label{thm:vanpq_ev}
Conjecture~\ref{conj:vanpq} holds whenever $n\le 3$.
\end{thm}

\begin{proof}
  As before, we may reduce to the case $\blambda=0$.  In addition, the
  above considerations involving the function $g$ have the effect that the
  two terms on the right-hand side are $q$-theta functions of $v$, but with
  different multipliers.  We may thus use this to write either term on the
  right as a linear combination of two instances of the term on the left.
  Finally, it suffices to consider the case $v=q^{-1/2}u_0$, since the
  $\blambda=0$ case is symmetrical between $u_0$ and the $t_r$ parameters.
  We can then argue as in Theorem~\ref{thm:vanmq_ev} to see that both sides
  satisfy the same elliptic $q$-difference equations.  Since we only have
  consistency with respect to the $p$-elliptic kernel, this does not give
  us the requisite $p$-difference equation to finish the proof.  However,
  if we shift $u_0\mapsto p^{-1/2}u_0$ and square $p$, the integrals on the
  right become symmetrical under permutations of $t_0$, $t_1$, $t_2$, $u_0$
  as well as under swapping $p$ and $q$.  For $n=2$, the desired identity
  becomes
\begin{align}
\prod_{0\le r<s<3}
&
  \frac{\Gamppq(pt_rt_s,ptt_rt_s)}
       {\Gamppq(t_rt_s/q,tt_rt_s/q)}
\II^{(1)}_2(q^{\pm 1/2}t_0,q^{\pm 1/2}t_1,q^{\pm 1/2}t_2,(p^2q)^{\pm
  1/2}u_0;t;p^2,q^2)\notag\\
&{}=
\II^{(1)}_1(t_0,t_1,t_2,u_0,\pm 1,\pm t^{1/2};t;p,q)
+
\Gamppqq(t_0^2,t_1^2,t_2^2,u_0^2,t,t)
\Gampq(-1,-t)
,
\end{align}
(with balancing condition $tt_0t_1t_2u_0=p^2q^2$), while for $n=3$, it
becomes
\begin{align}
\prod_{\substack{0\le i<3\\0\le r<s<3}}
  \frac{\Gamppq(t^i pt_rt_s)}
       {\Gamppq(t^i t_rt_s/q)}&
\II^{(1)}_3(q^{\pm 1/2}t_0,q^{\pm 1/2}t_1,q^{\pm 1/2}t_2,(p^2q)^{\pm 1/2}u_0;t;p^2,q^2)\notag\\
&{}=
\hphantom{{}+{}}
\Gampq(t_0,t_1,t_2,u_0,-1,\pm t^{1/2})
\II^{(1)}_1(t_0,t_1,t_2,u_0,t,-1,\pm t^{1/2};t;p,q)\notag\\
&\hphantom{{}={}}{}+
\Gampq(-t_0,-t_1,-t_2,-u_0,-1,\pm t^{1/2})
\II^{(1)}_1(t_0,t_1,t_2,u_0,1,-t,\pm t^{1/2};t;p,q),
\end{align}
with balancing condition $t^2t_0t_1t_2u_0=p^2q^2$.  The corresponding
symmetries of the left-hand side follow from Proposition~\ref{prop:e7ab},
as do the relevant symmetries when $v=q^{-3/2}u_0$.  Thus, the fact that
both sides satisfy the same $p$-difference equations implies that both sides
satisfy the same $q$-difference equations, and the identity follows.
\end{proof}

\begin{rem}
The extension to $t=p$, $t=q$ can also be made to work; if one
represents the left-hand side as a ``pfaffian'' of $n=2$ and $n=1$
instances, then the transform is a sum of $2^{\lceil n/2\rceil}$
``pfaffians''.  Most of these vanish, however, and the survivors are all
proportional to one of the two terms of the right-hand side.  (When
$\lambda=0$ so the ``pfaffian'' is actually a pfaffian, the corresponding
alternating matrix is a sum of two alternating matrices, one of rank 2.)
\end{rem}

%
%

When $t_1=-p^{1/2}q$, $v=q^{1/2}/t_2$, one of the two terms vanishes, since
$\Gamma(p^{1/2}q;p^{1/2},q)=0$, and we thus obtain (after an application of
Conjecture~\ref{conj:littbig_dual}) the following vanishing conjecture:
that when $t^{n-1}t_0t_1u_0=-p^{1/2}q$, the integral
\[
\langle
\tcR^{(n)}_{\blambda}(\dots,p^{1/4}z_i^{\pm 1},\dots;p^{1/4}t_0{:}p^{-1/4}t_0,p^{1/4}t_1,p^{-1/4}t_1;p^{1/4}u_0,p^{-1/4}u_0/q;t;p,q)
\rangle^{(n/2)}_{t_0,t_1,u_0,1,t^{1/2},-t^{1/2};t;p^{1/2},q}
\]
or
\[
\langle
\tcR^{(n)}_{\blambda}(\dots,p^{1/4}z_i^{\pm 1},\dots,p^{1/4};p^{1/4}t_0{:}p^{-1/4}t_0,p^{1/4}t_1,p^{-1/4}t_1;p^{1/4}u_0,p^{-1/4}u_0/q;t;p,q)
\rangle^{((n-1)/2)}_{t_0,t_1,u_0,t,t^{1/2},-t^{1/2};t;p^{1/2},q},
\]
as appropriate, vanishes unless $\blambda=(1,2)\bmu$, when its value is
given by \eqref{eq:nonzero_q}.

Similarly, in the limit $t_1,t_2\to q^{1/2}v^{\pm 1}$, the $n$-dimensional
integral degenerates to a (dual) Littlewood-style sum, and an application
of connection coefficients gives the conjecture that when
$t^{n-1}t_0u_0=p^{1/2}q$, the integrals
\begin{align}
&\frac{1}{2}
\langle
\cR^{*(n)}_{\blambda}(\dots,p^{1/4}z_i^{\pm 1},\dots;p^{-1/4}t_0,p^{1/4}u_0;t;p,q)
\rangle^{(n/2)}_{t_0,u_0,1,-1,t^{1/2},-t^{1/2};t;p^{1/2},q}
\notag\\
{}+{}&
\frac{1}{2}
\langle
\cR^{*(n)}_{\blambda}(\dots,p^{1/4}z_i^{\pm 1},\dots,\pm p^{1/4};p^{-1/4}t_0,p^{1/4}u_0;t;p,q)
\rangle^{(n/2-1)}_{t_0,u_0,t,-t,t^{1/2},-t^{1/2};t;p^{1/2},q}
\end{align}
and
\begin{align}
&\frac{1}{2}
\langle
\cR^{*(n)}_\blambda(\dots,p^{1/4}z_i^{\pm 1},\dots,p^{1/4};p^{-1/4}t_0,p^{1/4}u_0;t;p,q)
\rangle^{((n-1)/2)}_{t_0,u_0,t,-1,t^{1/2},-t^{1/2};t;p^{1/2},q}
\notag\\
{}+{}&
\frac{1}{2}
\langle
\cR^{*(n)}_\blambda(\dots,p^{1/4}z_i^{\pm 1},\dots,-p^{1/4};p^{-1/4}t_0,p^{1/4}u_0;t;p,q)
\rangle^{((n-1)/2)}_{t_0,u_0,1,-t,t^{1/2},-t^{1/2};t;p^{1/2},q}
\end{align}
vanish unless $\blambda=(1,2)\bmu$, when they have value
\[
\frac{\Delta_{     \bmu}(q/u_0^2|t^{n},t^{n-1};t;p,q^2)}
     {\Delta_{(1,2)\bmu}(q/u_0^2|t^{n},t^{n-1};t;p,q)}.
\]
Taking $t_0=p^{1/4}a$ and $p\to 0$ turns the interpolation functions into
Macdonald polynomials, and one again obtains a result of \cite{vanish} (the
dual of Theorem 4.1 op.~cit.).  When $q=t$, the Macdonald polynomials
become Schur functions, and one obtains the well-known
representation-theoretic fact that the Haar integral
\[
\int_{O\in O(n)} s_\lambda(O)
\]
vanishes unless $\lambda=2\mu$, when it equals 1.

\begin{rem}
  One might think to obtain the interpolation function case by a limit of
  the biorthogonal function case (as this works in the other cases).
  However, the interpolation function limit only works when the parameters
  are otherwise generic; in this instance it fails when $\ell(\blambda)=n$, as
  the biorthogonal function becomes singular (the first two parameters
  multiply to 1).
\end{rem}

\medskip If we swap $p$ and $q$ above, we find that
Conjecture~\ref{conj:vanmq} becomes self-dual, while
Conjectures~\ref{conj:vantq} and \ref{conj:vanpq} become dual to each
other.  However, we now have the possibility again of modular
transformations.  Given the lack of guidance from the trigonometric level,
the resulting conjectures are rather more speculative than those above.
The overall form of the integrals is fairly straightforward to determine,
especially since in each case the normalization without extra parameter
reduces via Proposition~\ref{prop:e7ab} to a previously conjectured
normalization.  The $\blambda$-dependent factors are then uniquely
determined by the requirement of consistency under the Pieri identity (more
precisely, that the obvious argument for consistency should work, as it did
in all previous cases).

For Conjecture~\ref{conj:vanmq}, one obtains the following transform, of
which the case $n=1$ is straightforward.

\begin{conjQ}[$p^{1/2},-1$]\label{conj:vanpm} 
For otherwise generic parameters satisfying
$t^{n-1}t_0t_1t_2u_0=p^{1/2}q$,
the integral
\begin{align}
&\frac{\Delta^0_{\blambda}(t^{n-1}t_0/u_0|p^{1/2} t^{n-1}t_0t_1;t;p,q)}
     {\Delta^0_{\blambda}(t^{n-1}t_0/u_0|t^{n-1}t_0t_1;t;p,q)}
\prod_{0\le i<n}\prod_{0\le r<s<3}
  \frac{\Gampq(t^i p^{1/2}t_rt_s)}{\Gampq(t^it_rt_s)}\notag\\
&\quad\times\int
\cR^{*(n)}_\blambda(\dots,z_i^2,\dots;-t_0,-u_0;t;p,q)
\Delta^{(n)}(;\pm \sqrt{-t_0},\pm \sqrt{-t_1},\pm \sqrt{-t_2},\pm \sqrt{-u_0},
               p^{1/2}q^{1/4}v^{\pm 1};t^{1/2};p^{1/2},q^{1/2})
\end{align}
is equal to
\[
\int
\cR^{*(n)}_\blambda(\dots,p^{1/4}z_i^{\pm 1},\dots;p^{1/4}t_0,p^{1/4}u_0;t;p,q)
\Delta^{(n/2)}(;t_0,t_1,t_2,u_0,1,q^{1/2},t^{1/2},q^{1/2}t^{1/2},
                -p^{1/2}q^{1/2}v^{\pm 2};t;p^{1/2},q)
\]
if $n$ is even, and
\begin{align}
&\Gamphq(t_0,t_1,t_2,u_0,q^{1/2},t^{1/2},q^{1/2}t^{1/2},-p^{1/2}q^{1/2}v^{\pm
  2})\notag\\
&\qquad\qquad
\times\int
\cR^{*(n)}_\blambda(\dots,p^{1/4}z_i^{\pm
  1},\dots,p^{1/4};p^{1/4}t_0,p^{1/4}u_0;t;p,q)\notag\\
&\phantom{\qquad\qquad\times\int}\qquad\qquad\times
\Delta^{((n-1)/2)}(;t_0,t_1,t_2,u_0,t,q^{1/2},t^{1/2},q^{1/2}t^{1/2},
                    -p^{1/2}q^{1/2}v^{\pm 2};t;p^{1/2},q)
\end{align}
if $n$ is odd.
\end{conjQ}

\begin{rem}
Once again, the factor
\[
\prod_{1\le i\le n}
\Gamphqh(p^{1/2}q^{1/4}v^{\pm 1}z_i^{\pm 1})
=
\prod_{1\le i\le n} \theta_{q^{1/2}}(q^{1/4}v z_i^{\pm 1})
\]
on the left can be replaced by an arbitrary $BC_n$ $q^{1/2}$-theta function
$g$ of degree 1, in which case the extra factor on the right is
\[
g(\pm \sqrt{-z_1},\dots,\pm\sqrt{-z_{n/2}}),\quad\text{or}\quad
g(\pm \sqrt{-z_1},\dots,\pm\sqrt{-z_{(n-1)/2}},\sqrt{-1}),
\]
as appropriate.  The factor on the right vanishes iff
\[
g(-z_1,\dots,-z_n) = -g(z_1,\dots,z_n).
\]
\end{rem}

\begin{rem}
In fact, algebraically speaking, there is another modular transformation,
since this conjecture depends on an ordered pair of $2$-torsion points
(namely $-1$ and $p^{1/2}$), so there is a ``vanishing'' conjecture associated
to the pair $(\pm p^{1/2})$.  However, this conjecture does not appear
amenable to extension to a full integral.
\end{rem}

\begin{thm}
Conjecture~\ref{conj:vanpm} holds for $n\le 3$.
\end{thm}

\begin{proof}
As usual, we may reduce to the case $\blambda=0$, and it will suffice to
consider the case $v=q^{-1/4}\sqrt{-u_0}$.  If we rescale $u_0\mapsto
p^{-1/2}u_0$, we find that we need to prove the identities
\begin{align}
\prod_{\substack{0\le i<2\\0\le r<s<3}}
&
  \frac{\Gampq(t^i p^{1/2}t_rt_s)}{\Gampq(t^it_rt_s)}
\II_2(\pm \sqrt{-t_0},\pm \sqrt{-t_1},\pm \sqrt{-t_2},p^{1/4}\sqrt{-u_0}
             ,-p^{-1/4}\sqrt{-u_0};t^{1/2};p^{1/2},q^{1/2})\notag\\
&=
\II_1(t_0,t_1,t_2,u_0,1,q^{1/2},t^{1/2},q^{1/2}t^{1/2};t;p^{1/2},q)
\end{align}
with $tt_0t_1t_2u_0=pq$, and
\begin{align}
\prod_{\substack{0\le i<3\\0\le r<s<3}}
&  \frac{\Gampq(t^i p^{1/2}t_rt_s)}{\Gampq(t^it_rt_s)}
\II_3(\pm \sqrt{-t_0},\pm \sqrt{-t_1},\pm \sqrt{-t_2},p^{1/4}\sqrt{-u_0}
             ,-p^{-1/4}\sqrt{-u_0};t^{1/2};p^{1/2},q^{1/2})\notag\\
&=
\Gamphq(t_0,t_1,t_2,u_0,q^{1/2},t^{1/2},q^{1/2}t^{1/2})
\II_1(t_0,t_1,t_2,u_0,t,q^{1/2},t^{1/2},q^{1/2}t^{1/2};t;p^{1/2},q)
\end{align}
with $t^2t_0t_1t_2u_0=pq$.  But these identities follow from
Theorem~\ref{thm:vanmq_ev} by Proposition~\ref{prop:e7ab}.
\end{proof}

\begin{rem}
  Again, this implies the pfaffian cases $t^{1/2}\in \{\pm p^{1/2},\pm
  q^{1/2}\}$.
\end{rem}

The corresponding ``vanishing'' result states that the integral
\[
\langle
\tcR^{(n)}_{\lambda}(\dots,p^{1/4}z_i^{\pm 1},\dots;p^{1/4}t_0{:}p^{-1/4}t_0,p^{1/4}t_1,p^{-1/4}t_1;p^{1/4}u_0,p^{-1/4}u_0;t;p,q)
\rangle^{(n/2)}_{t_0,t_1,u_0,1,t^{1/2},q^{1/2}t^{1/2};t;p^{1/2},q}
\]
or
\[
\langle
\tcR^{(n)}_{\lambda}(\dots,p^{1/4}z_i^{\pm 1},\dots,p^{1/4};p^{1/4}t_0{:}p^{-1/4}t_0,p^{1/4}t_1,p^{-1/4}t_1;p^{1/4}u_0,p^{-1/4}u_0;t;p,q)
\rangle^{((n-1)/2)}_{t_0,t_1,u_0,t,t^{1/2},q^{1/2}t^{1/2};t;p^{1/2},q},
\]
as appropriate, evaluates to
\[
\frac{\Delta_{\blambda}(-1/u_0|t^{n/2},t^{(n-1)/2}t_0,\pm (t^{n-1}t_0u_0)^{-1/2};t^{1/2};p^{1/2},q^{1/2})}
     {\Delta_{\blambda}(1/u_0^2|t^n,t^{n-1}t_0^2,1/t^{n-1}t_0u_0,1/t^{n-1}t_0u_0;t;p,q)}.
\label{eq:nonzero_m}
\]
Again, for $n=1$, this is a known quadratic evaluation
\cite[(1.4)]{WarnaarSO:2005}.  If one sets $t_0=\sqrt{q}$, the biorthogonal
function becomes an interpolation polynomial; taking $t_1,u_0\sim p^{1/4}$
and $p\to 0$ gives a conjecture for Macdonald polynomials which for $n$
even reads
\begin{align}
\frac{1}{Z}
\int
P_\lambda(\dots,z_i^{\pm 1},\dots;q,t)
\prod_{1\le i<j\le n/2}
  \frac{(z_i^{\pm 1}z_j^{\pm 1};q)}{(t z_i^{\pm 1}z_j^{\pm 1};q)}
\prod_{1\le i\le n/2}
  \frac{(-z_i^{\pm 1};q^{1/2})}{(t^{1/2}z_i^{\pm 1};q^{1/2})}
  \frac{dz_i}{2\pi\sqrt{-1}z_i}\qquad\qquad &\notag\\
{}=
\frac{C^0_{\lambda}(t^{n/2};q^{1/2},t^{1/2})C^-_\lambda(-q^{1/2};q^{1/2},t^{1/2})}
{C^0_\lambda(-q^{1/2}t^{(n-1)/2};q^{1/2},t^{1/2})C^-_\lambda(t^{1/2};q^{1/2},t^{1/2})}.&
\end{align}
(For a proof in the special case $q=0$, see
\cite[Cor.~6.4]{VenkateswaranV:2010}.)  If we replace the integrand by the
right-hand side of the Cauchy identity, then take the limit $n\to\infty$,
this again becomes Kawanaka's conjecture.  The case $t^{1/2}=-q^{1/2}$ of
the Macdonald polynomial conjecture is also of interest, as it gives the
well-known identity
\[
\int_{O\in O(n)} \det(1+O)s_\lambda(O) = 1.
\]
\medskip

The conjecture obtained from Conjecture~\ref{conj:vanpq} in the
corresponding way is the same, except with $p$ and $q$ swapped.  (The fact
that this changes a sum of two integrals to a single integral should not be
a concern, since after all Conjectures~\ref{conj:vanmq} and
\ref{conj:vanpq} are each other's modular transforms.) We thus have only
one more transform to consider, namely that obtained from
Conjecture~\ref{conj:vantq}.

\begin{conjQ}[$t^{-1/2},-1$]\label{conj:vantm}
Subject to the balancing conditions
$t^{n-1} t_0t_1t_2u_0=pqt^{1/2}$, $v_0v_1=p^{1/2}q^{1/2}/t^{1/2}$, one has
\begin{align}
&\int
\cR^{*(n)}_{\blambda}(;t^{-1/4}t_0,t^{-1/4}u_0;t;p,q)\notag\\
&\phantom{\int}
\times\Delta^{(n)}(;t^{-1/4}t_0,t^{-1/4}t_1,t^{-1/4}t_2,t^{-1/4}u_0,
              t^{1/4},p^{1/2}t^{1/4},q^{1/2}t^{1/4},p^{1/2}q^{1/2}t^{1/4},
              -t^{1/4}v_0^2,-t^{1/4}v_1^2;t^{1/2};p,q)\notag\\
{}={}&
\frac{\Delta^0_{\blambda}(t^{n-1}t_0/u_0|t^{n-1}t^{-1/2}t_0t_1;t;p,q)}
     {\Delta^0_{\blambda}(t^{n-1}t_0/u_0|t^{n-1}t_0t_1;t;p,q)}
\prod_{0\le i<n}\prod_{0\le r<s<3}
\frac{\Gampq(t^{i-1/2}t_rt_s)}{\Gampq(t^i t_rt_s)}
\notag\\
&\times\int
\cR^{*(n)}_\blambda(\dots,z_i^2,\dots;-t_0,-u_0;t;p,q)
\Delta^{(n)}(;\pm \sqrt{-t_0},\pm \sqrt{-t_1},\pm \sqrt{-t_2},\pm \sqrt{-u_0},
              v_0,v_1
;t^{1/2};p^{1/2},q^{1/2}).
\end{align}
\end{conjQ}

\begin{prop}\label{prop:univ_tm}
Conjecture~\ref{conj:vantm} holds for $n=1$.
\end{prop}

\begin{proof}
Exchange order of integration in the double integral
\[
\int
\int
\Gampq(t^{1/4} y^{\pm 2} x^{\pm 1})
\frac{
\prod_{0\le r<4} \Gampq(u_r x^{\pm 1})}
{\Gampq(x^{\pm 2})}
\frac{dx}{2\pi\sqrt{-1}x}
\frac{
\Gamphqh(v_0 y^{\pm 1},v_1 y^{\pm 1})}
{
\Gamphqh(y^{\pm 2})
}
\frac{dy}{2\pi\sqrt{-1}y},
\]
then use gamma function identities to express the integrands in the
standard form.
\end{proof}

\begin{rem}
  Again, Van de Bult \cite{vandeBultFJ:2011} has recently proved the case
  $\blambda=0$ of Conjecture~\ref{conj:vantm}, which by the usual
  considerations implies the general $n\le 3$ case and the pfaffian cases
  $t^{1/2}\in \{\pm p^{1/2},\pm q^{1/2}\}$.
\end{rem}

The corresponding ``vanishing'' integral states that
\[
\langle
\tilde{R}^{(n)}_{\blambda}(;t^{-1/4}t_0{:}t^{1/4}t_0,t^{-1/4}t_1,t^{1/4}t_1;t^{-1/4}u_0,t^{1/4}u_0;t;p,q)
\rangle^{(n)}_{t^{-1/4}t_0,t^{-1/4}t_1,t^{-1/4}u_0,
              t^{1/4},p^{1/2}t^{1/4},q^{1/2}t^{1/4};t^{1/2};p,q}
\]
takes value \eqref{eq:nonzero_m}.

\bigskip We close with a combinatorial remark.  The ``Q'' conjectures, if
we count both forms of those integrals not symmetric between $p$ and $q$,
give rise to twelve conjectures, one for each ordered pair $(a,b)$ with
$a\ne b\in \{-1,p^{1/2},q^{1/2},t^{-1/2}\}$.  Furthermore, the three
involutions ``modular transform'', ``swap $p$ and $q$'', and ``dualize''
act on the labels via their natural action on this set of square roots.  In
addition, in the conjecture associated to the pair $(a,b)$, the integrals
are related by a factor
\[
\frac{\Delta^0_{\blambda}(t^{n-1}t_0/u_0|t^{n-1}t_0t_1 a;t;p,q)}
     {\Delta^0_{\blambda}(t^{n-1}t_0/u_0|t^{n-1}t_0t_1/b^2;t;p,q)}
\prod_{\substack{0\le i<n\\0\le r<s<3}}
  \frac{\Gampq(t^i t_r t_s s)}
       {\Gampq(t^i t_r t_s/b^2)},
\]
with balancing condition $t^{n-1} t_0t_1t_2u_0=pqb^2/a$.  This pattern,
together with corresponding patterns in the parameters of the interpolation
functions, allows us to verify consistency with respect to the Pieri
identity and the connection coefficient identity for all of the cases at
once, apart from checking that multiplying the interpolation functions by
\[
\prod_{1\le i\le n}
\frac{\Gampq(a^{1/2}Qt_0 z_i^{\pm 1},a^{1/2}u_0 z_i^{\pm 1}/Q)}
     {\Gampq(a^{1/2}t_0 z_i^{\pm 1},a^{1/2}u_0 z_i^{\pm 1})},
\quad\text{or}\quad
\prod_{1\le i\le n}
\frac{\Gampq(Qt_0 z_i^{\pm 1}/b,u_0 z_i^{\pm 1}/bQ)}
     {\Gampq(t_0 z_i^{\pm 1}/b ,u_0 z_i^{\pm 1}/b)},
\]
as appropriate, before specializing, has the effect, after specializing, of
shifting parameters in the corresponding integrand.  Similarly, the L
conjectures correspond to four identities in natural bijection with the
above four square roots.  Unfortunately, there are enough quirks in the
various cases to make it unclear how to formulate the conjectures in a more
uniform manner.

\end{document}